\documentclass{amsart}
\usepackage[utf8x]{inputenc}
\usepackage[T1]{fontenc}
\usepackage{ae,aecompl}
\usepackage{graphics}

\usepackage{amssymb, amsmath, amscd}
\usepackage{rotating} %

\usepackage{tikz}
\usetikzlibrary{arrows,matrix,shapes}
\definecolor{green}{RGB}{0,150,0}

\usepackage{enumerate}
\usepackage{verbatim}
\usepackage{hyperref}
\usepackage{multicol}
\usepackage{xspace}

\usepackage{ifthen}
\newboolean{draft}
\setboolean{draft}{false}
\newboolean{longversion}
\setboolean{longversion}{true}

\newtheorem{theorem}{Theorem}[section]
\newtheorem{lemma}[theorem]{Lemma}
\newtheorem{proposition}[theorem]{Proposition}
\newtheorem{corollary}[theorem]{Corollary}
\newtheorem{definition}[theorem]{Definition}
\newtheorem{problem}[theorem]{Problem}

\newtheorem{remark}[theorem]{Remark}
\newtheorem{remarks}[theorem]{Remarks}

\theoremstyle{definition}
\newtheorem{example}[theorem]{Example}

\numberwithin{equation}{section}

\newcommand{\regressive}{regressive\xspace}
\newcommand{\extensive}{extensive\xspace}
\newcommand{\cut}{\sqsubset}
\newcommand{\cuteq}{\sqsubseteq}
\newcommand{\fibers}{\operatorname{fibers}}
\newcommand{\finer}{\preceq}
\newcommand{\fix}{\operatorname{fix}}
\newcommand{\End}{\operatorname{End}}
\newcommand{\id}{\operatorname{id}}
\newcommand{\one}{1}
\newcommand{\im}{\operatorname{im}}
\newcommand{\C}{\mathbb{C}}
\newcommand{\K}{\mathbb{K}}
\newcommand{\N}{\mathbb{N}}
\newcommand{\Z}{\mathbb{Z}}
\newcommand{\KRR}{\K\RR}
\newcommand{\KLL}{\K\LL}
\newcommand{\lowerset}[1]{{\downarrow\!#1}}
\newcommand{\upperset}[1]{{\uparrow\!#1}}
\newcommand{\len}{\ell}
\newcommand{\ls}[1][s] {{\overleftarrow{#1}}}
\newcommand{\opi}{\overline{\pi}}
\newcommand{\rad}{\operatorname{rad}}
\newcommand{\rank}{\operatorname{rank}}
\newcommand{\sg}[1][n]{{\mathfrak{S}_{#1}}}
\newcommand{\suchthat}{\mid}
\newcommand{\tensorBA}{\otimes_{\K B} \K A}

\newcommand{\Order}{\mathcal{O}}
\newcommand{\type}{\operatorname{type}}
\newcommand{\reduce}{\operatorname{red}}

\newcommand{\sym}{{\operatorname{Sym}}}
\newcommand{\qsym}{{\operatorname{QSym}}}
\newcommand{\fqsym}{{\operatorname{FQSym}}}

\newcommand{\lfix}{\operatorname{lfix}}
\newcommand{\rfix}{\operatorname{rfix}}

\newcommand{\lcoset}[2]{{}^{#2}\!#1}
\newcommand{\rcoset}[2]{\!#1^{#2}}
\newcommand{\induced}[3]{{{#1}\!\!\uparrow}_{#2}^{#3}}

\newcommand{\biheckemonoid}[1][]{M}
\newcommand{\biheckealgebra}[1][W]{\mathcal{H}#1}
\newcommand{\wbiheckealgebra}[1]{\mathcal{H}W^{(#1)}}
\newcommand{\Mone}{{\biheckemonoid_\one}}
\newcommand{\Mzero}{{\biheckemonoid_{w_0}}}
\newcommand{\Sone}{S^\one}
\newcommand{\Szero}{S^{w_0}}
\newcommand{\SHzero}{S^{H_0}}
\newcommand{\Pone}{P^\one}

\newcommand{\Ks}[1]{\mathcal{B_R}(#1)}
\newcommand{\RKs}[1]{\mathcal{RB_R}(#1)}
\newcommand{\DKs}[1]{\mathcal{DB_R}(#1)}
\newcommand{\KKs}[1]{\mathcal{K}^{(#1)}}
\newcommand{\Js}[1]{\mathcal{B_L}(#1)}
\newcommand{\RJs}[1]{\mathcal{RB_L}(#1)}
\newcommand{\JJs}[1]{\mathcal{J}^{(#1)}}
\newcommand{\DJs}[1]{\mathcal{DB_L}(#1)}

\newcommand{\booleanlattice}[1]{{\mathcal{P}(#1)}}

\newcommand{\Des}{{\operatorname{D}_R}}
\newcommand{\Rec}{{\operatorname{D}_L}}

\newcommand{\cDes}{{\overline{\operatorname{D}}_R}}
\newcommand{\cRec}{{\overline{\operatorname{D}}_L}}

\newcommand{\Jblock}[1]{J^{(#1)}}
\newcommand{\Kblock}[1]{K^{(#1)}}

\newcommand{\RR}{\mathcal{R}}
\newcommand{\LL}{\mathcal{L}}
\newcommand{\JJ}{\mathcal{J}}
\newcommand{\HH}{\mathcal{H}}
\newcommand{\BB}{\mathcal{B}}
\newcommand{\KK}{\mathcal{K}}

\newcommand{\pred}[1]{\operatorname{Pred}(#1)}
\newcommand{\join}{\vee}
\newcommand{\meet}{\wedge}
\newcommand{\bigjoin}{\bigvee}
\newcommand{\bigmeet}{\bigwedge}
\newcommand{\irred}{I}
\newcommand{\lowersets}{O}

\newcommand{\regJcl}[1][M]{\mathcal{U}(#1)}
\newcommand{\Top}[1]{\operatorname{top}(#1)}

\makeatletter
\newskip\@bigflushglue \@bigflushglue = -100pt plus 1fil

\def\bigcentering{\let\\\@centercr\rightskip\@bigflushglue
\leftskip\@bigflushglue
\parindent\z@\parfillskip\z@skip}

\makeatother

\ifdraft
\newcommand{\TODO}[2][To do: ]{\textcolor{red}{\textbf{#1#2}}}
\else
\newcommand{\TODO}[2][]{}
\fi

\newenvironment{previewthispiece}{}{}
\usepackage{Fig/fixedsizearrowtips}

\begin{document}

\title[The biHecke monoid of a finite Coxeter group]{The biHecke monoid of a finite Coxeter group\\ and its representations}

\author{Florent Hivert}
\address{Laboratoire de Recherche en Informatique (UMR CNRS 8623),
Bâtiment 650, Université Paris Sud 11, 91405 ORSAY CEDEX, France;
CNRS; LITIS (EA 4108); IGM (UMR 8049)}
\email{florent.hivert@lri.fr}

\author{Anne Schilling}
\address{Department of Mathematics, University of California, One
 Shields Avenue, Davis, CA 95616, U.S.A.}
\email{anne@math.ucdavis.edu}

\author{Nicolas M.~Thi\'ery}
\address{Univ Paris-Sud, Laboratoire de Math\'ematiques d'Orsay,
  Orsay, F-91405; CNRS, Orsay, F-91405, France}
\email{Nicolas.Thiery@u-psud.fr}



\begin{abstract}
  For any finite Coxeter group $W$, we introduce two new objects: its cutting poset
  and its biHecke monoid. The cutting poset, constructed using a
  generalization of the notion of blocks in permutation matrices, almost forms
  a lattice on $W$.
  The construction of the biHecke monoid relies on the usual combinatorial
  model for the $0$-Hecke algebra $H_0(W)$, that is, for the symmetric group,
  the algebra (or monoid) generated by the elementary bubble sort
  operators. The authors previously introduced the Hecke group algebra,
  constructed as the algebra generated simultaneously by the bubble sort and
  antisort operators, and described its representation theory.
  In this paper, we consider instead the \emph{monoid} generated by
  these operators. We prove that it admits $|W|$ simple and projective
  modules. In order to construct the simple
  modules, we introduce for each $w\in W$ a combinatorial module $T_w$
  whose support is the interval $[1,w]_R$ in right weak order. This
  module yields an algebra, whose representation theory generalizes
  that of the Hecke group algebra, with the combinatorics of descents
  replaced by that of blocks and of the cutting poset.
\end{abstract}

\maketitle
\tableofcontents

\section{Introduction}

In this paper we introduce two novel objects for any finite Coxeter
group $W$: its \emph{cutting poset} and its \emph{biHecke monoid}. The
cutting poset is constructed using a generalization of blocks in
permutation matrices to any Coxeter group and is almost a lattice.
The biHecke monoid is generated simultaneously by the sorting and
antisorting operators associated to the combinatorial model of the
$0$-Hecke algebra $H_0(W)$. It turns out that the representation
theory of the biHecke monoid, and in particular the construction of
its simple modules, is closely tied to the cutting poset.

The study of these objects combines methods from and impacts several
areas of mathematics: Coxeter group theory, monoid theory,
representation theory, combinatorics (posets, permutations, descent
sets), as well as computer algebra. The guiding principle is the use
of representation theory, combined with computer exploration, to
extract combinatorial structures from an algebra, and in particular a
monoid algebra, often in the form of posets or lattices. This includes
the structures associated to monoid theory (such as for example Green's relations),
but also goes beyond.  For example, we find connections between the classical
orders of Coxeter groups (left, right, and left-right weak order and
Bruhat order) and Green's relations on our monoids ($\RR$, $\LL$,
$\JJ$, and $\HH$-order and ordered monoids), and these orders play a
crucial r\^ole in the combinatorics and representation theory of the
biHecke monoid.
\medskip

The usual combinatorial model for the $0$-Hecke algebra $H_0(\sg[n])$
of the symmetric group is the algebra (or monoid) generated by the
(anti) bubble sort operators $\pi_1,\ldots,\pi_{n-1}$, where $\pi_i$
acts on words of length $n$ and sorts the letters in positions $i$ and
$i+1$ decreasingly. By symmetry, one can also construct the bubble
sort operators $\opi_1,\ldots,\opi_{n-1}$, where $\opi_i$ acts by
sorting increasingly, and this gives an isomorphic construction
$\overline H_0$ of the $0$-Hecke algebra. This construction
generalizes naturally to any finite Coxeter group $W$.  Furthermore,
when $W$ is a Weyl group, and hence can be affinized, there is an
additional operator $\pi_0$ projecting along the highest root.

In~\cite{Hivert_Thiery.HeckeGroup.2007} the first and last author
constructed the \emph{Hecke group algebra} $\biheckealgebra$ by gluing
together the $0$-Hecke algebra and the group algebra of $W$ along
their right regular representation. Alternatively, $\biheckealgebra$
can be constructed as the \emph{biHecke algebra} of $W$, by gluing
together the two realizations $H_0(W)$ and $\overline H_0(W)$ of the
$0$-Hecke algebra. $\biheckealgebra$ admits a more conceptual
description as the algebra of all operators on $\K W$ preserving left
antisymmetries; the representation theory of $\biheckealgebra$
follows, governed by the combinatorics of descents.
In~\cite{Hivert_Schilling_Thiery.HeckeGroupAffine.2008}, the authors
further proved that, when $W$ is a Weyl group, $\biheckealgebra$ is a
natural quotient of the affine Hecke algebra.

In this paper, following a suggestion of Alain Lascoux, we study the
\emph{biHecke monoid} $\biheckemonoid(W)$, obtained by gluing together
the two $0$-Hecke \emph{monoids}. This involves the combinatorics of
the usual poset structures on $W$ (left, right, left-right, Bruhat
order), as well as the new cutting poset. Building upon the extensive
study of the representation theory of the $0$-Hecke
algebra~\cite{Norton.1979, Carter.1986,Denton.2010.FPSAC,Denton.2011}, we explore the
representation theory of the biHecke monoid. In the process, we prove
that the biHecke monoid is aperiodic and its Borel submonoid fixing
the identity is $\JJ$-trivial. This sparked our interest in the
representation theory of $\JJ$-trivial and aperiodic monoids, and the
general results we found along the way are presented
in~\cite{Denton_Hivert_Schilling_Thiery.JTrivialMonoids}.

We further prove that the simple and projective modules of
$\biheckemonoid$ are indexed by the elements of $W$. In order to
construct the simple modules, we introduce for each $w\in W$ a
combinatorial module $T_w$ whose support is the interval $[1,w]_R$ in
right weak order. This module yields an algebra, whose representation
theory generalizes that of the Hecke group algebra, with the
combinatorics of descents replaced by that of blocks and of the
cutting poset.

Let us finish by giving some additional motivation for the study of the biHecke monoid.
In type $A$, the tower of algebras $(\K\biheckemonoid(\sg))_{n\in \N}$
possesses long sought-after properties. Indeed, it is
well-known that several combinatorial Hopf algebras arise as
Grothendieck rings of towers of algebras. The prototypical example is
the tower of algebras of the symmetric groups which gives rise to the
Hopf algebra $\sym$ of symmetric functions, on the Schur
basis~\cite{MacDonald.SF.1995, Zelevinski.1981}. Another example, due to Krob and
Thibon~\cite{Krob_Thibon.NCSF4.1997}, is the tower of the $0$-Hecke
algebras of the symmetric groups which gives rise to the Hopf algebra
$\qsym$ of quasi-symmetric functions of~\cite{Gessel.QSym.1984}, on
the $F_I$ basis. The product rule on the $F_I$'s is naturally lifted
through the descent map to a product on permutations, leading to the
Hopf algebra $\fqsym$ of free quasi-symmetric
functions~\cite{Duchamp_Hivert_Thibon.2002}.  This calls
for the existence of a tower of algebras $(A_n)_{n\in \N}$, such that
each $A_n$ contains $H_0(\sg)$ and has its simple modules indexed by
the elements of $\sg$. The biHecke monoids $\biheckemonoid(\sg)$, and their Borel
submonoids $\Mone(\sg)$ and $\Mzero(\sg)$, satisfy these properties,
and are therefore expected to yield new representation theoretical
interpretations of the bases of $\fqsym$.

In the remainder of this introduction, we briefly review Coxeter
groups and their $0$-Hecke monoids, introduce the biHecke monoid which
is our main object of study, and outline the rest of the paper.

\subsection{Coxeter groups}

Let $(W, S)$ be a Coxeter group, that is, a group $W$ with a
presentation
\begin{equation}
  W = \langle\, S\ \suchthat\  (ss')^{m(s,s')},\ \forall s,s'\in S\,\rangle\,,
\end{equation}
with $m(s,s') \in \{1,2,\dots,\infty\}$ and $m(s,s)=1$.
The elements $s\in S$ are called \emph{simple reflections}, and the
relations can be rewritten as:
\begin{equation}
  \begin{alignedat}{2}
    s^2 &=\one &\quad& \text{ for all $s\in S$}\,,\\
    \underbrace{ss'ss's\cdots}_{m(s,s')} &=
    \underbrace{s'ss'ss'\cdots}_{m(s,s')} && \text{ for all $s,s'\in S$}\, ,
  \end{alignedat}
\end{equation}
where $\one$ denotes the identity in $W$.

Most of the time, we just write $W$ for $(W,S)$. In general, we follow
the notation of~\cite{Bjorner_Brenti.2005}, and we refer to this
monograph and to~\cite{Humphreys.1990} for details on Coxeter groups
and their Hecke algebras.  Unless stated otherwise, we always assume
that $W$ is finite, and denote its generators by $S=(s_i)_{i\in I}$,
where $I=\{1,2,\ldots, n\}$ is the \emph{index set} of $W$.

The prototypical example is the Coxeter group of type $A_{n-1}$ which is the $n$-th symmetric
group $(W,S) := (\sg[n], \{s_1,\dots,s_{n-1}\})$, where
$s_i$ denotes the \emph{elementary transposition} which exchanges $i$ and $i+1$.
The relations are given by:
\begin{equation}
  \begin{alignedat}{2}
    s_i^2           & = \one                &     & \text{ for } 1\leq i\leq n-1\,,\\
    s_i s_j         & = s_j s_i            &     & \text{ for } |i-j|\geq2\,, \\
    s_i s_{i+1} s_i & = s_{i+1} s_i s_{i+1} &\quad& \text{ for } 1\leq i\leq n-2\,;
  \end{alignedat}
\end{equation}
the last two relations are called the \emph{braid relations}.
When writing a permutation $\mu \in \sg[n]$ explicitly, we use
\emph{one-line notation}, that is the sequence
$\mu_1\mu_2\ldots\mu_n$, where $\mu_i:=\mu(i)$.

A \emph{reduced word} $i_1 \ldots i_k$ for an element $w \in W$
corresponds to a decomposition $w=s_{i_1}\cdots s_{i_k}$ of $w$
into a product of generators in $S$ of minimal length $k=\ell(w)$. A
\emph{(right) descent} of $w$ is an element $i\in I$ such that
$\ell(w s_i) < \ell(w)$. If $w$ is a permutation, this
translates into $w_i > w_{i+1}$. \emph{Left descents} are defined
analogously. The sets of left and right descents of $w$ are denoted
by $\Rec(w)$ and $\Des(w)$, respectively.

For $J\subseteq I$, we denote by $W_J = \langle s_j \mid j \in J
\rangle$ the subgroup of $W$ generated by $s_j$ with $j\in
J$. Furthermore, the longest element in $W_J$ (resp. $W$) is denoted
by $s_J$ (resp. $w_0$).
Any finite Coxeter group $W := \langle s_i \mid i \in I \rangle$ can
be realized as a finite reflection group (see for
example~\cite[Chapter 5.6]{Humphreys.1990} and~\cite[Chapter
4]{Bjorner_Brenti.2005}).  The generators $s_i$ of $W$ can be
interpreted as reflections on hyperplanes in some $|I|$-dimensional
vector space $V$. The simple roots $\alpha_i$ for $i\in I$ form a
basis for $V$; the set of all roots is given by $\Phi := \{w(\alpha_i)
\mid i\in I, w\in W\}$.  One can associate reflections $s_\alpha$ to
all roots $\alpha\in \Phi$.  If $\alpha,\beta\in \Phi$ and $w\in W$,
then $w(\alpha)=\beta$ if and only if $w s_\alpha w^{-1} = s_\beta$
(see~\cite[Chapter 5.7]{Humphreys.1990}).

\subsection{The $0$-Hecke monoid}
\label{ss.0_hecke_monoid}

The \emph{$0$-Hecke monoid} $H_0(W) = \langle \pi_i \mid i \in I
\rangle$ of a Coxeter group $W$ is generated by the \emph{simple
  projections} $\pi_i$ with relations
\begin{equation}
  \begin{alignedat}{2}
    \pi_i^2 &=\pi_i &\quad& \text{ for all $i\in I$,}\\
    \underbrace{\pi_i\pi_j\pi_i\pi_j\cdots}_{m(s_i,s_{j})} &=
    \underbrace{\pi_j\pi_i\pi_j\pi_i\cdots}_{m(s_i,s_{j})} && \text{ for all $i,j\in I$}\ .
  \end{alignedat}
\end{equation}
Thanks to these relations, the elements of $H_0(W)$ are canonically
indexed by the elements of $W$ by setting $\pi_w :=
\pi_{i_1}\cdots\pi_{i_k}$ for any reduced word $i_1 \dots i_k$ of $w$.
We further denote by $\pi_J$ the longest element of the
\emph{parabolic submonoid} $H_0(W_J) := \langle \pi_i \mid i \in J
\rangle$.

As mentioned before, any finite Coxeter group $W$ can be realized as a
finite reflection group, each generator $s_i$ of $W$ acting by
reflection along an hyperplane. The corresponding generator $\pi_i$ of
the $0$-Hecke monoid acts as a \emph{folding}, reflecting away from
the fundamental chamber on one side of the hyperplane and as the
identity on the other side. Both the action of $W$ and of $H_0(W)$
stabilize the set of reflecting hyperplanes and therefore induce an
action on chambers.

The right regular representation of $H_0(W)$, or equivalently the
action on chambers, induce a concrete realization of $H_0(W)$ as a
monoid of operators acting on $W$, with generators $\pi_1,\dots,\pi_n$
defined by:
\begin{equation}
  \label{equation.antisorting_action}
  w. \pi_i := \begin{cases}
    w & \text{if $i \in \Des(w)$,}\\
    ws_i & \text{otherwise.}
  \end{cases}
\end{equation}
In type $A$, $\pi_i$ sorts the letters at positions $i$ and $i+1$
decreasingly, and for any permutation $w$, $w.\pi_{w_0}=n\cdots21$.
This justifies naming $\pi_i$ an \emph{elementary bubble antisorting
  operator}.

Another concrete realization of $H_0(W)$ can be obtained by
considering instead the \emph{elementary bubble sorting operators}
$\opi_1,\dots,\opi_n$, whose action on $W$ are defined by:
\begin{equation}
  \label{equation.sorting_action}
  w. \opi_i := \begin{cases}
    ws_i  & \text{if $i\in \Des(w)$,}\\
    w & \text{otherwise.}
  \end{cases}
\end{equation}
In geometric terms, this is folding toward the fundamental chamber.
In type $A$, and for any permutation $w$, one has
$w.\opi_{w_0}=12\cdots n$.

\begin{remark}
\label{remark.opi_in_terms_of_pi}
For a given $w\in W$, define $v$ by $wv=w_0$, where $w_0$ is the longest element of $W$.
Then
\begin{equation*}
	i\in \Des(w) \; \Longleftrightarrow \; i\notin \Rec(v) \; \Longleftrightarrow \;
	i \notin \Des(v^{-1}) = \Des(w_0w).
\end{equation*}
Hence, the action of $\opi_i$ on $W$ can be expressed from the action of $\pi_i$ on $W$ using
$w_0$:
\begin{equation*}
	w.\opi_i = w_0 [ (w_0w).\pi_i].
\end{equation*}
\end{remark}

\subsection{The biHecke monoid $\biheckemonoid(W)$}

We now introduce our main object of study, the biHecke monoid.
\begin{definition}
\label{definition.biHecke monoid}
  Let $W$ be a finite Coxeter group. The \emph{biHecke monoid} is the
  submonoid of functions from $W$ to $W$ generated simultaneously by
  the elementary bubble sorting and antisorting operators
  of~\eqref{equation.antisorting_action}
  and~\eqref{equation.sorting_action}:
  \begin{equation*}
    \biheckemonoid := \biheckemonoid(W) :=
    \langle \pi_1, \pi_2, \ldots, \pi_n, \opi_1, \opi_2, \ldots, \opi_n \rangle\,.
  \end{equation*}
\end{definition}

As mentioned in~\cite{Hivert_Thiery.HeckeGroup.2007, Hivert_Schilling_Thiery.HeckeGroupAffine.2008}
this monoid admits several natural variants, depending on the choice of the generators:
\begin{equation*}
\begin{split}
	& \langle \pi_1, \pi_2, \ldots, \pi_n, s_1, s_2, \ldots, s_n \rangle,\\
	& \langle \pi_0, \pi_1, \pi_2, \ldots, \pi_n \rangle,
\end{split}
\end{equation*}
where $\pi_0$ is defined when $W$ is a Weyl group and hence can be
affinized. Unlike the algebras they generate, which all coincide with
the biHecke algebra (in particular due to the linear relation $1+s_i=\pi_i+\opi_i$ which expresses
how to recover a reflection by gluing together the two corresponding foldings),
these monoids are all distinct as soon as $W$ is large enough. Another
close variant is the monoid of all strictly order preserving functions
on the Boolean lattice~\cite{Gaucher.2010}. All of these monoids, and
their representation theory, remain to be studied.

\subsection{Outline}

The remainder of this paper consists of two parts: we first introduce
and study the new cutting poset structure on finite Coxeter groups,
and then proceed to the biHecke monoid and its representation theory.

In Section~\ref{section.background}, we recall some basic facts,
definitions, and properties about posets, Coxeter groups, monoids, and
representation theory that are used throughout the paper.

In Section~\ref{section.blocks_cutting_poset}, we generalize the
notion of blocks of permutation matrices to any Coxeter group, and use
it to define a new poset structure on $W$, which we call the
\emph{cutting poset}; we prove that it is (almost) a lattice, and
derive that its M\"obius function is essentially that of the
hypercube.

In Section~\ref{section.M.combinatorics}, we study the combinatorial
properties of $\biheckemonoid(W)$. In particular, we prove that it
preserves left and Bruhat order, derive consequences on the fibers and
image sets of its elements, prove that it is aperiodic, and study Green's relations
and idempotents.

In Section~\ref{section.M1}, our strategy is to consider a ``Borel''
triangular submonoid of $\biheckemonoid(W)$ whose representation
theory is simpler, but with the same number of simple modules, to
later induce back information about the representation theory of
$\biheckemonoid(W)$. Namely, we study the submonoid
$\biheckemonoid_1(W)$ of the elements fixing $\one$ in
$\biheckemonoid(W)$.  This monoid not only preserves Bruhat order, but
furthermore is \regressive.  It follows that it is $\JJ$-trivial (in
fact $\BB$-trivial) which is the desired triangularity property. It is
for example easily derived that $\biheckemonoid_1(W)$ has $|W|$ simple
modules, all of dimension $1$. In fact most of our results about
$\biheckemonoid_1$ generalize to any $\JJ$-trivial monoid, which is
the topic of a separate paper on the representation theory of $\JJ$-trivial
monoids~\cite{Denton_Hivert_Schilling_Thiery.JTrivialMonoids}.
We also provide properties of the Cartan matrix and a combinatorial description
of the quiver of $\Mone$.

In Section~\ref{section.translation_modules}, we construct, for each
$w \in W$, the \emph{translation module} $T_w$ by induction of the
corresponding simple $\K\biheckemonoid_1(W)$-module. It is a quotient of
the indecomposable projective module $P_w$ of $\K\biheckemonoid(W)$, and
therefore admits the simple module $S_w$ of $\K\biheckemonoid(W)$ as
top.  It further admits a simple combinatorial model using the right
classes with the interval $[1,w]_R$ as support, and which passes down
to $S_w$. We derive a formula for the dimension of $S_w$, using an
inclusion-exclusion on the sizes of intervals in $(W, \le_R)$ along
the cutting poset. On the way, we study the algebra
$\wbiheckealgebra{w}$ induced by the action of $\biheckemonoid(W)$ on
$T_w$. It turns out to be a natural $w$-analogue of the Hecke group
algebra, acting not anymore on the full Coxeter group, but on the
interval $[1,w]_R$ in right order.  All the properties of the Hecke
group algebra pass through this generalization, with the combinatorics
of descents being replaced by that of blocks and of the cutting
poset. In particular, $\wbiheckealgebra{w}$ is Morita equivalent to
the incidence algebra of the sublattice induced by the cutting poset
on the interval $[1,w]_\cuteq$.

In Section~\ref{section.M.representation_theory}, we apply
the findings of Sections~\ref{section.M.combinatorics}, \ref{section.M1},
and~\ref{section.translation_modules} to derive results on
the representation theory of $\biheckemonoid(W)$.
We conclude in Section~\ref{section.in_progress} with discussions on
further research in progress.

There are two appendices. Appendix~\ref{appendix.colored_graphs}
summarizes some results on colored graphs which are used in
Section~\ref{section.M.combinatorics} to prove properties of the fibers and image
sets of elements in the biHecke monoid. In Appendix~\ref{appendix.tables}
we present tables of $q$-Cartan invariant and decomposition matrices for
$\biheckemonoid(\sg[n])$ for $n=2,3,4$.

\subsection*{Acknowledgments}

A short version with the announcements of (some of) the results of
this paper was published
in~\cite{Hivert_Schilling_Thiery.BiHeckeMonoid.2010}.  The discovery
and analysis of the cutting poset is due to the last two authors. The
study of the biHecke monoid and of its representation theory is joint
work of all three authors, under an initial suggestion of Alain
Lascoux.

We would like to thank Jason Bandlow, Mahir Can, Brant Jones,
Jean-Christophe Novelli, Jean-\'Eric Pin, Nathan Reading, Vic Reiner,
Franco Saliola, Benjamin Steinberg, and Jean-Yves Thibon for
enlightening discussions. In particular, Franco Saliola pointed us to
the large body of literature on semigroups, and the recent progress in
their representation theory. Special thanks to the anonymous referee
for carefully reading the paper, suggesting many improvements, and
many insightful comments. We are grateful to Aladin Virmaux for
pointing out several misprints and helping to work out the details of
Examples~\ref{example.bihecke.Ip} and~\ref{example.M0.Ip}.

This research was driven by computer exploration, using the
open-source mathematical software \texttt{Sage}~\cite{sage} and its
algebraic combinatorics features developed by the
\texttt{Sage-Combinat} community~\cite{Sage-Combinat}.

AS would like to thank the Universit\'e Paris Sud, Orsay and the Kastler foundation for support
during her visit in November 2008. NT would like to thank the Department of Mathematics at
UC Davis for hospitality during his visit in 2007-2008 and spring 2009.
FH was partly supported by ANR grant 06-BLAN-0380.
AS was in part supported by NSF grants DMS--0652641, DMS--0652652, and DMS--1001256.
NT was in part supported by NSF grants DMS--0652641 and DMS--0652652.

\section{Background}

\label{section.background}

We review some basic facts about partial orders and finite posets in Section~\ref{ss.posets},
finite lattices and Birkhoff's theorem in Section~\ref{ss.birkhoff}, order-preserving functions
in Section~\ref{ss.order_functions}, the usual partial orders on Coxeter groups (left and right weak
order, Bruhat order) in Section~\ref{ss.orders}, and the notion of
$\JJ$-order (and related orders) and aperiodic monoids in Section~\ref{ss.preorders on monoids}.
We also prove a result in Proposition~\ref{proposition.uniqueness.idempotents} about the image
sets of order-preserving and \regressive idempotents on a poset that will be
used later in the study of idempotents of the biHecke monoid.
Sections~\ref{ss.representation_theory} and~\ref{ss.representation_theory_monoid} contain
reviews of some representation theory of algebras and monoids that will be relevant in our study
of translation modules.

\subsection{Finite posets}
\label{ss.posets}

For a general introduction to posets and lattices, we refer the reader
to e.g.~\cite{Pouzet.Ordre,Stanley.1999.EnumerativeCombinatorics1}
or~\cite[Poset, Lattice]{Wikipedia.2010}. Throughout this paper, all posets
are finite.

A partially ordered set (or \emph{poset} for short) $(P,\preceq)$ is a
set $P$ with a binary relation $\preceq$ so that for all $x,y,z\in P$:
\begin{enumerate}[(i)]
\item $x \preceq x$ (reflexivity);
\item if $x \preceq  y$ and $y \preceq x$, then $x = y$ (antisymmetry);
\item if $x \preceq y$ and $y \preceq z$, then $x \preceq z$ (transitivity).
\end{enumerate}
When we exclude the possibility that $x=y$, we write $x \prec y$.

If $x\preceq y$ in $P$, we define the \emph{interval}
\begin{equation*}
  [x,y]_P := \{z \in P \mid x \preceq z \preceq y \}.
\end{equation*}
A pair $(x,y)$ such that $x \prec y$ and there is no $z\in P$ such that
$x\prec z \prec y$ is called a \emph{covering}. We denote coverings by $x\to
y$.  The \emph{Hasse diagram} of $(P,\preceq)$ is the diagram where the
vertices are the elements $x\in P$, and there is an upward-directed edge
between $x$ and $y$ if $x\to y$.

\begin{definition}
\label{definition.convex_connected}
  Let $(P,\preceq)$ be a poset and $X\subseteq P$.
  \begin{enumerate}[(i)]
  \item $X$ is \emph{convex} if for any $x,y\in X$ with $x \preceq  y$ we have
    $[x,y]  \subseteq X$.
  \item $X$ is \emph{connected} if for any $x,y\in X$ with $x \prec  y$ there is a
    path in the Hasse diagram $x=x_0 \rightarrow x_1 \rightarrow
    \cdots \rightarrow x_k = y$ such that $x_i\in X$ for $0\le i\le k$.
  \end{enumerate}
\end{definition}

The M\"obius inversion formula~\cite[Proposition
3.7.1]{Stanley.1999.EnumerativeCombinatorics1} generalizes the
inclusion-exclusion principle to any poset. Namely, there exists a unique
function $\mu$, called the \emph{M\"obius function} of $P$, which assigns an
integer to each ordered pair $x \preceq y$
and enjoys the following property: for any two functions $f,g:
P\to G$ taking values in an additive group $G$,
\begin{equation}
g(x) = \sum_{y\preceq x} f(y)\quad\text{if and only if}\quad
f(y) = \sum_{x\preceq y} \mu(x,y)\ g(x)\,.
\end{equation}
The M\"obius function can be computed thanks to the following recursion:
\begin{equation*}
  \mu(x,y) = \begin{cases}1 & \text{if $x=y$,}\\
    -\sum_{x\preceq z \prec y} \mu(x,z), & \text{for $x\prec y$.}
  \end{cases}
\end{equation*}

\subsection{Finite lattices and Birkhoff's theorem}
\label{ss.birkhoff}

Let $(P,\preceq)$ be a poset. The \emph{meet} $z=\bigmeet A$ of a
subset $A \subseteq P$ is an element such that (1) $z\preceq x$ for
all $x \in A$ and (2) $u \preceq x$ for all $x \in A$ implies that
$u\preceq z$. When the meet exists, it is unique and is denoted by
$\bigmeet A$. The meet of the empty set $A=\{\}$ is the largest
element of the poset, if it exists. The meet of two elements $x,y\in
P$ is denoted by $x \meet y$. A poset $(P, \preceq)$ for which every
pair of elements has a meet is called a \emph{meet-semilattice}. In
that case, $P$ endowed with the meet operation is a commutative
$\JJ$-trivial semigroup, and in fact a monoid with unit the maximal
element of $P$, if the latter exists.

Reversing all comparisons, one can similarly define the \emph{join}
$\bigjoin A$ of a subset $A \subseteq P$ or $x\join y$ of two elements
$x,y\in P$, and \emph{join-semilattices}. A \emph{lattice} is a poset
for which both meets and joins exist for pair of elements. Recall that
we only consider finite posets, so we do not have to worry about the
distinction between lattices and complete lattices.

A lattice $(L,\join, \meet)$ is \emph{distributive} if the following
additional identity holds for all $x, y, z \in L$:
\begin{equation*}
    x \meet (y \join z) = (x \meet y) \join (x \meet z)\,.
\end{equation*}
This condition is equivalent to its dual:
\begin{equation*}
    x \join (y \meet z) = (x \join y) \meet (x \join z)\,.
\end{equation*}

Birkhoff's representation theorem (see e.g.~\cite[Birkhoff's
representation theorem]{Wikipedia.2010}, or~\cite[Theorem
3.4.1]{Stanley.1999.EnumerativeCombinatorics1}) states that any finite
distributive lattice can be represented as a sublattice of a Boolean
lattice, that is a collection of sets stable under union and
intersection. Furthermore, there is a canonical such representation
which we construct now.

An element $z$ in a lattice $L$ is called \emph{join-irreducible} if
$z$ is not the smallest element in $L$ and $z = x \join y$ implies $z
= x$ or $z = y$ for any $x,y\in L$ (and similarly for
meet-irreducible).  Equivalently, since $L$ is finite, $z$ is
join-irreducible if and only if it covers exactly one element in
$L$. We denote by $\irred(L)$ the \emph{poset of join-irreducible
  elements} of $L$, that is the restriction of $L$ to its
join-irreducible elements. Note that this definition still makes sense
for nonlattices. From a monoid point of view, $\irred(L)$ is the
minimal generating set of $L$.

A \emph{lower set} of a poset $P$ is a subset $Y$ of $P$ such that,
for any pair $x\leq y$ of comparable elements of $P$, $x$ is in $Y$
whenever $y$ is. \emph{Upper sets} are defined dually. The family of
lower sets of $P$ ordered by inclusion is a distributive lattice, the
\emph{lower sets lattice} $\lowersets(P)$.
Birkhoff's representation theorem~\cite{Birkhoff.1937} states that any
finite distributive lattice $L$ is isomorphic to the lattice
$\lowersets(\irred(L))$ of lower sets of the poset $\irred(L)$ of its
join-irreducible elements, via the reciprocal isomorphisms:
\begin{displaymath}
  \begin{cases}
    L & \to \lowersets(\irred(L))\\
    x & \mapsto \{y\in \irred(L) \suchthat y \leq x\}
  \end{cases} \qquad \text{and} \qquad
  \bigjoin:
  \begin{cases}
    \lowersets(\irred(L)) & \to L\\
    I & \mapsto \bigjoin I \; .
  \end{cases}
\end{displaymath}

Following Edelman~\cite{Edelman.1986}, a meet-semilattice $L$ is
\emph{meet-distributive} if for every $y\in L$, if $x\in L$ is the
meet of elements covered by $y$ then $[x,y]$ is a Boolean algebra. A
stronger condition is that any interval of $L$ is a distributive
lattice. A straightforward application of Birkhoff's representation
theorem yields that $L$ is then isomorphic to a lower set of
$\lowersets(\irred(L))$.

\subsection{Order-preserving functions}
\label{ss.order_functions}

\begin{definition}
Let $(P,\preceq)$ be a poset and $f:P\to P$ a function.
\begin{enumerate}[(i)]
\item $f$ is called \emph{order-preserving} if $x\preceq y$ implies $f(x)\preceq f(y)$.
We also say $f$ preserves the order $\preceq$.
\item $f$ is called \emph{\regressive} if $f(x) \preceq x$ for all $x\in P$.
\item $f$ is called \emph{\extensive} if $x \preceq f(x)$ for all $x\in P$.
\end{enumerate}
\end{definition}

\begin{lemma}
  \label{lemma.preimage_convex}
  Let $(P,\preceq)$ be a poset and $f:P\to P$ an order-preserving map.
  Then, the preimage $f^{-1}(C)$ of a convex subset $C \subseteq P$ is convex.
  In particular, the preimage of a point is convex.
\end{lemma}
\begin{proof}
  Let $x,y\in f^{-1}(C)$ with $x\preceq y$.
  Since $f$ is order-preserving, for any $z\in [x,y]$, we have $f(x)\preceq f(z)
  \preceq f(y)$, and therefore $f(z)\in C$.
\end{proof}

\begin{proposition}
  \label{proposition.uniqueness.idempotents}
  Let $(P,\preceq)$ be a poset and $f:P\to P$ be an order-preserving and \regressive idempotent.
  Then, $f$ is determined by its image set. Namely, for $u\in P$ we have:
  \begin{equation*}
    f(u) = \sup_\preceq \bigl( \lowerset{u} \cap \im(f) \bigr),
  \end{equation*}
  the supremum being always well-defined. Here $\lowerset{u}=\{x \in P
  \mid x\preceq u\}$.

  An equivalent statement is that, for $v \in \im(f)$,
  \begin{displaymath}
    f^{-1}(v) = \upperset{v}\  \backslash\
    \bigcup_{\substack{v'\in \im(f)\\ v'\succ v}} \upperset{v'}\,,
  \end{displaymath}
  where $\upperset{v}= \{x\in P \mid x \succeq v\}$.
\end{proposition}

\begin{proof}
  We first prove that $\lowerset{u}\cap \im(f) = f(\lowerset{u})$.
  The inclusion $\supseteq$ follows from the fact that $f$ is \regressive: taking
  $v\in \lowerset{u}$, we have $f(v)\preceq v \preceq u$ and therefore $f(v)
  \in \lowerset{u}\cap \im(f)$.
  The inclusion $\subseteq$ follows from the assumption that $f$ is an idempotent:
  for $v\in\im(f)$ with $v\preceq u$, one has $v=f(v)$, so $v\in f(\lowerset{u})$.

  Since $f$ is order-preserving, $f(\lowerset{u})$ has a unique maximal element,
  namely $f(u)$. The first statement of the proposition follows. The
  second statement is a straightforward reformulation of the first one.
\end{proof}

An \emph{interior operator} (sometimes also called a kernel operator) is a
function $L \to L$ on a lattice $L$ which is order-preserving, \regressive and
idempotent (see e.g.~\cite[Moore~Family]{Wikipedia.2010}).  A subset $A
\subseteq L$ is a \emph{dual Moore family} if it contains the smallest element
$\bot_L$ of $L$ and is stable under joins.  The image set of an interior
operator is a \emph{dual Moore family}. Reciprocally, any dual Moore family
$A$ defines an interior operator by:
\begin{equation}
  \begin{split}
    L & \longrightarrow L\\
    x & \longmapsto \reduce(x) := \bigjoin_{a\in A, a\preceq x} a \; ,
  \end{split}
\end{equation}
where $\bigjoin_{\{\}} = \bot_L$ by convention.

A (dual) Moore family is itself a lattice with the order and join inherited
from $L$. The meet operation usually differs from that of $L$ and is
given by $x\meet_A y=\reduce(x\meet_L y)$.

\subsection{Classical partial orders on Coxeter groups}
\label{ss.orders}

A Coxeter group $W = \langle s_i \mid i \in I \rangle$ comes endowed
with several natural partial orders: left (weak) order, right (weak)
order, left-right (weak) order, and Bruhat order. All of these play an
important role for the representation theory of the biHecke monoid $\biheckemonoid(W)$.

Fix $u,w\in W$. Then, in \emph{right (weak) order},
\begin{equation*}
	u\le_R w \quad \text{if $w=u s_{i_1}\cdots s_{i_k}$ for some $i_j\in I$ and
	$\ell(w)=\ell(u)+k$.}
\end{equation*}
Similarly, in \emph{left (weak) order},
\begin{equation*}
	u\le_L w \quad \text{if $w=s_{i_1}\cdots s_{i_k} u$ for some $i_j\in I$ and
	$\ell(w)=\ell(u)+k$,}
\end{equation*}
and in \emph{left-right (weak) order},
\begin{equation*}
	u\le_{LR} w \quad \text{if $w=s_{i_1}\cdots s_{i_k} u s_{i'_1} \cdots s_{i'_\ell}$
	for some $i_j, i_j'\in I$ and $\ell(w)=\ell(u)+k+\ell$.}
\end{equation*}
Note that left-right order is the transitive closure of the union of left
and right order. Thanks to associativity, this is equivalent to
the existence of a $v\in W$ such that $u\le_L v$ and $v\le_R w$.

Let $w=s_{i_1} s_{i_2} \cdots s_{i_\ell}$ be a reduced expression for
$w$. Then, in \emph{Bruhat order},
\begin{equation*}
	u\le_B w \quad \begin{array}[t]{l}
	\text{if there exists a reduced expression $u = s_{j_1} \cdots s_{j_k}$}\\
	\text{where $j_1 \ldots j_k$ is a subword of $i_1 \ldots i_\ell$.} \end{array}
\end{equation*}

For any finite Coxeter group $W$, the posets $(W,\le_R)$ and $(W,\le_L)$ are
graded lattices~\cite[Section 3.2]{Bjorner_Brenti.2005}. The following
proposition states that any interval is isomorphic to some interval starting
at $1$:
\begin{proposition} \cite[Proposition 3.1.6]{Bjorner_Brenti.2005}
  \label{proposition.interval}
  Let $\Order \in \{L,R\}$, $u\le_\Order w\in W$. Then $[u,w]_\Order \cong
  [1,t]_\Order$ where $t = wu^{-1}$.
\end{proposition}
This motivates the following definition:
\begin{definition}
  The \emph{type} of an interval in left (resp. right) order is defined to be
  $\type([u,w]_L):=wu^{-1}$ (resp. $\type([u,w]_R):= u^{-1}w$).
\end{definition}

It is easily shown that, if $\Order$ is considered as a colored poset,
then the converse of Proposition~\ref{proposition.interval} holds as well:
\begin{remark} \label{remark.interval} Fix a type $t$. Then, the
  collection of all intervals in left weak order of type $t$ is in
  bijection with $[1,t^{-1} w_0]_R$, and the operators $\pi_i$ and $\opi_i$ act
  transitively on the right on this collection. More precisely:
  $\pi_a$ induces an isomorphism from $[1, ba^{-1}]_L$ to $[a,b]_L$,
  and $\opi_{a^{-1}}$ induces an isomorphism from $[a,b]_L$ to
  $[1,ba^{-1}]_L$.
\end{remark}
\begin{proof}
  Take $u \in [a,b]_L$, and let $s_{i_1} \cdots s_{i_k}$ be a reduced
  decomposition of $a$.
  Let $s_{j_1} \cdots s_{j_\ell}$ be a reduced decomposition of $ua^{-1}= u s_{i_k}\cdots s_{i_1}$.
  Then
  $$u=(s_{j_1} \cdots s_{j_\ell})(s_{i_1}\cdots s_{i_k})$$
  is a reduced decomposition of $u$ and $u.\opi_{a^{-1}} = s_{j_1} \cdots s_{j_\ell} = ua^{-1}$.
  Reciprocially, applying $\pi_a$ to an element
  $u\in [1,ba^{-1}]_L$ progressively builds up a reduced word for $a$. The result follows.
\end{proof}

\subsection{Preorders on monoids}
\label{ss.preorders on monoids}

In 1951 Green~\cite{Green.1951} introduced several preorders on monoids which are essential
for the study of their structures (see for example~\cite[Chapter V]{Pin.2009}).
\emph{Throughout this paper, we only consider finite monoids}.
Define $\le_\RR, \le_\LL, \le_\JJ, \le_\HH$ for $x,y\in M$ as follows:
\begin{equation*}
\begin{split}
	&x \le_\RR y \quad \text{if and only if $x=yu$ for some $u\in M$}\\
	&x \le_\LL y \quad \text{if and only if $x=uy$ for some $u\in M$}\\
	&x \le_\JJ y \quad \text{if and only if $x=uyv$ for some $u,v\in M$}\\
	&x \le_\HH y \quad \text{if and only if $x\le_\RR y$ and $x\le_\LL y$.}
\end{split}
\end{equation*}
These preorders give rise to equivalence relations:
\begin{equation*}
\begin{split}
	&x \; \RR \; y \quad \text{if and only if $xM = yM$}\\
	&x \; \LL \; y \quad \text{if and only if $Mx = My$}\\
	&x \; \JJ \; y \quad \text{if and only if $MxM = MyM$}\\
	&x \; \HH \; y \quad \text{if and only if $x \; \RR \; y$ and $x \; \LL \; y$.}
\end{split}
\end{equation*}
Strict comparisons are defined by $x<_{\RR} y$ if $x\le_{\RR} y$ but
$x\notin\RR(y)$, or equivalently $\RR(x)\subset \RR(y)$, and similarly
for $<_\LL,<_\JJ,<_\HH$.

We further add the relation $\le_\BB$ (and its associated equivalence
relation $\BB$) defined as the finest preorder such that $x\leq_\BB
1$, and
\[
  \text{$x\le_\BB y$ implies that $uxv\le_\BB uyv$ for all $x,y,u,v\in M$.}
\]
(One can view $\le_\BB$ as the intersection of all preorders with the above
property. There exists at least one such preorder, namely $x\le y$ for all $x,y\in M$).
In the semigroup community, this order is sometimes colloquially referred to
as the \emph{multiplicative $\JJ$-order}.

Beware that $1$ is the largest element of those (pre)-orders. This is
the usual convention in the semigroup community, but is the converse
convention from the closely related notions of left/right/left-right/Bruhat order
in Coxeter groups as introduced in Section~\ref{ss.orders}.

\begin{example}
For the $0$-Hecke monoid introduced in Section~\ref{ss.0_hecke_monoid},
$\KK$-order for $\KK\in \{\RR,\LL,\JJ,\BB\}$ corresponds to the reverse
of right, left, left-right and Bruhat order of Section~\ref{ss.orders}. More precisely
for $x,y\in H_0(W)$, $x\le_\KK y$ if and only if $x\ge_K y$ for $\KK\in \{\RR,\LL,\JJ,\BB\}$
and $K\in \{R,L,LR,B\}$ the corresponding letter.
\end{example}

\begin{definition}
Elements of a monoid $M$ in the same $\KK$-equivalence class
are called $\KK$-classes, where $\KK\in \{\RR,\LL,\JJ,\HH,\BB\}$.
The $\KK$-class of $x\in M$ is denoted by $\KK(x)$.

A monoid $M$ is called $\KK$-\emph{trivial} if all $\KK$-classes are
of cardinality one.

An element $x\in M$ is called \emph{regular} if it is $\JJ$-equivalent to an
idempotent.
\end{definition}

An equivalent formulation of $\KK$-triviality is given in terms of
\emph{ordered} monoids. A monoid $M$ is called:
\begin{equation*}
\begin{aligned}
  & \text{\emph{right-ordered}} && \text{if $xy\le x$ for all $x,y\in M$}\\
  & \text{\emph{left-ordered}} && \text{if $xy\le y$ for all $x,y\in M$}\\
  & \text{\emph{left-right-ordered}} && \text{if $xy\leq x$ and $xy\leq y$ for all $x,y\in M$}\\
  & \text{\emph{two-sided-ordered}} && \text{if $xy=yz \leq y$ for all $x,y,z\in M$ with $xy=yz$}\\
  & \text{\emph{ordered with $1$ on top}} && \text{if $x\leq 1$, and $x\le y$ implies $uxv\le uyv$ for all $x,y,u,v\in M$}
\end{aligned}
\end{equation*}
for some partial order $\le$ on $M$.

\begin{proposition}
  \label{proposition.ordered}
  $M$ is right-ordered (resp. left-ordered, left-right-ordered, two-sided-ordered,
  ordered with $1$ on top) if and only if $M$ is $\RR$-trivial
  (resp. $\LL$-trivial, $\JJ$-trivial, $\HH$-trivial, $\BB$-trivial).

  When $M$ is $\KK$-trivial for $\KK\in \{\RR,\LL,\JJ,\HH,\BB\}$, the partial order $\le$ is finer
  than $\le_\KK$; that is for any $x, y\in M$, $x \le_\KK y$ implies $x \leq y$.
\end{proposition}
\begin{proof}
  We give the proof for right-order as the other cases can be proved in a similar fashion.

  Suppose $M$ is right-ordered and that $x,y\in M$ are in the same
  $\RR$-class. Then $x=ya$ and $y=xb$ for some $a,b\in M$. This
  implies that $x\leq y$ and $y\leq x$ so that $x=y$.
  Conversely, suppose that all $\RR$-classes are singletons. Then
  $x\le_\RR y$ and $y\le_\RR x$ imply that $x=y$, so that the
  $\RR$-preorder turns into a partial order. Hence $M$ is
  right-ordered using $xy \le_\RR x$.
\end{proof}

\begin{definition}
A monoid $M$ is \emph{aperiodic} if there is an integer $N>0$ such that for each $x\in M$,
$x^N=x^{N+1}$.
\end{definition}
Since we are only dealing with finite monoids, it is enough to find such an
$N=N_x$ depending on the element $x$. Indeed, taking $N:=\max\{N_x\}$ gives a
uniform bound.  From this definition it is clear that, for an aperiodic monoid
$M$, the sequence $(x^n)_{n\in\N}$ eventually stabilizes for every $x\in
M$. We write $x^\omega$ for the stable element, which is idempotent, and $E(M)
:= \{x^\omega \mid x\in M\}$ for the set of idempotents.

Equivalent characterizations of (finite) aperiodic monoids $M$ are that they are
$\HH$-trivial, or that the sub-semigroup $S$ of $M$ (the identity of $S$ is
not necessarily the one of $M$), which are also groups, are trivial (see for
example~\cite[VII, 4.2, Aperiodic monoids]{Pin.2009}).  In this sense, the
notion of aperiodic monoids is orthogonal to that of groups as they contain no
group-like structure. On the same token, their representation theory is
orthogonal to that of groups.

As we will see in Section~\ref{ss.aperiodic}, the biHecke monoid $\biheckemonoid(W)$
of Definition~\ref{definition.biHecke monoid} is aperiodic. Its Borel submonoid
$M_1(W)$ of functions fixing the identity is $\JJ$-trivial (see Section~\ref{section.M1}).

\subsection{Representation theory of algebras}
\label{ss.representation_theory}

We refer to~\cite{Curtis_Reiner.1962} for an introduction to
representation theory, and to \cite{Benson.1991} for more advanced
notions such as Cartan matrices and quivers. Here we mostly review
composition series and characters.

Let $A$ be a finite-dimensional algebra. Given an $A$-module $X$, any
strictly increasing sequence $(X_i)_{i\leq k}$ of submodules
\begin{equation*}
  \{0\} = X_0 \subset X_1 \subset X_2 \subset \dots \subset X_k = X
\end{equation*}
is called a \emph{filtration} of $X$. A filtration
$(Y_j)_{i\leq \ell}$ such that, for any $i$, $Y_i = X_j$ for some $j$
is called a \emph{refinement of $(X_i)_{i\le k}$}. A filtration
$(X_i)_{i\leq k}$ without a non-trivial refinement is called a
\emph{composition series}. For a composition series,
each quotient module $X_j/X_{j-1}$ is simple and is called a
\emph{composition factor}. The multiplicity of a simple module $S$ in
the composition series is the number of indices $j$ such that
$X_j/X_{j-1}$ is isomorphic to $S$.  The Jordan-H\"older theorem
states that this multiplicity does not depend on the choice of the
composition series.
Hence, we may define the \emph{generalized character} (or
\emph{character} for short) of a module $X$ as the formal sum
\begin{equation*}
  [X] := \sum_{i\in I} c_i [S_i]\,,
\end{equation*}
where $I$ indexes the simple modules of $A$ and $c_i$ is the
multiplicity of the simple module $S_i$ in any composition series for
$X$.

The additive group of formal sums $\sum_{i\in I} m_i [S_i]$, with
$m_i\in \Z$, is called the \emph{Grothendieck group of the category
of $A$-modules} and is denoted by $G_0(A)$. By definition, the
character verifies that, for any exact sequence
\begin{equation*}
  0 \to X \to Y \to Z \to 0\,,
\end{equation*}
the following equality holds in the Grothendieck group
\begin{equation*}
  [Y] = [X] + [Z]\,.
\end{equation*}
See~\cite{Serre.1977} for more information about Grothendieck groups.
\medskip

Suppose that $B$ is a subalgebra of $A$.  Any $A$-module $X$ naturally
inherits an action from $B$. The thereby constructed $B$-module is called the
\emph{restriction of $X$ to $B$} and its $B$-character $[X]_{B}$ depends only
on its $A$-character $[X]_{A}$. Indeed, any $A$-composition series can be
refined to a $B$-composition series and the resulting multiplicities
depend only on those in the $A$-composition series and in the
composition series of the simple modules of $A$ restricted to $B$. This
defines a $\Z$-linear map $[X]_{A}\mapsto [X]_{B}$, called the
\emph{decomposition map}.  Let $(S^A_i)_{i\in I}$ and $(S^B_j)_{j\in J}$ be
complete families of simple module representatives for $A$ and $B$,
respectively. The matrix of the decomposition map is called the
\emph{decomposition matrix} of $A$ over $B$; its coefficient $(i,j)$ is the
multiplicity of $S_j^B$ as a composition factor of $S^A_i$, viewed as a
$B$-module.

The adjoint construction of restriction is
called \emph{induction}: for any right $B$-module $X$ the space
\begin{equation*}
  \induced{X}{B}{A} := X \otimes_B A
\end{equation*}
is naturally endowed with a right $A$-module structure by right
multiplication by elements of $A$, and is called the \emph{module
  induced by $X$ from $B$ to $A$}.
\medskip

The next subsection, and in particular the statement of
Theorem~\ref{theorem.simple.generic}, requires a slightly more general setting,
where the identity $e$ of $B$ does not coincide with that of $A$. More
precisely, let $B$ be a subalgebra of $eAe$ for some idempotent $e$ of
$A$. Then, for any $A$-module $Y$, the \emph{restriction} of $Y$ to $B$ is
defined as $Ye$, whereas, for any $B$-module $X$, the \emph{induction} of $X$
to $A$ is defined as $\induced{X}{B}{A} := X \otimes_B eA$.

\subsection{Representation theory of monoids}
\label{ss.representation_theory_monoid}
Although representation theory started at the beginning of the 20th
century with groups before being extended to more general algebraic
structures such as algebras, one has to wait until
1942~\cite{Clifford.1942} for the first results on the representation
theory of semigroups and monoids. Renewed interest in this subject was
sparked more recently by the emergence of connections with probability
theory and combinatorics (see e.g.~~\cite{Brown.2000,Saliola.2007}).
Compared to groups, only a few
general results are known, the most important one being the
construction of the simple modules. It is originally due to Clifford,
Munn, and Ponizovski\v \i, and we recall here the construction
of~\cite{Ganyushkin_Mazorchuk_Steinberg.2009} (see also the historical
references therein) from the regular $\JJ$-classes and corresponding
right class modules.

In principle, one should be specific about the ground field $\K$; in
other words, one should consider the representation theory of the
\emph{monoid algebra} $\K M$ of a monoid $M$, and not of the monoid
itself. However, the monoids under study in this paper are aperiodic,
and their representation theory only depends on the characteristic. We
focus on the case where $\K$ is of characteristic $0$. Note that the
general statements mentioned in this section may further require $\K$
to be large enough (e.g. $\K=\C$) for non-aperiodic monoids.

Let $M$ be a finite monoid. Fix a \emph{regular} $\JJ$-class $J$,
that is, a $\JJ$-class containing an idempotent. Consider the sets
\begin{equation*}
  M_{\ge J} := \bigcup_{K\in\JJ(M),\  K\ge_\JJ J} K
  \qquad\text{and}\qquad I_J := M - M_{\ge J}\,.
\end{equation*}
Then, $I_J$ is an ideal of $M$, so that the vector space $\K M_{\ge
  J}$ can be endowed with an algebra structure by identifying it with
the quotient $\K M / \K I_J$. Note that any $\K M_{\ge J}$-module is
then a $\K M$-module.

\begin{definition}
  \label{definition.right_class_module}
  Let $f\in M$. Set $\KRR_<(f) := \K\{b \in fM \mid b <_\RR f\}$. The
  \emph{right class module} of $f$ (also known as \emph{right Schützenberger
    representation}) is the $\K M$-module
  \begin{equation*}
    \KRR(f) := \K fM / \KRR_<(f)\,.
  \end{equation*}
\end{definition}
$\KRR(f)$ is clearly a right module since $\KRR_<(f)$ is a submodule
of $\K fM$. Also, as suggested by the notation, $\RR(f)$ forms a basis
of $\KRR(f)$. Moreover, for a fixed $\JJ$-class $J$ and thanks to
associativity and finiteness, the right class module $\KRR(f)$ does
not depend on the choice of $f\in J$ (up to isomorphism). Our main
tool for studying the representation theory of the biHecke monoid will
be a combinatorial model for its right class modules, which we will
call \emph{translation modules} (see
Section~\ref{subsection.translation}).

We now choose a $\JJ$-class $J$, fix an idempotent $e_J$ in $J$, and set
$\KRR_J:=\KRR(e_J)$. Recall that
\begin{displaymath}
  \RR(e_J) = e_J M \cap J = e_J M_{\geq J} \cap J\,.
\end{displaymath}
Define similarly
\begin{equation*}
  G_J:=G_{e_J}:=e_JMe_J\cap J = e_J M_{\ge J} e_J \cap J\,.
\end{equation*}
Then, $G_J$ is a group which does not depend on the choice of
$e_J$. More precisely, if $e,f$ are two idempotents in $J$, the ideals
$MeM$ and $MfM$ are equal and the groups $G_e$ and $G_f$ are conjugate
and isomorphic. Note that when working with the quotient algebra $\K
M_{\geq J}$, the above equations simplify to:
\begin{displaymath}
  \KRR_J = e_J \K M_{\geq J} \qquad \text{ and } \qquad \K G_J = e_J \K M_{\ge J} e_J\,.
\end{displaymath}

With these notations, the simple $\K M$-modules can be constructed as follows:
\begin{theorem}[Clifford, Munn, Ponizovski\v \i,
  see~\cite{Ganyushkin_Mazorchuk_Steinberg.2009} Theorem 7]
  \label{theorem.simple.generic}
  Let $M$ be a monoid, and $\regJcl$ be the set of its regular
  $\JJ$-classes. For any $J\in\regJcl$, define the right class module
  $\KRR_J$ and groups $G_J$ as above, let $S^J_1,\dots, S^J_{n_J}$ be
  a complete family of simple $\K G_J$-modules, and set
  \begin{equation}
    X^J_i := \Top{\induced{S^J_i}{\K G_J}{\K M_{\ge J}}} =
    \Top{S^J_i \otimes_{\K G_J} e_J \K M_{\ge J}} =
    \Top{S^J_i \otimes_{\K G_J} \KRR_J}\,,
  \end{equation}
  where $\Top{X} := X / \rad{X}$ is the semi-simple quotient of the module
  $X$. Then, $(X^J_i \text{ for } J\in\regJcl \text{ and } i=1,\dots, n_J)$ is
  a complete family of simple $\K M$-modules.
\end{theorem}
In the present paper we only need the very particular case of
aperiodic monoids. The key point is that a monoid is aperiodic if and
only if all the groups $G_J$ are trivial~\cite[Proposition
4.9]{Pin.2009}: $G_J = \{e_J\}$. As a consequence, the only
$\K G_J$-module is the trivial one, $1$, so that the previous
construction boils down to the following theorem:
\begin{theorem}
  \label{theorem.simple_modules.aperiodic}
  Let $M$ be an aperiodic monoid. Choose an idempotent transversal $E=\{e_J\mid
  J\in\regJcl\}$ of the regular $\JJ$-classes. Further set
  \begin{equation}
    X^J := \Top{\induced{1}{\K e_J}{\K M_{\ge J}}} = \Top{e_J\K M_{\ge J}} =
    \Top{\KRR_J}\,.
  \end{equation}
  Then, the family $(X^J)_{J\in\regJcl}$ is a complete family of representatives
  of simple $\K M$-modules. In particular, there are as many isomorphic types of
  simple modules as regular $\JJ$-classes.
\end{theorem}
Since the top of $\KRR_J$ is simple, one obtains immediately
(see~\cite[Corollary 54.14]{Curtis_Reiner.1962}) the following
corollary.
\begin{corollary}
  \label{corollary.right_class_modules_quotient_of_projective.aperiodic}
  Each regular right class module $\KRR_J$ is indecomposable and a quotient
  of the projective module $P_J$ corresponding to $S_J$.
\end{corollary}
Note that, for a non-aperiodic finite monoid, each right class module
remains indecomposable even if its top is not necessarily simple
(see~\cite[Corollary~1.10]{Zalcstein.1971}).

It should be noted that the top of a right class module $\KRR_J$ is
easy to compute; indeed, the radical of this module is nothing but
the annihilator of $J$ acting on it. This in turn boils down to the
calculation of the kernel of a matrix as we see below.

\emph{Rees matrix monoids}~\cite{Rees.1940} play an important r\^ole in the representation theory of
monoids, because any $\JJ$-class $J$ of any monoid $M$ is, roughly
speaking, isomorphic to such a monoid. We give here the definition of
aperiodic Rees matrix monoids, which we use in a couple of examples
(see Examples~\ref{example.rees counterexample} and~\ref{example.rees}).

\begin{definition}[Aperiodic Rees matrix monoid] \label{definition.rees}
  Let $P=(p_{ij})$ be an $n\times m$ $0$-$1$-matrix. The
  \emph{aperiodic Rees matrix monoid} $M(P)$ is obtained by endowing
  the disjoint union
  \begin{displaymath}
    \{1\}\ \cup\ \{1,\dots,m\}\times \{1,\dots,n\}\ \cup\ \{0\}
  \end{displaymath}
  with the product
  \begin{displaymath}
    (i,j)(i',j') :=
    \begin{cases}
      (i,j') & \text{ if $p_{ji'}=1$, }\\
      0      & \text{ otherwise,}
    \end{cases}
  \end{displaymath}
  $1$ being neutral and $0$ being the zero element. 

  Note that $(i,j)$ is an idempotent if and only if $p_{j,i}=1$; hence
  $M(P)$ can be alternatively described by specifying which elements $(i,j)$
  are idempotent.
\end{definition}
Without entering into the details, the radical of the unique (up to
isomorphism) non-trivial right class modules of $\K M(P)$ is given by the kernel
of the matrix $P$, and thus the dimension of the non-trivial simple module of
$\K M(P)$ is given by the rank of $P$~\cite{Clifford_Preston.1961,Lallement_Petrich.1969,Rhodes_Zalcstein.1991,
Margolis_Steinberg.2011}.

\section{Blocks of Coxeter group elements and the cutting poset}
\label{section.blocks_cutting_poset}

In this section, we develop the combinatorics underlying the representation
theory of the translation modules studied in
Section~\ref{section.translation_modules}. The key question is:
given $w\in W$, for which subsets $J \subseteq I$ does the
canonical bijection between a Coxeter group $W$ and the Cartesian product $W_J
\times \lcoset{W}{J}$ of a parabolic subgroup $W_J$ by its set of coset
representatives $\lcoset{W}{J}$ in $W$ restrict properly to an interval
$[1,w]_R$ in right order (see Figure~\ref{figure.4312})? In type $A$, the
answer is given by the so-called blocks in the permutation matrix of $w$, and
we generalize this notion to any Coxeter group.

We start with some results on parabolic subgroups and quotients in
Section~\ref{ss.parabolic}, which are used to define \emph{blocks} and
\emph{cutting points} of Coxeter group elements in
Section~\ref{ss.blocks}.  Then, we illustrate the notion of blocks in
type $A$ in Section~\ref{ss.blocks.type_A}, recovering the usual
blocks in permutation matrices.  In Section~\ref{ss.cutting_poset} it
is shown that $(W,\cuteq)$ with the cutting order $\cuteq$ is a poset
(see Theorem~\ref{theorem.cutting_poset}).  In
Section~\ref{ss.lattice} we show that blocks are closed under unions
and intersections, and relate these to meets and joins in left and
right order, thereby endowing the set of cutting points of a Coxeter
group element with the structure of a distributive lattice (see
Theorem~\ref{theorem.lattice}).  In
Section~\ref{ss.indexing_cutting_points}, we discuss various indexing
sets for cutting points, which leads to the notion of $w$-analogues of
descent sets in Section~\ref{ss.w_descents}. Properties of the
\emph{cutting poset} are studied in
Section~\ref{ss.properties.cutting_poset} (see
Theorem~\ref{theorem.cutting}, which also recapitulates the previous
theorems).

Throughout this section $W := \langle s_i \mid i \in I \rangle$
denotes a finite Coxeter group.

\subsection{Parabolic subgroups and cosets representatives}
\label{ss.parabolic}

For a subset $J \subseteq I$, the \emph{parabolic subgroup} $W_J$ of $W$ is
the Coxeter subgroup of $W$ generated by $s_j$ for $j\in J$.  A complete
system of minimal length representatives of the right cosets $W_J w$
(resp. of the left cosets $w W_J$) are given respectively by:
\begin{equation*}
  \begin{split}
    \lcoset{W}{J} &:= \{ x \in W \suchthat \Rec(x) \cap J = \emptyset\}\,,\\
    \rcoset{W}{J} &:= \{ x \in W \suchthat \Des(x) \cap J = \emptyset\}\,.
  \end{split}
\end{equation*}

Every $w\in W$ has a unique decomposition $w = w_J \lcoset{w}{J}$ with
$w_J\in W_J$ and $\lcoset{w}{J} \in \lcoset{W}{J}$. Similarly, there
is a unique decomposition $w = \rcoset{w}{K} {}_Kw$ with ${}_Kw \in
{}_KW=W_K$ and $\rcoset{w}{K} \in \rcoset{W}{K}$.

\begin{lemma}
  \label{lemma.action_on_decomposition}
  Take $w\in W$.
  \begin{enumerate}[(i)]
  \item \label{item.left_decomp}
  For $J\subseteq I$ consider the unique decomposition
  $w=uv$ where $u = w_J$ and $v=\lcoset{w}{J}$. Then, the unique
  decomposition of $ws_k$ is $ws_k = (us_j) v$ if $vs_kv^{-1}$ is a
  simple reflection $s_j$ with $j\in J$ and $ws_k=u (vs_k)$ otherwise.
  \item \label{item.right_decomp}
  For $K\subseteq I$ consider the unique decomposition
  $w=vu$ where $u = {}_Kw$ and $v=\rcoset{w}{K}$. Then, the unique
  decomposition of $s_jw$ is $s_jw = v(s_ku)$ if $v^{-1}s_jv$ is a
  simple reflection $s_k$ with $k\in K$ and $s_jw=(s_j v) u$ otherwise.
  \end{enumerate}
\end{lemma}
\begin{proof}
  This follows directly from~\cite[Lemma 2.4.3 and Proposition 2.4.4]{Bjorner_Brenti.2005}.
\end{proof}

Note in particular that, if we are in case~\eqref{item.left_decomp} of
Lemma~\ref{lemma.action_on_decomposition}, we have:
\begin{itemize}
\item If $k$ is a right descent of $w$, then $(ws_k)_J \in [1,w_J]_R$ and
    $\lcoset{(ws_k)}{J}\in [1,\lcoset{w}{J}s_k]_R$.
  \item If $k$ is not a right descent of $w$, then either $s_k$
    skew commutes with $\lcoset{w}{J}$ (that is, there exists an $i$ such that
    $s_i\lcoset{w}{J}=\lcoset{w}{J}s_k)$, or
    $\lcoset{(ws_k)}{J}=\lcoset{w}{J}s_k$. In particular,
    $\lcoset{(ws_k)}{J}\leq_R\lcoset{ws_k}{J}$.
\end{itemize}

\begin{definition}
\label{definition.reduced}
A subset $J \subseteq I$ is \emph{left reduced} with respect to $w$ if $J'\subset J$
implies $\lcoset{w}{J} <_L \lcoset{w}{J'}$ (or  equivalently, if for any $j\in J$, $s_j$ appears
in some and hence all reduced words for $w_J$).

Similarly, $K\subseteq I$ is \emph{right reduced} with respect to $w$ if $K'\subset K$
implies $w^K <_R w^{K'}$.
\end{definition}

\begin{lemma}
  \label{lemma.skew_commutation}
  Let $w\in W$ and $J \subseteq I$ be left reduced with respect to $w$. Then
  \begin{enumerate}
  \item[(i)] \label{item.phiR}
        $v=\lcoset{w}{J} \le_R w$ if and only if there exists $K\subseteq I$ and
        a bijection $\phi_R: J\to K$ such that $s_jv = v s_{\phi_R(j)}$ for all $j\in J$.
  \end{enumerate}
  For $K\subseteq I$ right reduced with respect to $w$, we have
  \begin{enumerate}
  \item[(ii)] \label{item.phiL}
        $v=\rcoset{w}{K} \le_L w$ if and only if there exists $J \subseteq I$ and
        a bijection $\phi_L: K\to J$ such that $vs_k = s_{\phi_L(k)} v$ for all $k\in K$.
  \end{enumerate}
\end{lemma}
\begin{proof}
  Assume first that the bijection $\phi_R$ exists, and write
  $w=s_{j_1}\cdots s_{j_\ell}v$, where the product is reduced and $j_i\in
  J$. Then,
  \begin{equation*}
    w=s_{j_1}\cdots s_{j_\ell}v = s_{j_1}\cdots s_{j_{\ell-1}}vs_{\phi_R(j_\ell)}
    = vs_{\phi_R(j_1)}\cdots s_{\phi_R(j_\ell)}\,,
  \end{equation*}
  where the last product is reduced. Therefore $v\le_R w$.

  Assume conversely that $v=\lcoset{w}{J}\le_R w$, write the reduced expression
  $w=vs_{k_1}\cdots s_{k_\ell}\ge_R v$, and set
  $K=\{k_1,\ldots,k_\ell\}$. By Lemma~\ref{lemma.action_on_decomposition}, the sequence
  \begin{displaymath}
    v=\lcoset{v}{J}, \lcoset{(vs_{k_1})}{J}, \dots,
    \lcoset{(vs_{k_1}\cdots s_{k_\ell})}{J} = \lcoset{w}{J} = v
  \end{displaymath}
  preserves right order, and therefore is constant. Hence,
  at each step $i$
  \begin{displaymath}
    \lcoset{(vs_{k_1}\cdots s_{k_i})}{J} =
    \lcoset{(\lcoset{(vs_{k_1}\cdots s_{k_{i-1}})}{J}s_{k_i})}{J} =
    \lcoset{(vs_{k_i})}{J}=v\,.
  \end{displaymath}
  Applying Lemma~\ref{lemma.action_on_decomposition} again, it follows
  that there is a subset $J'\subseteq J$, and a bijective map $\phi_R:
  J'\to K$ such that $s_jv = v s_{\phi_R(j)}$ for all $j\in J'$.
  Then, $w = s_{\phi_R^{-1}(k_1)}\cdots s_{\phi_R^{-1}(k_\ell)}v$,
  and, since $J$ is left reduced, $J=J'$.

  The second part is the symmetric statement.
\end{proof}

By Lemma~\ref{lemma.action_on_decomposition}, for any $w\in W$ and $J\subseteq I$
we have $[1,w]_R \subseteq [1,w_J]_R [1,\lcoset{w}{J}]_R$ and similarly for any $K\subseteq I$
we have $[1,w]_L \subseteq [1,\rcoset{w}{K}]_L [1,{}_Kw]_L$.

\begin{lemma}
  \label{lemma.order_representative}
  Take $w\in W$, $K\subseteq I$, and assume that $s_iw = w s_k$ for $i\in I$ and $k\in K$,
  where the products are reduced.
  Then, there exists $k'\in K$ such that $s_i w^K = w^K s_{k'}$, where the products are
  again reduced.
\end{lemma}
\begin{proof}
  We have $w^K = (w s_k)^K = (s_i w)^K = (s_i w^K)^K$. Hence, by
  Lemma~\ref{lemma.action_on_decomposition}~\eqref{item.right_decomp} there exists
  $k'\in K$ such that $w^K s_{k'} = s_i w^K$, as desired.
\end{proof}

\subsection{Definition and characterizations of blocks and cutting points}
\label{ss.blocks}

We now come to the definition of \emph{blocks} of Coxeter group elements and
associated \emph{cutting points}. They will lead to a new poset on the Coxeter group $W$,
which we coin the \emph{cutting poset} in Section~\ref{ss.cutting_poset}.

\begin{definition}[Blocks and cutting points]
  \label{definition.block}
  \label{definition.cutting_point}
  Let $w\in W$.  We call $K\subseteq I$ a \emph{right block} (resp. $J
  \subseteq I$ a \emph{left block}) of $w$, if there exists
  $J\subseteq I$ (resp. $K\subseteq I$) such that
  \begin{equation*}
    W_J w = w W_K\,.
  \end{equation*}
  In that case, $v:=w^K$ is called a \emph{cutting point} of $w$,
  which we denote by $v \cuteq w$. Furthermore, $K$ is \emph{proper}
  if $K\neq \emptyset$ and $K\neq I$; it is \emph{nontrivial} if
  $\rcoset{w}{K}\neq w$ (or equivalently ${}_K w\neq 1$); analogous
  definitions are made for left blocks.

  We denote by $\Ks{w}$ the set of all right blocks for $w$, and by $\RKs{w}$
  the set of all (right) reduced (see definition~\ref{definition.reduced})
  right blocks for $w$. The sets $\Js{w}$, $\RJs{w}$ are similarly defined on
  the left.
\end{definition}

Here is an equivalent characterization of blocks which also shows that
cutting points can be equivalently defined using $\lcoset{w}{J}$
instead of $w^K$.
\begin{proposition}
  \label{proposition.equiv_block}
  Let $w\in W$ and $J, K\subseteq I$. Then, the following are equivalent:
  \begin{enumerate}[(i)]
  \item \label{i.block} $W_J w = w W_K$;
  \item \label{i.bijection} There exists a bijection $\phi:K\to J$
    such that $w^K s_k = s_{\phi(k)} w^K$
    (or equivalently $w^K(\alpha_k) = \alpha_{\phi(k)}$)
    for all $k\in K$.
  \end{enumerate}

  Furthermore, when any, and therefore all, of the above hold then,
  \begin{enumerate}[(iii)]
  \item \label{i.JK} $\rcoset{w}{K} = \lcoset{w}{J}$.
  \end{enumerate}
\end{proposition}

\begin{proof}
  Suppose \eqref{i.block} holds. Then $W_J \lcoset{w}{J} = w^K W_K$.
  Since $\lcoset{w}{J}$ has no left descents in $J$ and $w^K$ has no right descents
  in $K$, we know that on both sides $\lcoset{w}{J}$ and $w^K$ are the shortest elements
  and hence have to be equal: $\lcoset{w}{J} = w^K$; this proves
  (iii). %
  Furthermore, every reduced expression $w^K s_k$ with
  $k\in K$ must correspond to some reduced expression $s_j
  \lcoset{w}{J}$ for some $j\in J$, and vice versa. Hence there exists
  a bijection $\phi :K\to J$ such that $w^K s_k = s_{\phi(k)}
  \lcoset{w}{J} = s_{\phi(k)} w^K$. Therefore
  point~\eqref{i.bijection} holds.

  Suppose now that point~\eqref{i.bijection} holds. Then, for any
  expression $s_{k_1} \cdots s_{k_\ell} \in W_K$, we have
  \begin{equation*}
    w^K s_{k_1} \cdots s_{k_\ell} =
    s_{\phi(k_1)}w^K s_{k_2} \cdots s_{k_\ell} = \dots =
    s_{\phi(k_1)} \cdots s_{\phi(k_\ell)} w^K\,.
  \end{equation*}
  It follows that
  \begin{equation*}
    w^K W_K = W_J w^K.
  \end{equation*}
  In particular $w\in W_J w^K$ and therefore
  \begin{equation*}
    W_J w = W_J w^K = w^K W_K = w W_K\,.\qedhere
  \end{equation*}
\end{proof}

In general, condition~(iii) of Proposition~\ref{proposition.equiv_block} is only
a necessary, but not sufficient condition for $K$ to be a block.
See Example~\ref{example.block_condition}.

\begin{proposition}
  \label{proposition.tiling}
  If $K$ is a right block of $w$ (or more generally if $w^K=w^{K'}$ with $K'$
  a right block), then the bijection
  \begin{equation*}
    \begin{cases}
      \rcoset{W}{K} \times {}_K W &\to W\\
      (v,u) &\mapsto vu
    \end{cases}
  \end{equation*}
  restricts to a bijection $[1,\rcoset{w}{K}]_L \times [1,{}_K w]_L
  \to [1,w]_L$.

  Similarly, if $J$ is a left block (or more generally if
  $\lcoset{w}{J}=\lcoset{w}{J'}$ with $J'$ a left block), then the
  bijection
  \begin{equation*}
    \begin{cases}
      W_J \times \lcoset{W}{J} &\to W\\
      (u,v) &\mapsto uv
    \end{cases}
  \end{equation*}
  restricts to a bijection $[1,w_J]_R \times [1,\lcoset{w}{J}]_R \to
  [1,w]_R$ (see Figure~\ref{figure.4312}).
\end{proposition}

\begin{proof}
  By Proposition~\ref{proposition.equiv_block} we know that, if $K$ is
  a right block, then there exists a bijection $\phi:K\to J$ such that
  $w^K s_k = s_{\phi(k)} w^K$. Hence the map $y\mapsto w^K y$ induces
  a skew-isomorphism between $[1,{}_K w]_L$ and $[w^K,w]_L$, where an
  edge $k$ is mapped to edge $\phi(k)$. It follows in particular that
  $uv\le_L w^K \; v \le_L w^K \; {}_K w =w$ for any $u\in [1,w^K]_L$
  and $v\in [1, {}_K w]_L$, as desired.

  Assume now that $K$ is not a block, but $w^K=w^{K'}$ with $K'$ a
  block. Then, $[1,w^K]_L=[1,w^{K'}]_L$ and $[1,{}_K
  w]_L=[1,{}_{K'}w]_L$ and we are reduced to the previous case.

  The second statement can be proved in the same fashion.
\end{proof}

Due to Proposition~\ref{proposition.tiling}, we also say that
$[1,v]_R$ \emph{tiles} $[1,w]_R$ if $v=\lcoset{w}{J}$ for some left
block $J$ (or equivalently $v= \rcoset{w}{K}$ for some right block
$K$).

\begin{proposition}
  \label{proposition.reduced_block}
  Let $w\in W$ and $K$ be right reduced with respect to $w$. Then, the
  following are equivalent:
  \begin{enumerate}[(i)]
  \item \label{i.LR} $K$ is a reduced right block of $w$;
  \item \label{i.red_block} $\rcoset{w}{K}\leq_L w$.
  \end{enumerate}
  The analogous statement can be made for left blocks.
\end{proposition}
See also Proposition~\ref{proposition.cond_Jtiling} for yet another equivalent condition
of reduced blocks.

\begin{proof}[Proof of Proposition~\ref{proposition.reduced_block}]
  If $K$ is a right block, then by
  Proposition~\ref{proposition.equiv_block} we have $\rcoset{w}{K} =
  \lcoset{w}{J}$, where $J$ is the associated left block. In
  particular, $\rcoset{w}{K} = \lcoset{w}{J} \leq_L w$.

  The converse statement follows from Lemma~\ref{lemma.skew_commutation} and
  Proposition~\ref{proposition.equiv_block}.
\end{proof}

\begin{example}
\label{example.reduced_blocks_w0}
  For $w=w_0$, any $K\subseteq I$ is a reduced right block; of course
  $\;\rcoset{w_0}{K} \le_L w_0$ and ${}_K w_0$ is the maximal element
  of the parabolic subgroup $W_K = {}_K W$. The cutting point
  $w^K\cuteq w$ is the maximal element of the right descent class for
  the complement of $K$.

  The associated left block is given by $J=\phi(K)$, where $\phi$ is the
  automorphism of the Dynkin diagram induced by conjugation by $w_0$
  on the simple reflections.  The tiling corresponds to the usual
  decomposition of $W$ into right $W_K$ cosets, or of $W$ into left
  $W_J$ cosets.
\end{example}

\subsection{Blocks of permutations}
\label{ss.blocks.type_A}

In this section we illustrate the notion of blocks and cutting points introduced in the
previous section for type $A$. We show that, for a permutation $w \in \sg[n]$, the
blocks of Definition~\ref{definition.block} correspond to the usual
notion of blocks of the permutation matrix of $w$ (or unions thereof),
and the cutting points $w^K$ for right blocks $K$ correspond to
putting the identity in those blocks.

A \emph{matrix-block} of a permutation $w$ is an interval $[k',k'+1,\dots,k]$
which is mapped to another interval. Pictorially, this corresponds to
a square submatrix of the matrix of $w$ which is again a permutation
matrix (that of the \emph{associated permutation}).  For example, the
interval $[2,3,4,5]$ is mapped to the interval $[4,5,6,7]$ by the
permutation $w=36475812\in \sg[8]$, and is therefore a matrix-block of
$w$ with associated permutation $3142$. Similarly, $[7,8]$ is a
matrix-block with associated permutation $12$:
\begin{center}
\scalebox{.6}[-.6]{
$
  \begin{array}{|c|cccc|c|cc|}
    \hline
     & & & & &\bullet &&\\
    \hline
     & & &\bullet& & & &\\
     &\bullet& & & & & &\\
     & & & &\bullet& & &\\
     & &\bullet& & & & &\\
    \hline
    \bullet& & & & & & &\\
    \hline
     & & & & & & &\bullet\\
     & & & & & &\bullet &\\
    \hline
  \end{array}
$}
\end{center}
For any permutation $w$, the singletons $[i]$ and the full set
$[1,2,\dots,n]$ are always matrix-blocks; the other matrix-blocks of
$w$ are called \emph{proper}.  A permutation with no proper
matrix-block, such as $58317462$, is called \emph{simple}.
See~\cite{Nozaki_Miyakawa_Pogosyan_Rosenberg.1995,Albert_Atkinson_Klazar.2003,Albert_Atkinson.2005}
for a review of simple permutations. Simple permutations are also
strongly related to dimension $2$ posets.

A permutation $w\in \sg[n]$ is \emph{connected} if it does
not stabilize any subinterval $[1,\dots,k]$ with $1\leq k<n$, that is if
$w$ is not in any proper parabolic subgroup $\sg[k]\times \sg[n-k]$.
Pictorially, this means that there are no diagonal matrix-blocks.  A matrix-block is
\emph{connected} if the corresponding induced permutation is
connected. In the above example, the matrix-block $[2,3,4]$ is connected,
but the matrix-block $[7,8]$ is not.

\begin{proposition}
  \label{proposition.blocks.matrix_blocks}
  Let $w\in \sg[n]$. The right blocks of $w$ are in bijection with
  disjoint unions of (nonsingleton) matrix-blocks for $w$; each
  matrix-block with column set $[i,i+1,\dots,k]$ contributes
  $\{i,i+1,\dots,k-1\}$ to the right block; each matrix-block with row
  set $[i,i+1,\ldots,k]$ contributes $\{i,i+1,\dots,k-1\}$ to the left
  block.

  In addition, trivial right blocks correspond to unions of identity
  matrix-blocks.  Also, reduced right blocks correspond to unions of
  connected matrix-blocks.
\end{proposition}

\begin{proof}
  Suppose $w\in \sg[n]$ with a disjoint union of matrix-blocks with
  consecutive column sets $[i_1,\ldots,k_1]$ up to
  $[i_\ell,\ldots,k_\ell]$. Set $K_j=\{i_j,\ldots,k_j-1\}$ for $1\le
  j\le \ell$ and $K=K_1\cup \cdots \cup K_\ell$. Define similarly $J$
  according to the rows of the blocks.

  Then, multiplying $w$ on the right by some element of $W_K$ permutes
  some columns of $w$, while stabilizing each blocks. Therefore, the
  same transformation can be achieved by some permutation of the rows
  stabilising each block, that is by multiplication of $w$ on the left
  by some element of $W_J$. Hence, using symmetry, $W_J w = w W_K$,
  that is $J$ and $K$ are corresponding left and right blocks for $w$.

  Conversely, if $K$ is a right block of $w$, then $w^K$ maps each
  $\alpha_k$ with $k\in K$ to another simple root by
  Proposition~\ref{proposition.equiv_block}.  But then, splitting
  $K=K_1\cup\cdots\cup K_\ell$ into consecutive subsets with $K_j =
  \{i_j,\ldots,k_j-1\}$, the permutation $w^K$ must contain the
  identity permutation in each matrix-block with column indices
  $[i_j,\ldots,k_j]$. This implies that $w$ itself has matrix-blocks
  with column indices $[i_j,\ldots,k_j]$ for $1\le j\le \ell$.

  Note that, in the described correspondence, $w^K=w$ if and only if
  all matrix-blocks contain the identity. This proves the statement
  about trivial right blocks.

  A reduced right block $K$ has the property that for every
  $K'\subset K$ we have $w^{K'} \neq w^K$. This implies that no
  matrix-block is in a proper parabolic subgroup, and hence they are all
  connected.
\end{proof}

\begin{example}
\label{example.4312}
  As in Figure~\ref{figure.4312}, consider the permutation $4312$, whose permutation matrix is:
  \begin{center}
  \scalebox{.6}[-.6]{$
    \begin{array}{|c|c|cc|}
      \hline
      \bullet&&&\\\hline
      &\bullet&&\\    \hline
      & & &\bullet\\
      & &\bullet &\\
      \hline
    \end{array}
  $}
  \end{center}
  The reduced (right)-blocks are $K=\{\}$, $\{1\}$, $\{2,3\}$, and
  $\{1,2,3\}$. The cutting points are $4312$, $3412$, $4123$, and
  $1234$, respectively. The corresponding left blocks are $J=\{\}$,
  $\{3\}$, $\{1,2\}$ and $\{1,2,3\}$, respectively.  The nonreduced
  (right) blocks are $\{3\}$ and $\{1,3\}$, as they are respectively
  equivalent to the blocks $\{\}$ and $\{1\}$. The trivial blocks are
  $\{\}$ and $\{3\}$.
\end{example}

\begin{example}
\label{example.block_condition}
  In general, condition~(iii) of Proposition~\ref{proposition.equiv_block}
  is only a necessary, but not sufficient condition for $K$ to be a block.
  For example, for $w=43125$ (similar to $4312$ of Example~\ref{example.4312}, but
  embedded in $\sg[5]$), $J=\{3,4\}$, and $K=\{1,4\}$, one has
  $\lcoset{w}{J}=\rcoset{w}{K}$ yet neither $J$ nor $K$ are blocks. On
  the other hand (iii) of Proposition~\ref{proposition.equiv_block} becomes both necessary
  and sufficient for reduced blocks.
\end{example}

\begin{figure}
  \centerline{

\tikzstyle{cutting}=[shape=ellipse,draw=black,inner sep=.5pt]
\tikzstyle{interval}=[ultra thick]
\begin{tikzpicture}[baseline=(current bounding box.east), >= fixedsizelatex]
  \node[cutting] (1234) at ( 0, 0) {1234};

  \node (1324) at ( 1, 1) {1324};

  \node (3124) at ( 1.5, 2) {3124};
  \node (1342) at ( 2.5, 2) {1342};

  \node (3142) at ( 3, 3) {3142};
  \node[cutting] (3412) at ( 4, 4) {3412};

  \node (1243) at (-1, 1) {1243};
  \node (1423) at ( 0, 2) {1423};

  \node[cutting] (4123) at (0.5, 3) {4123};
  \node (1432) at ( 1.5, 3) {1432};

  \node (4132) at (2, 4) {4132};

  \node[cutting] (4312) at ( 3, 5) {4312};

  \draw[<-,red,interval]   (1234)  -- (1324);
  \draw[<-,green] (1234)  -- (1243);

  \draw[blue,interval]  (1324)  -- (3124);
  \draw[green,interval] (1324)  -- (1342);

  \draw[red,interval]   (1243)  -- (1423);

  \draw[<-,red]   (1342)  -- (1432);
  \draw[green,interval] (3124)  -- (3142);
  \draw[blue,interval]  (1342)  -- (3142);

  \draw[blue,interval]  (1423)  -- (4123);
  \draw[<-,green,interval] (1423)  -- (1432);

  \draw[red,interval]   (3142)  -- (3412);

  \draw[blue,interval]  (1432)  -- (4132);
  \draw[<-,green,interval] (4123)  -- (4132);

  \draw[<-,blue]  (3412)  -- (4312);
  \draw[red,interval]   (4132)  -- (4312);
\end{tikzpicture} \quad 

\tikzstyle{cutting}=[shape=ellipse,draw=black,inner sep=.5pt]
\tikzstyle{interval}=[ultra thick]
\begin{tikzpicture}[baseline=(current bounding box.east), >=fixedsizelatex]
  \node[cutting] (1234) at ( 0, 0) {1234};

  \node (1324) at (-1, 1) {1324};
  \node (1243) at ( 1, 1) {1243};

  \node (3124) at (-2, 2) {3124};
  \node (1342) at ( 0, 2) {1342};
  \node (1423) at ( 2, 2) {1423};

  \node (3142) at (-1, 3) {3142};
  \node (1432) at ( 1, 3) {1432};
  \node[cutting] (4123) at ( 3, 3) {4123};

  \node[cutting] (3412) at ( 0, 4) {3412};
  \node (4132) at ( 2, 4) {4132};

  \node[cutting] (4312) at ( 1, 5) {4312};

  \draw[<-,red]   (1234)  -- (1324);
  \draw[<-,green,interval] (1234)  -- (1243);

  \draw[blue]  (1324)  -- (3124);
  \draw[green,interval] (1324)  -- (1342);

  \draw[red,interval]   (1243)  -- (1423);

  \draw[green,interval] (3124)  -- (3142);
  \draw[blue]  (1342)  -- (3142);
  \draw[<-,red,interval]   (1342)  -- (1432);

  \draw[blue,interval]  (1423)  -- (4123);
  \draw[<-,green] (1423)  -- (1432);

  \draw[red,interval]   (3142)  -- (3412);

  \draw[blue,interval]  (1432)  -- (4132);
  \draw[<-,green] (4123)  -- (4132);

  \draw[<-,blue,interval]  (3412)  -- (4312);
  \draw[red]   (4132)  -- (4312);
\end{tikzpicture}

}
  \caption{Two pictures of the interval $[1234,4312]_R$ in right order in $\sg[4]$
    illustrating its proper tilings, for $J:=\{3\}$ and $J:=\{1,2\}$,
    respectively. The thick edges highlight the tiling. The circled
    permutations are the cutting points, which are at the top of the
    tiling intervals.  Blue, red, green lines correspond to $s_1$,
    $s_2$, $s_3$, respectively. See Section~\ref{subsection.bihecke}
    for the definition of the orientation of the edges (this is
    $G^{(4312)}$); edges with no arrow tips point in both directions.}
  \label{figure.4312}
\end{figure}
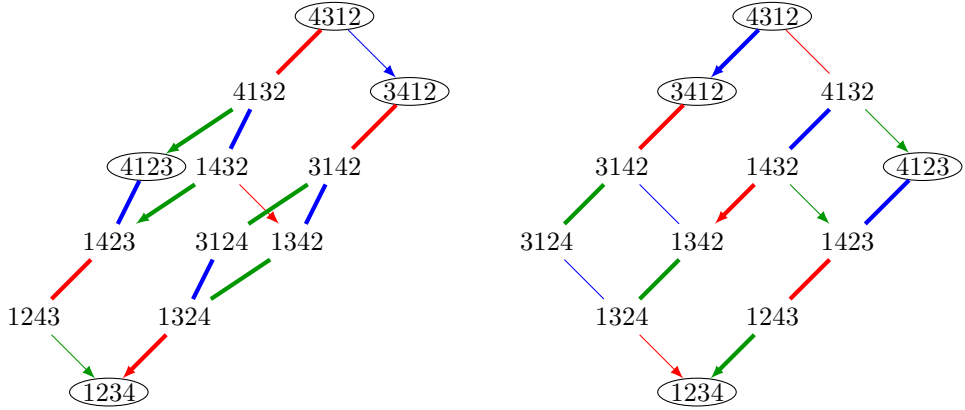

\begin{remark}
\label{remark.union.typeA}
  It is obvious that the union and intersection of overlapping (possibly with a trivial overlap)
  matrix-blocks in $\sg[n]$ are again matrix-blocks; we will see in
  Proposition~\ref{proposition.block.union} that this property generalizes to all types.
\end{remark}

\begin{problem}
  Fix $J \subseteq \{1,2,\ldots,n-1\}$ and enumerate the permutations $w\in \sg[n]$
  for which $J$ is a left block.
\end{problem}

\subsection{The cutting poset}
\label{ss.cutting_poset}

In this section, we show that $(W,\cuteq)$ indeed forms a poset.
We start by showing that for a fixed $u\in W$, the set of elements $w$ such
that $u\cuteq w$ admits a simple description. Recall that for $J \subseteq I$,
we denote by $s_J$ the longest element of $W_J$.
Proposition~\ref{proposition.equiv_block} suggests the following definition.
\begin{definition}
  \label{definition.short_nondescents}
  Let $u\in W$. We call $k\in I$ a \emph{short right nondescent}
  (resp. $j\in I$ a \emph{short left nondescent}) of $u$ if there
  exists $j\in I$ (resp. $k \in I$) such that
  \begin{equation*}
    s_j u = u s_k\, ,
  \end{equation*}
  where the product is reduced (that is $j$ and $k$ are nondescents).
  An equivalent condition is that $u$ maps the simple root $\alpha_k$
  to a simple root (resp. the preimage of $\alpha_j$ is a simple
  root).

  Set further
  \begin{equation*}
    U_u := u W_K= [u, u s_K]_R\, = W_J u= [u,s_J u]_L\,,
  \end{equation*}
  where $K:=K(u)$ (resp. $J:=J(u)$) is the set of short right
  (resp. left) nondescents of $u$.
\end{definition}

Pictorially, one takes left and right order on $W$ and associates to
each vertex $u$ the translate $U_u$ above $u$ of the parabolic
subgroup generated by the short nondescents of $u$, which correspond
to the simultaneous covers of $u$ in both left and right order.

\begin{example}
In type $A$, $i$ is short for $u\in \sg[n]$ if $u(i+1)=u(i)+1$, that is,
there is a $2\times2$ identity block in columns $(i,i+1)$ of the
permutation matrix of $u$. Furthermore $U_u$ is obtained by looking at
all identity blocks in $u$ and replacing each by any permutation
matrix.

The permutation $4312$ of Example~\ref{example.4312} has a single nondescent $3$
which is short, and $U_{4312}=\{4312, 4321\}$.
\end{example}

\begin{proposition}
  \label{proposition.cutting.U}
  $U_u$ is the set of all $w$ such that $u\cuteq w$.

  In particular, it follows that:
  \begin{itemize}
  \item If $u \le_R v \le_R w$ and $u\cuteq w$, then $u\cuteq v$.
  \item If $u\cuteq w$ and $u\cuteq w'$, then $u\cuteq w\join_R w'$.
  \end{itemize}
\end{proposition}
\begin{proof}
  Note that $w$ is in $U_u$ if and only if there exists $K$ such that
  $K\subseteq K(u)$ and $w^K=u$. By
  Proposition~\ref{proposition.equiv_block}, this is equivalent to the
  existence of a block $K$ such that $w^K=u$, that is $u\cuteq w$.
\end{proof}

The following related lemma is used to prove that $(W,\cuteq)$ is a poset.
\begin{lemma}
  \label{lemma.short_descents}
  If $u \cuteq w$, then the set of short nondescents of $w$ is a
  subset of the short nondescents of $u$, namely $K(w) \subseteq
  K(u)$.
\end{lemma}
\begin{proof}
  Let $k \in K(w)$, so that $ws_k = s_jw$ for some $j\in I$ and both sides
  are reduced. It follows from Lemma~\ref{lemma.order_representative} that
  there exists $k'\in K(w)$ such that $s_j u = u s_{k'}$ and both sides are reduced.
  Hence $k'\in K(u)$. Since the map $k\mapsto k'$ is injective it follows that
  $K(w) \subseteq K(u)$.
\end{proof}

\begin{theorem}
  \label{theorem.cutting_poset}
  $(W,\cuteq)$ is a subposet of both left and right order.
\end{theorem}
\begin{proof}
  The relation $\cuteq$ is reflexive since $v$ is a cutting point of
  $v$ with right block $\emptyset$; hence $v \cuteq v$. Applying
  Proposition~\ref{proposition.equiv_block}, it is a subrelation of
  left and right order: if $v \cuteq w$ then $v=w^K\le_Rw$ for some
  $K$ and $v=\lcoset{w}{J}\le_Lw$ for some $J$. Antisymmetry follows
  from the antisymmetry of left (or right) order.

  For transitivity, let $v\cuteq w$ and $w\cuteq z$.  Then $v=w^K$ and
  $w=z^{K'}$ for some right block $K$ of $w$ and $K'$ of $z$. We claim
  that $v= z^{K\cup K'}$ with $K\cup K'$ a right block of
  $z$. Certainly $k\not \in \Des(v)$ for $k\in K$ since $v=w^K$. Since
  $w=z^{K'}$ with $K'$ a block of $z$, all $k'\in K'$ are short
  nondescents of $w$ and hence by Lemma~\ref{lemma.short_descents}
  also short nondescents of $v$. This proves the claim. Therefore
  $v\cuteq z$.
\end{proof}

\begin{example}
The cutting poset for $\sg[3]$ and $\sg[4]$ is given in
Figure~\ref{figure.cutting_poset}. As we can see on those figures, the cutting
poset is not the intersection of the right and left order since $w_0$ is
maximal for left and right order but not for cutting poset.
\end{example}

\begin{sidewaysfigure}
  \begin{bigcenter}
   \vspace{7cm}

    \scalebox{.7}{\begin{tikzpicture}[>=latex,line join=bevel,]
  \def\b{\bullet} 
  \def\a{}

  \node (213) at (206bp,102bp) [draw,draw=none] {$\left(\begin{array}{rrr}\a & \b & \a \\\b & \a & \a \\\a & \a & \b\end{array}\right)$};
  \node (312) at (34bp,102bp) [draw,draw=none] {$\left(\begin{array}{rrr}\a & \b & \a \\\a & \a & \b \\\b & \a & \a\end{array}\right)$};
  \node (231) at (120bp,102bp) [draw,draw=none] {$\left(\begin{array}{rrr}\a & \a & \b \\\b & \a & \a \\\a & \b & \a\end{array}\right)$};
  \node (132) at (292bp,102bp) [draw,draw=none] {$\left(\begin{array}{rrr}\b & \a & \a \\\a & \a & \b \\\a & \b & \a\end{array}\right)$};
  \node (123) at (163bp,22bp) [draw,draw=none] {$\left(\begin{array}{rrr}\b & \a & \a \\\a & \b & \a \\\a & \a & \b\end{array}\right)$};
  \node (321) at (77bp,182bp) [draw,draw=none] {$\left(\begin{array}{rrr}\a & \a & \b \\\a & \b & \a \\\b & \a & \a\end{array}\right)$};

  \draw [<-] (321) ..controls (61bp,152bp) and (56bp,142bp)  .. (312);
  \draw [<-] (231) ..controls (136bp,72bp) and (141bp,62bp)  .. (123);
  \draw [<-] (321) ..controls (93bp,152bp) and (98bp,142bp)  .. (231);
  \draw [<-] (213) ..controls (190bp,72bp) and (185bp,62bp)  .. (123);
  \draw [<-] (132) ..controls (242bp,71bp) and (223bp,59bp)  .. (123);
  \draw [<-] (312) ..controls (84bp,71bp) and (103bp,59bp)  .. (123);
\end{tikzpicture}}

    \scalebox{.4}{\begin{tikzpicture}[>=latex,line join=bevel,]
\def\b{\bullet} 
\def\a{}
\node (4132) at (1368bp,207bp) [draw,draw=none] {$\left(\begin{array}{rrrr}\a& \b& \a& \a\\\a& \a& \a& \b\\\a& \a& \b& \a\\\b& \a& \a& \a\end{array}\right)$};
  \node (1324) at (654bp,117bp) [draw,draw=none] {$\left(\begin{array}{rrrr}\b& \a& \a& \a\\\a& \a& \b& \a\\\a& \b& \a& \a\\\a& \a& \a& \b\end{array}\right)$};
  \node (4321) at (1062bp,297bp) [draw,draw=none] {$\left(\begin{array}{rrrr}\a& \a& \a& \b\\\a& \a& \b& \a\\\a& \b& \a& \a\\\b& \a& \a& \a\end{array}\right)$};
  \node (2431) at (756bp,207bp) [draw,draw=none] {$\left(\begin{array}{rrrr}\a& \a& \a& \b\\\b& \a& \a& \a\\\a& \a& \b& \a\\\a& \b& \a& \a\end{array}\right)$};
  \node (3142) at (756bp,117bp) [draw,draw=none] {$\left(\begin{array}{rrrr}\a& \b& \a& \a\\\a& \a& \a& \b\\\b& \a& \a& \a\\\a& \a& \b& \a\end{array}\right)$};
  \node (3412) at (1062bp,117bp) [draw,draw=none] {$\left(\begin{array}{rrrr}\a& \a& \b& \a\\\a& \a& \a& \b\\\b& \a& \a& \a\\\a& \b& \a& \a\end{array}\right)$};
  \node (3214) at (399bp,207bp) [draw,draw=none] {$\left(\begin{array}{rrrr}\a& \a& \b& \a\\\a& \b& \a& \a\\\b& \a& \a& \a\\\a& \a& \a& \b\end{array}\right)$};
  \node (2134) at (246bp,117bp) [draw,draw=none] {$\left(\begin{array}{rrrr}\a& \b& \a& \a\\\b& \a& \a& \a\\\a& \a& \b& \a\\\a& \a& \a& \b\end{array}\right)$};
  \node (2341) at (960bp,117bp) [draw,draw=none] {$\left(\begin{array}{rrrr}\a& \a& \a& \b\\\b& \a& \a& \a\\\a& \b& \a& \a\\\a& \a& \b& \a\end{array}\right)$};
  \node (4123) at (1164bp,117bp) [draw,draw=none] {$\left(\begin{array}{rrrr}\a& \b& \a& \a\\\a& \a& \b& \a\\\a& \a& \a& \b\\\b& \a& \a& \a\end{array}\right)$};
  \node (1243) at (42bp,117bp) [draw,draw=none] {$\left(\begin{array}{rrrr}\b& \a& \a& \a\\\a& \b& \a& \a\\\a& \a& \a& \b\\\a& \a& \b& \a\end{array}\right)$};
  \node (3124) at (450bp,117bp) [draw,draw=none] {$\left(\begin{array}{rrrr}\a& \b& \a& \a\\\a& \a& \b& \a\\\b& \a& \a& \a\\\a& \a& \a& \b\end{array}\right)$};
  \node (1342) at (858bp,117bp) [draw,draw=none] {$\left(\begin{array}{rrrr}\b& \a& \a& \a\\\a& \a& \a& \b\\\a& \b& \a& \a\\\a& \a& \b& \a\end{array}\right)$};
  \node (1234) at (603bp,27bp) [draw,draw=none] {$\left(\begin{array}{rrrr}\b& \a& \a& \a\\\a& \b& \a& \a\\\a& \a& \b& \a\\\a& \a& \a& \b\end{array}\right)$};
  \node (4312) at (1164bp,207bp) [draw,draw=none] {$\left(\begin{array}{rrrr}\a& \a& \b& \a\\\a& \a& \a& \b\\\a& \b& \a& \a\\\b& \a& \a& \a\end{array}\right)$};
  \node (4213) at (1266bp,207bp) [draw,draw=none] {$\left(\begin{array}{rrrr}\a& \a& \b& \a\\\a& \b& \a& \a\\\a& \a& \a& \b\\\b& \a& \a& \a\end{array}\right)$};
  \node (1423) at (552bp,117bp) [draw,draw=none] {$\left(\begin{array}{rrrr}\b& \a& \a& \a\\\a& \a& \b& \a\\\a& \a& \a& \b\\\a& \b& \a& \a\end{array}\right)$};
  \node (4231) at (1062bp,207bp) [draw,draw=none] {$\left(\begin{array}{rrrr}\a& \a& \a& \b\\\a& \b& \a& \a\\\a& \a& \b& \a\\\b& \a& \a& \a\end{array}\right)$};
  \node (1432) at (603bp,207bp) [draw,draw=none] {$\left(\begin{array}{rrrr}\b& \a& \a& \a\\\a& \a& \a& \b\\\a& \a& \b& \a\\\a& \b& \a& \a\end{array}\right)$};
  \node (3241) at (858bp,207bp) [draw,draw=none] {$\left(\begin{array}{rrrr}\a& \a& \a& \b\\\a& \b& \a& \a\\\b& \a& \a& \a\\\a& \a& \b& \a\end{array}\right)$};
  \node (2413) at (144bp,117bp) [draw,draw=none] {$\left(\begin{array}{rrrr}\a& \a& \b& \a\\\b& \a& \a& \a\\\a& \a& \a& \b\\\a& \b& \a& \a\end{array}\right)$};
  \node (2143) at (144bp,207bp) [draw,draw=none] {$\left(\begin{array}{rrrr}\a& \b& \a& \a\\\b& \a& \a& \a\\\a& \a& \a& \b\\\a& \a& \b& \a\end{array}\right)$};
  \node (2314) at (348bp,117bp) [draw,draw=none] {$\left(\begin{array}{rrrr}\a& \a& \b& \a\\\b& \a& \a& \a\\\a& \b& \a& \a\\\a& \a& \a& \b\end{array}\right)$};
  \node (3421) at (960bp,207bp) [draw,draw=none] {$\left(\begin{array}{rrrr}\a& \a& \a& \b\\\a& \a& \b& \a\\\b& \a& \a& \a\\\a& \b& \a& \a\end{array}\right)$};
  \draw [<-] (4321) ..controls (1103bp,261bp) and (1115bp,251bp)  .. (4312);
  \draw [<-] (4321) ..controls (1021bp,261bp) and (1009bp,251bp)  .. (3421);
  \draw [<-] (4321) ..controls (1062bp,262bp) and (1062bp,253bp)  .. (4231);
  \draw [<-] (3214) ..controls (380bp,172bp) and (374bp,162bp)  .. (2314);
  \draw [<-] (1432) ..controls (584bp,172bp) and (578bp,162bp)  .. (1423);
  \draw [<-] (3241) ..controls (899bp,171bp) and (911bp,161bp)  .. (2341);
  \draw [<-] (1243) ..controls (87bp,92bp) and (90bp,91bp)  .. (93bp,90bp) .. controls (177bp,63bp) and (439bp,39bp)  .. (1234);
  \draw [<-] (2143) ..controls (185bp,171bp) and (197bp,161bp)  .. (2134);
  \draw [<-] (4231) ..controls (1103bp,171bp) and (1115bp,161bp)  .. (4123);
  \draw [<-] (3421) ..controls (960bp,172bp) and (960bp,163bp)  .. (2341);
  \draw [<-] (1324) ..controls (635bp,82bp) and (629bp,72bp)  .. (1234);
  \draw [<-] (2413) ..controls (189bp,92bp) and (192bp,91bp)  .. (195bp,90bp) .. controls (319bp,49bp) and (472bp,34bp)  .. (1234);
  \draw [<-] (3142) ..controls (695bp,82bp) and (673bp,69bp)  .. (1234);
  \draw [<-] (1432) ..controls (685bp,186bp) and (746bp,168bp)  .. (1342);
  \draw [<-] (3214) ..controls (418bp,172bp) and (424bp,162bp)  .. (3124);
  \draw [<-] (4312) ..controls (1123bp,171bp) and (1111bp,161bp)  .. (3412);
  \draw [<-] (1423) ..controls (571bp,82bp) and (577bp,72bp)  .. (1234);
  \draw [<-] (4213) ..controls (1225bp,171bp) and (1213bp,161bp)  .. (4123);
  \draw [<-] (3421) ..controls (1001bp,171bp) and (1013bp,161bp)  .. (3412);
  \draw [<-] (4231) ..controls (1021bp,171bp) and (1009bp,161bp)  .. (2341);
  \draw [<-] (2314) ..controls (393bp,92bp) and (396bp,91bp)  .. (399bp,90bp) .. controls (450bp,69bp) and (509bp,51bp)  .. (1234);
  \draw [<-] (2143) ..controls (103bp,171bp) and (91bp,161bp)  .. (1243);
  \draw [<-] (4132) ..controls (1323bp,183bp) and (1320bp,181bp)  .. (1317bp,180bp) .. controls (1284bp,164bp) and (1245bp,148bp)  .. (4123);
  \draw [<-] (2431) ..controls (801bp,183bp) and (804bp,181bp)  .. (807bp,180bp) .. controls (849bp,161bp) and (863bp,163bp)  .. (2341);
  \draw [<-] (2134) ..controls (291bp,92bp) and (294bp,91bp)  .. (297bp,90bp) .. controls (383bp,59bp) and (489bp,42bp)  .. (1234);
  \draw [<-] (4123) ..controls (1119bp,92bp) and (1116bp,91bp)  .. (1113bp,90bp) .. controls (1030bp,63bp) and (768bp,39bp)  .. (1234);
  \draw [<-] (3412) ..controls (1017bp,92bp) and (1014bp,91bp)  .. (1011bp,90bp) .. controls (888bp,49bp) and (735bp,34bp)  .. (1234);
  \draw [<-] (3124) ..controls (511bp,82bp) and (533bp,69bp)  .. (1234);
  \draw [<-] (4312) ..controls (1164bp,172bp) and (1164bp,163bp)  .. (4123);
  \draw [<-] (1342) ..controls (813bp,92bp) and (810bp,91bp)  .. (807bp,90bp) .. controls (757bp,69bp) and (697bp,51bp)  .. (1234);
  \draw [<-] (2341) ..controls (915bp,92bp) and (912bp,91bp)  .. (909bp,90bp) .. controls (823bp,59bp) and (717bp,42bp)  .. (1234);
\end{tikzpicture}}
    \end{bigcenter}
    \caption{The cutting posets for the symmetric groups $\sg[3]$ and
      $\sg[4]$.  Each permutation is represented by its permutation
      matrix, with the bullets marking the positions of the ones.
      Notice the Boolean sublattice appearing as the interval between
      the identity permutation at the bottom and the maximal
      permutation at the top; its elements are the minimal elements of
      the descent classes.}
  \label{figure.cutting_poset}
\end{sidewaysfigure}

\subsection{Lattice properties of intervals}
\label{ss.lattice}
In this section we show that the set of blocks and the set of cutting points $\{u\suchthat u\cuteq w\}$
of a fixed $w\in W$ are endowed with the structure of distributive lattices
(see Theorem~\ref{theorem.lattice}).

We begin with a lemma which gives some properties of blocks that are contained in each other.

\begin{lemma}
  \label{lemma.subblocks}
  Fix $w\in W$. Let $K\subseteq K'$ be two right blocks of $w$ and
  $J\subseteq J'$ be the corresponding left blocks, so that
  \begin{equation*}
    W_J  w = w W_K,   \quad
    W_{J'}w = w W_{K'}, \quad
    \lcoset{w}{J } = w^K   \cuteq w, \quad \text{ and } \quad
    \lcoset{w}{J'} = w^{K'} \cuteq w\,.
  \end{equation*}
  Then,
  \begin{enumerate}[(i)]
  \item \label{item.LR.block} $w^{K'}\leq_R w^K$ and $w^{K'}\leq_L w^K$;
  \item \label{item.Kp.upper} $K'$ is a right block of $w^K$ and $w^{K'}\cuteq w^K$;
  \item \label{item.K.lower} $K$ is a right block of ${}_{K'}w$ and ${}_{K'} w^K \cuteq
    {}_{K'} w$.\\
    Furthermore $K$ is reduced for ${}_{K'}w$ if and only
    if it is reduced for $w$.
  \end{enumerate}
  The same statements hold for left blocks.
\end{lemma}
\begin{proof}
  \eqref{item.LR.block} holds because $w^{K'} = (w^{K})^{K'} \leq_R w^K \leq_R w$, and
  similarly on the left.

  \eqref{item.Kp.upper} is a trivial consequence of~\eqref{item.LR.block} and
  Proposition~\ref{proposition.cutting.U}.

  For~\eqref{item.K.lower}, first note that $({}_{K'}w)^K={}_{K'}(w^K)$, so that the
  notation ${}_{K'} w^K$ is unambiguous. Consider the bijection $\phi$
  from $K'$ to $J'$ of Proposition~\ref{proposition.equiv_block}, and
  note that $W_J w^{K'} = w^{K'} W_{\phi^{-1}(J)}$. Therefore,
  \begin{equation*}
    w^{K'} {}_{K'}w W_K = w W_K = W_J\; w = W_J\; w^{K'} {}_{K'}w =
    w^{K'} W_{\phi^{-1}(J)} \; {}_{K'}w\,.
  \end{equation*}
  Simplifying by $w^{K'}$ on the left, one obtains that
  \begin{equation*}
    {}_{K'}w \; W_K = W_{\phi^{-1}(J)} \; {}_{K'}w\,,
  \end{equation*}
  proving that $K$ is also a block of ${}_{K'}w$. The reduction
  statement is trivial.
\end{proof}

We saw in Remark~\ref{remark.union.typeA} that the set of blocks is closed under
unions and intersections in type $A$. This holds for general type.

\begin{proposition}
  \label{proposition.block.union}
  The set $\Ks{w}$ (resp. $\Js{w}$) of right (resp. left) blocks is
  stable under union and intersection. Hence, it forms a distributive
  sublattice of the Boolean lattice $\booleanlattice{I}$.
\end{proposition}
\begin{proof}
  Let $K$ and $K'$  be right blocks for $w\in W$, and $J$ and $J'$
  be the corresponding left blocks, so that:
  \begin{equation*}
  	w W_K = W_J w \quad \text{and} \quad w W_{K'} = W_{J'} w.
  \end{equation*}

  Take $u\in W_{K\cap K'} = W_K \cap W_{K'}$. Then, $w u w^{-1}$ is
  both in $W_J$ and $W_{J'}$ and therefore in $W_J\cap W_{J'} =
  W_{J\cap J'}$. This implies $w W_{K\cap K'} w^{-1}\subseteq W_{J\cap J'}$.
  By symmetry, the inclusion $w^{-1} W_{J\cap J'} w \subseteq
  W_{K\cap K'}$ holds as well, and therefore $W_{J\cap J'} w = w
  W_{K\cap K'}$. In conclusion, $K\cap K'$ is a right block, with $J\cap
  J'$ as corresponding left block.

  Now take $u \in W_{K \cup K'} = \langle W_K, W_{K'}\rangle$, and
  write $u$ as a product $u_1 u'_1 u_2 u'_2\cdots u_\ell u'_\ell$, where
  $u_i \in W_K$ and $u'_i \in W_{K'}$ for all $1\le i\le \ell$. Then, for each $i$,
  $w u_i w^{-1} \in W_{J}$ and $w u'_i w^{-1} \in W_{J'}$. By
  composition, $w u w^{-1}\in W_{J} W_{J'} W_{J} W_{J'}\cdots W_J
  W_{J'} \subseteq W_{J\cup J'}$. Using symmetry as above, we conclude
  that $w W_{K\cup K'} = W_{J\cup J'} w$. In summary, $K\cup K'$ is a
  right block, with $J\cup J'$ as corresponding left block.

  Finally, since blocks are stable under union and intersection, they form a sublattice
  of the Boolean lattice. Any sublattice of a distributive lattice is distributive.
\end{proof}

Next we relate the union and intersection operation on blocks with the
meet and join operations in right and left order. We start with the following
general statement which must be classical, though we have not found it
in the literature.
\begin{lemma}
  \label{lemma.intersection_join}
  Take $w\in W$ and $J,J',K,K'\subseteq I$. Then
  \begin{equation*}
    w^{K\cap K'} = w^K \join_R w^{K'} \qquad \text{and} \qquad
    \lcoset{w}{J\cap J'} = \lcoset{w}{J} \join_L \lcoset{w}{J'}\,.
  \end{equation*}
\end{lemma}
\begin{proof}
  We include a proof for the sake of completeness. By
  Lemma~\ref{lemma.subblocks} (i), $w^K, w^{K'} \le_R w^{K\cap K'}$,
  and therefore $v \le_R w^{K\cap K'}$, where $v=w^K\join_R w^{K'}$.
  Suppose that $v$ has a right descent $k\in K\cap K'$.  Then $vs_k$
  is still bigger than $w^K$ and $w^{K'}$ in right order, a
  contradiction to the definition of $v$. Hence $w^{K\cap K'} = w^K
  \join_R w^{K'}$, as desired. The statement on the left follows by
  symmetry.
\end{proof}

\begin{corollary}
  Take $w\in W$. Let $K, K' \subseteq I$ be two right blocks of $w$
  and $J, J' \subseteq I$ the corresponding left blocks. Then, for the right
  block $K\cap K'$ (resp. left block $J\cap J'$)
  \begin{equation*}
    w^{K\cap K'} = \lcoset{w}{J\cap J'} = w^K \join_R w^{K'} = \lcoset{w}{J} \join_L \lcoset{w}{J'}\,.
  \end{equation*}
\end{corollary}

The analogous statement of Lemma~\ref{lemma.intersection_join} for
unions fails in general: take for example $w=4231$ and $K=\{3\}$ and
$K'=\{1,2\}$, so that $w^K=4213$ and $w^{K'}=2341$; then $w^{K\cup K'}
= 1234$, but $w^K \meet_R w^{K'} = 2134$. However, it holds for
blocks:
\begin{lemma}
  \label{lemma.blocks_union_meet}
  Take $w\in W$. Let $K, K' \subseteq I$ be two right blocks of $w$
  and $J, J' \subseteq I$ the corresponding left blocks. Then, for the
  right block $K\cup K'$ (resp. left block $J \cup J'$):
  \begin{equation*}
    w^{K\cup K'} = \lcoset{w}{J\cup J'} = w^K \meet_R w^{K'} = \lcoset{w}{J} \meet_L \lcoset{w}{J'}\,.
  \end{equation*}
  Furthermore, $K\cup K'$ is reduced whenever $K$ and $K'$ are reduced, and
  similarly for the left blocks.
\end{lemma}
\begin{proof}
  By symmetry, it is enough to prove the statements for right blocks.

  By Lemma~\ref{lemma.subblocks}~\eqref{item.LR.block}, $w^{K\cup K'} \le_R w^K,
  w^{K'}$, and therefore $w^{K\cup K'} \le_R w^K \meet_R w^{K'}$.

  Note that the interval $[w^{K\cup K'}, w]_R$ contains all the
  relevant points: $w^K$, $w^{K'}$, and $w^K \meet_R w^{K'}$. Consider
  the translate of this interval obtained by dividing on the left by
  $w^{K\cup K'}$, or equivalently by using the map $u\mapsto {}_{K\cup
    K'}u$. By Lemma~\ref{lemma.subblocks}~\eqref{item.K.lower}, $K$ and $K'$ are
  still blocks of ${}_{K\cup K'} w$. From now on, we may therefore
  assume without loss of generality that $w^{K\cup K'}=1$. It follows
  at once that $[1,w]_R$ lies in the parabolic subgroup $W_{K\cup K'}$
  and that $J\cup J'=K\cup K'$.

  If $w^K \meet_R w^{K'}=1=w^{K\cup K'}$, then we are done. Otherwise,
  let $i\in K\cup K'=J\cup J'$ be the first letter of some reduced
  word for $w^K \meet_R w^{K'}$. Since $w^K \meet_R w^{K'}$ is in the
  interval $[1,w^K]_R$, $i$ cannot be in $J$; by symmetry $i$ cannot
  be in $J'$ either, a contradiction.

  Assume further that $K$ and $K'$ are reduced. Then, any $k\in K$
  appears in any reduced word for ${}_K w$, and therefore in any
  reduced word for ${}_{K\cup K'}w$ since ${}_K w \leq_L {}_{K\cup K'}
  w$. By symmetry, the same holds for $k'\in K'$. Hence $K\cup K'$ is
  reduced.
\end{proof}

\begin{theorem}
\label{theorem.lattice}
  The map $K\mapsto w^K$ (resp. $J\mapsto \lcoset{w}{J}$) defines a
  lattice antimorphism from the lattice $\Ks{w}$ (resp. $\Js{w}$) of
  right (resp. left) blocks of $w$ to both right and left order on
  $W$.

  The set of cutting points for $w$, which is the image set
  \begin{equation*}
    \{w^K \mid K\in \Ks{w}\}=\{\lcoset{w}{J} \mid J\in \Js{w}\}
  \end{equation*}
  of the previous map, is a distributive sublattice of right (resp. left)
  order.
\end{theorem}
\begin{proof}
  The first statement is the combination of
  Lemmas~\ref{lemma.intersection_join}
  and~\ref{lemma.blocks_union_meet}. The second statement follows from
  Proposition~\ref{proposition.block.union}, since the quotient of a
  distributive sublattice by a lattice morphism is a distributive
  lattice.
\end{proof}

\begin{corollary}
  \label{corollary.interval}
  Every interval of $(W,\cuteq)$ is a distributive sublattice and an induced
  subposet of both left and right order.
\end{corollary}
\begin{proof}
  Take an interval in $(W,\cuteq)$; without loss of generality, we may
  assume that it is of the form $[1,w]_\cuteq = \{w^K \mid K\in
  \RKs{w}\}$. The interval $[1,w]_\cuteq$ is not only a subposet of
  left (resp. right) order, but actually the induced subposet; indeed
  for $K$ and $K'$ right reduced blocks, and $J$ and $J'$ the
  corresponding left blocks,
  \begin{equation*}
    w^K \leq_L w^{K'}  \Leftrightarrow w^K \leq_R w^{K'}
    \Leftrightarrow J'\subseteq J \Leftrightarrow K'\subseteq K
    \Leftrightarrow w^K \leq_\cuteq w^{K'}\,.
  \end{equation*}
  Therefore, using Theorem~\ref{theorem.lattice}, it is a
  distributive sublattice of left (resp. right) order.
\end{proof}

Let us now consider the lower covers in the cutting poset for a fixed $w\in W$.
They correspond to nontrivial blocks $J$ which are minimal for inclusion, and in
particular reduced.
\begin{lemma}
\label{lemma.minimal}
  Each minimal nontrivial (left) block $J$ for $w\in W$ contains at least one
  element which is in no other minimal nontrivial block for $w$.
\end{lemma}
\begin{proof}
  Assume otherwise. Then, $J$ is the union of its intersections with
  the other nontrivial blocks. Each such intersection is necessarily
  a trivial block, and a union of trivial blocks is a trivial block.
  Therefore, $J$ is a trivial block, a contradiction.
\end{proof}

\begin{corollary}
  \label{corollary.hypercube}
  The semilattice of unions of minimal nontrivial blocks for a fixed
  $w\in W$ is free.
\end{corollary}
\begin{proof}
  This is a straightforward consequence of Lemma~\ref{lemma.minimal}.
  Alternatively, this property is also a direct consequence of
  Corollary~\ref{corollary.interval}, since it holds in general for
  any distributive lattice.
\end{proof}

\subsection{Index sets for cutting points}
\label{ss.indexing_cutting_points}

Recall that by Theorem~\ref{theorem.lattice} the cutting points of $w$
form a distributive lattice. Hence, by Birkhoff's representation
theorem, they can be indexed by some collection of subsets closed
under unions and intersections.
We therefore now aim at finding a suitable choice of indexing scheme
for the cutting points of $w$. More precisely, for each $w$, we are
looking for a pair $(\KKs{w}, \phi^{(w)})$, where $\KKs{w}$ is a
subset of some Boolean lattice (typically $\booleanlattice{I}$) such
that $\KKs{w}$ ordered by inclusion is a lattice, and
\begin{equation*}
  \phi^{(w)}\ :\ \KKs{w} \longrightarrow [1,w]_\cuteq
\end{equation*}
is an isomorphism (or antimorphism) of lattices.

Here are some of the desirable properties of this indexing:
\begin{enumerate}
\item The indexing gives a Birkhoff's representation of the lattice of
  cutting points of $w$. Namely, $\KKs{w}$ is a sublattice of the
  chosen Boolean lattice, and unions and intersections of indices
  correspond to joins and meets of cutting points.
\item The isomorphism $\phi^{(w)}$ is given by the map $J\mapsto
  \lcoset{w}{J}$. In that case the choice amounts to defining a
  section of those maps.
\item The indexing generalizes the usual combinatorics of descents.
\item The indices are blocks: $\KKs{w}\subseteq\Js{w}$.
\item We may actually want to have two indexing sets $\KKs{w}$ and
  $\KKs{w}$, one on the left and one on the right, with a natural
  isomorphism between them.
\item The index of $u$ in $\KKs{w}$ does not depend on $w$ (as long as
  $u$ is a cutting point of $w$). One may further ask for this index
  to not depend on $W$, so that the indexing does not change through
  embedding of parabolic subgroups.
\end{enumerate}
Unfortunately, there does not seem to be an ideal choice satisfying all of
these properties at once, and we therefore propose several imperfect
alternatives.

\subsubsection{Indexing by reduced blocks}

The first natural choice is to take reduced blocks as indices; then,
$\KKs{w}=\RKs{w}$ (and similarly $\JJs{w}=\RJs{w}$ on the left). This
indexing scheme satisfies most of the desired properties, except that
it does not provide a Birkhoff representation, and depends on $w$.
\begin{remark}
  By Lemma~\ref{lemma.blocks_union_meet}, if $K,K'\subseteq I$ are
  reduced right blocks for $w$, then $K\cup K'$ is also
  reduced. However, this is not necessarily the case for $K\cap K'$:
  consider for example the permutation $w=4231$, $K=\{1,2\}$ and
  $K'=\{2,3\}$; then $K\cap K'=\{2\}$ is a block which is equivalent
  to the reduced block $\{\}$: $4231^{\{2\}}=4231=4231^{\{\}}$.

  The union $K\cup K'$ of two blocks may be reduced even when the
  blocks are not both reduced. Consider for example the permutation
  $w=4312$ as in Figure~\ref{figure.4312}. Then $K=\{1,3\}$ and
  $K'=\{2,3\}$ are blocks and their union $K\cup K'=\{1,2,3\}$ is
  reduced, yet $K$ is not reduced.
\end{remark}

\begin{proposition}
  \label{proposition.cutting_points.sublattice}
  The poset $(\RKs{w}, \subseteq)$ of reduced right blocks is a
  distributive lattice, with the meet and join operation given
  respectively by:
  \begin{equation*}
    K\join K' = K \cup K' \qquad \text{and} \qquad
    K\meet K' = \reduce(K \cap K') \; ,
  \end{equation*}
  where, for a block $K$, $\reduce(K)$ is the unique largest reduced block
  contained in $K$.

  The map $\phi^{(w)}: K\mapsto w^K$ restricts to a lattice
  antiisomorphism from the lattice $\Ks{w}$ of reduced right blocks of
  $w$ to $[1,w]_\cuteq$.

  The same statements hold on the left.
\end{proposition}
\begin{proof}
  By Proposition~\ref{proposition.block.union} and
  Lemma~\ref{lemma.blocks_union_meet}, $\RKs{w}$ is a dual Moore
  family of the Boolean lattice of $I$, or even of $\Ks{w}$.
  Therefore, using Section~\ref{ss.posets}, it is a lattice, with
  the given join and meet operations.

  The lattice antiisomorphism of property follows from
  Lemma~\ref{lemma.blocks_union_meet} and the coincidence of right
  order and $\cuteq$ on $[1,w]_\cuteq$ (Theorem~\ref{theorem.lattice}).
\end{proof}

\subsubsection{Indexing by largest blocks}

The indexing by reduced blocks corresponds to the section of the
lattice morphism $K\mapsto w^K$ by choosing the smallest block $K$ in
the fiber of a cutting point $u$. Instead, one could choose the
largest block in the fiber of $u$, which is given by the set of short
nondescents of $u$.  This indexing scheme is independent of $w$. Also,
by the same reasoning as above, the indexing sets $\JJs{w}$ come
endowed with a natural lattice structure. However, it does not give a
Birkhoff representation: the meet is given by intersection,
but the join is not given by union (take $w=2143$; its cutting
points are $1234$, $1243$, $2134$, and $2143$, indexed respectively by
$\{1,2,3\}$, $\{1\}$, $\{3\}$, and $\{\}$).

\subsubsection{Birkhoff's representation using non-blocks}

We now relax the condition for the indices to be blocks. That is, we
consider $K\mapsto w^K$ as a function from the full Boolean lattice
$\booleanlattice{I}$ to the minimal coset representatives of $w$.  Beware that
this map is no longer a lattice antimorphism; yet, the fiber of any $u$ still
admits a largest set $K=\cDes(u)\subseteq I$, which is the complement of the
right descent set of $u$. One can define a similar indexing on the left by
$J=\cRec(u)$. These indexings are independent of $w$ and provide a Birkhoff
representation for the lattice of cutting points (see
Proposition~\ref{proposition.cutting_birkhoff}).  Define
\begin{equation}
\label{equation.complements_descents}
	\DJs{w} = \{\cRec(u)\suchthat u\cuteq w\} \qquad \text{and} \qquad
  	\DKs{w} = \{\cDes(u)\suchthat u\cuteq w\} \; .
\end{equation}

\begin{remark}
  Since $\cRec(u)$ and $\cDes(u)$ are not necessarily blocks anymore,
  the bijection between $\cRec(u)$ and $\cDes(u)$ is not induced
  anymore by a bijection at the level of descents: for example, for
  $u=3142$, one has $\cRec(u)=\{1,3\}$ and $\cDes(u)=\{2\}$.
\end{remark}

\begin{remark}
  Using $\Des(u)$ instead of $\cDes(u)$ would give an isomorphism
  instead of an antiisomorphism, and make the indexing further
  independent of $W$, at the price of slightly cluttering the notation
  $w^K$ for cutting points.
\end{remark}

\begin{proposition}[Birkhoff representation for the lattice of cutting points]
  \label{proposition.cutting_birkhoff}
  The set $\DKs{w}$ of Equation~\eqref{equation.complements_descents}
  is a sublattice of the Boolean lattice, and the maps $K\mapsto w^K$
  and $u\mapsto \cDes(u)$ form a pair of reciprocal lattice
  antiisomorphisms with the lattice of cutting points of $w$.
  The same statement holds on the left.
\end{proposition}

The proof of this Proposition uses the following property of left and
right order (recall that $[1,w]_\cuteq$ is a sublattice thereof).
\begin{lemma}[{\cite[Lemme 5]{PolyBarbut.1994}}]
  \label{lemma.desrec.lattice.morphism}
  The maps
  \begin{displaymath}
    \begin{cases}
      (W, \leq_L) & \to \booleanlattice{I}\\
      w           &\mapsto \Des(w)
    \end{cases} \qquad
    \begin{cases}
     (W, \leq_R) & \to \booleanlattice{I}\\
      w           &\mapsto \Rec(w)
    \end{cases}
  \end{displaymath}
  are surjective lattice morphisms.
\end{lemma}

\begin{proof}[Proof of Proposition~\ref{proposition.cutting_birkhoff}]
  By construction, $\cRec$ is a section of $K\mapsto w^K$, and these
  maps form a pair of reciprocal bijections between $\DJs{w}$ and the
  cutting points of $w$. Using
  Lemma~\ref{lemma.desrec.lattice.morphism}, the map $\cRec$ is a
  lattice antimorphism. Therefore its image set $\DKs{w}$ is a
  sublattice of the Boolean lattice. The argument on the left is the
  same.
\end{proof}

\subsection{A $w$-analogue of descent sets}
\label{ss.w_descents}

For each $w\in W$, we now provide a definition of a
$w$-analogue on the interval $[1,w]_R$ of the usual combinatorics of
(non)descents on $W$. From now on, we assume that we have chosen an
indexation scheme so that the cutting points of $w$ are given by
$(w^K)_{K\in \KKs{w}}$ or equivalently by $(\lcoset{w}{J})_{J\in
  \JJs{w}}$.

\begin{lemma}
  \label{lemma.condition_interval}
  Take a cutting point of $w$, and write it as $w^K=\lcoset{w}{J}$ for
  some $J,K\subseteq I$, which are not necessarily blocks. Then:
  \begin{enumerate}[(i)]
  \item \label{item.rec} for $u\in[1,w]_R$, $u\in [1,\lcoset{w}{J}]_R$ if and only if
           $\Rec(u)\cap J = \emptyset$;
  \item \label{item.des} for $u\in[1,w]_L$, $u\in [1,w^K]_L$ if and only if
           $\Des(u)\cap K = \emptyset$.
  \end{enumerate}
\end{lemma}

\begin{proof}
  This is a straightforward corollary of Proposition~\ref{proposition.tiling}: any element
  $u$ of $[1,w]_R$ can be written uniquely as a product $u'v$ with
  $u'\in W_J$ and $v\in [1,\lcoset{w}{J}]_R$. So $u$ is in
  $[1,\lcoset{w}{J}]_R$ if and only if $u'=1$, which in turn is
  equivalent to $v$ having no descents in $J$. This
  proves~\eqref{item.rec}. The argument for~\eqref{item.des} is
  analogous.
\end{proof}

\begin{example}
  For $w=w_0$, $\lcoset{w}{J}$ is the maximal element of a left
  descent class, and $[1,\lcoset{w}{J}]_R$ gives all elements of $W$
  whose left descent set is a subset of the left descent set of $w$.
\end{example}

\begin{definition}[$w$-nondescent sets]
  \label{definition.max_block}
  For $u\in [1,w]_R$ define $\Jblock{w}(u)$ to be the index
  $J\in\JJs{w}$ of the lowest cutting point $\lcoset{w}{J}$ such that
  $u\in [1,\lcoset{w}{J}]_R$ (or the equivalent condition of
  Lemma~\ref{lemma.condition_interval}). Define similarly
  $\Kblock{w}(u)$ as the index in $\KKs{w}$ of this cutting point.
\end{definition}

\begin{example}
  When $w=w_0$, $\Jblock{w_0}(u)$ and $\Kblock{w_0}(u)$ are
  respectively the sets $\cRec(u)$ and $\cDes(u)$ of left and right
  nondescents of $u$.
\end{example}

\begin{problem}
  Given $J$, describe all the elements $w\in W$ such that $J$ is a
  left block. This essentially only depends on $\lcoset{w}{J}$.
\end{problem}

\subsection{Properties of the cutting poset}
\label{ss.properties.cutting_poset}

In this section we study the properties of the cutting poset $(W,\cuteq)$
of Theorem~\ref{theorem.cutting_poset} for the cutting
relation $\cuteq$ introduced in Definition~\ref{definition.block} (see
also Figure~\ref{figure.cutting_poset}). The following theorem
summarizes the results.

\begin{theorem}
  \label{theorem.cutting}
  $(W,\cuteq)$ is a meet-distributive meet-semilattice with $\one$ as
  minimal element, and a subposet of both left and right order.

  Every interval of $(W,\cuteq)$ is a distributive sublattice and a
  sublattice of both left and right order.

  Let $w\in W$ and denote by $\pred{w}$ the set of its $\cuteq$-lower
  covers. Thanks to meet-distributivity, the meet-semilattice $L_w$
  generated by $\pred{w}$ using $\meet_\cuteq$ (or equivalently
  $\meet_L$, $\meet_R$ if viewed as a sublattice of left or right
  order) is free, that is isomorphic to a Boolean lattice.

  In particular, the M\"obius function of $(W,\cuteq)$ is given by
  $\mu(u,w)=(-1)^{r(u,w)}$ if $u\in L_w$ and $0$ otherwise, where
  $r(u,w):=|\{v\in\pred{w}\mid u\cuteq v\}|$.
\end{theorem}

This M\"obius function is used in Section~\ref{subsection.bihecke} to
compute the size of the simple modules of $\K \biheckemonoid$.

Since $(W,\cuteq)$ is almost a distributive lattice, Birkhoff's
representation theorem suggests to embed it in the distributive
lattice $\lowersets(\irred((W,\cuteq)))$ of the lower sets of its
join-irreducible elements (note that a block is join-irreducible if
there is only one minimal nontrivial block below it).
\begin{problem}
  Describe the set $\irred(W,\cuteq)$ of join-irreducible elements of
  $(W,\cuteq)$.
\end{problem}
\begin{problem}
  Determine the distributive lattice associated with the cutting poset
  from the join-irreducibles, via Birkhoff's theory.
\end{problem}

The join-irreducible elements of $(\sg,\cuteq)$, for $n$ small, are
counted by the sequence $0,1,4,16,78,462,3224$.
Figure~\ref{figure.cutting_poset} seems to suggest that they form a
tree, but this already fails for $n=5$. We now briefly comment on the
simplest join-irreducible elements, namely the immediate successors
$w$ of $\one$ in the cutting poset. Equivalent statements are that $w$
admits exactly two reduced blocks $\{\}$ and $B$, possibly with $B=I$,
or that the simple module $S_w$ is of dimension $|[1,w]_R|-1$.  For a
Coxeter group $W$, we denote by $S(W)$ the set of elements $w\ne \one$
having no proper reduced blocks, and $T(W)$ those having exactly two
reduced blocks. Note that $T(W)$ is the disjoint union of the $S(W_J)$
for $J\subseteq I$.
\begin{example}
  \label{example.cutting_poset.almost_minimal}
  In type $A$, a permutation $w\in S(\sg)$ is uniquely obtained by
  taking a simple permutation, and inflating each $\one$ of its
  permutation matrix by an identity matrix. An element of $T(\sg)$
  has a block diagonal matrix with one block in $S(\sg[m])$ for $m\leq
  n$, and $n-m$ $1\times 1$ blocks. This gives an easy way to
  construct the generating series for $S(\sg)_{n\in \N}$ and for
  $T(\sg)_{n\in \N}$ from that of the simple permutations given
  in~\cite{Albert_Atkinson.2005}.
\end{example}

We now turn to the proof of Theorem~\ref{theorem.cutting}, starting
with some preliminary results.

\begin{lemma}
  \label{lemma.cutting.join}
  $(W,\cuteq)$ is a partial join-semilattice. Namely, when the join
  exists, it is unique and given by the join in left and in right order:
  \begin{displaymath}
    v \join_\cuteq v' = v\join_L v' = v\join_R v'\,.
  \end{displaymath}
\end{lemma}
\begin{proof}
  Take $v$ and $v'$ with at least one common successor. Applying
  Corollary~\ref{corollary.interval} to the interval $[1,w]_\cuteq$ for any
  such common successor $w$, one obtains $v,v'\cuteq v\join_R v'=v\join_L v'
  \cuteq w$. Therefore, $v\join_R v'=v\join_L v'$ is the join of $v$
  and $v'$ in the cutting order.
\end{proof}

\begin{lemma}
  \label{lemma.cutting.meet}
  $(W,\cuteq)$ is a meet-semilattice. Namely, for $v, v'\in W$
  \begin{displaymath}
    v \meet_\cuteq v' = \bigjoin_{u\cuteq v,v'} u\,,
  \end{displaymath}
  where $\bigjoin$ is the join for the cutting order (or equivalently
  for left or right order). If further $v$ and $v'$ have a common
  successor, then
  \begin{displaymath}
    v \meet_\cuteq v' = v \meet_R v' = v\meet_L v'\,.
  \end{displaymath}
\end{lemma}
\begin{proof}
  The first part follows from a general result. Namely, for any poset,
  the following statements are equivalent (see
  e.g. \cite[Proposition~7.3]{Pouzet.Ordre}):
  \begin{enumerate}[(i)]
  \item Any bounded nonempty part has an upper bound.
  \item Any bounded nonempty part has a lower bound.
  \end{enumerate}
  Here we reprove this fact for the sake of self-containment. Take $u$
  and $u'$ two common cutting points for $v$ and $v'$. Then, using
  Lemma~\ref{lemma.cutting.join}, their join exists and $u
  \join_\cuteq u' = u \join_R u'=u\join_L u'$ is also a cutting point
  for $v$ and $v'$. The first statement follows by repeated iteration
  over all common cutting points.

  Now assume that $v$ and $v'$ have a common successor $w$. Then
  applying Corollary~\ref{corollary.interval} to the interval $[1,w]_\cuteq$,
  we find that $v\meet_R v'=v\meet_L v'$ is the meet of $v$ and $v'$
  in the cutting order.
\end{proof}

\begin{proof}[Proof of Theorem~\ref{theorem.cutting}]
  $(W,\cuteq)$ is a meet-semilattice by
  Lemma~\ref{lemma.cutting.meet}. Meet-distributivity follows from
  Corollary~\ref{corollary.hypercube}. The
  argument is in fact general: any poset with a minimal element $1$
  such that all intervals $[1,x]$ are distributive lattices and such
  that any two elements admit either a join or no common successor is
  a meet-distributive meet-semilattice (see~\cite{Edelman.1986} for
  literature on such). The end of the first statement is
  Theorem~\ref{theorem.cutting_poset}.

  The statement about intervals is Corollary~\ref{corollary.interval}.

  The $\cuteq$-lower covers of an element $w$ correspond to the
  nontrivial blocks of $w$ which are minimal for inclusion.
  The top part $L_w$ of an interval $[1,w]_\cuteq$ is further described in
  Corollary~\ref{corollary.hypercube}, through the bijection $\phi^{(w)}$
  between blocks of $w$ and the interval $[1,w]_\cuteq$ of
  Proposition~\ref{proposition.cutting_points.sublattice}. The value of
  $\mu(u,w)$ depends only on this interval, and we conclude the remaining
  statements using Rota's Crosscut Theorem~\cite{Rota.1964} on M\"obius
  functions for lattices (see also~\cite[Theorem 1.3]{Blass_Sagan.1997}).
\end{proof}

\section{Combinatorics of $\biheckemonoid(W)$}
\label{section.M.combinatorics}

In this section we study the combinatorics of the biHecke monoid
$\biheckemonoid(W)$ of a finite Coxeter group $W$. In particular, we
prove in Sections~\ref{ss.left} and~\ref{ss.bruhat} that its elements
preserve left order and Bruhat order, and derive in
Section~\ref{ss.fiber_image} properties of their image sets and
fibers. In Sections~\ref{ss.aperiodic} and~\ref{ss.idempotents},
we prove the key combinatorial ingredients
for the enumeration of the simple modules of $\K \biheckemonoid(W)$ in
Section~\ref{section.M.representation_theory}: $\biheckemonoid(W)$ is
aperiodic and its $\JJ$-classes of idempotents are indexed by $W$.
Finally, in Section~\ref{ss.green_relations} we study Green's relations as introduced in
Section~\ref{ss.preorders on monoids} and involutions on $\biheckemonoid(W)$ in
Section~\ref{subsection.involution}.

\subsection{Preservation of left order}
\label{ss.left}

Recall that $\biheckemonoid(W)$ is defined by its right action on
elements in $W$ by~\eqref{equation.antisorting_action}
and~\eqref{equation.sorting_action}. The following key proposition,
illustrated in Figure~\ref{figure.preservation_weak_order}, states
that it therefore preserves properties on the left.

\begin{proposition}
  \label{proposition.weak_order}
  Take $f\in \biheckemonoid(W)$, $w\in W$, and $j\in I$. Then, $(s_j
  w).f$ is either $w.f$ or $s_j(w.f)$.
\end{proposition}

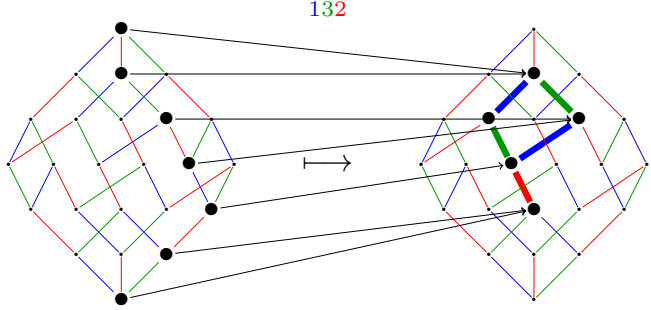
\begin{figure}
  \centerline{\scalebox{1.2}{{\def\fatedge{0pt}
$\stackrel{{\color{blue}1}{\color{green}3}{\color{red}\overline{2}}}{
\begin{tikzpicture}[xscale=.5,yscale=.5,inner sep=0pt,baseline=(current bounding box.east), remember picture]
  \node (1234) at (0, 0) {$\bullet$};
  \node (2134) at (-1, 1) {.};
  \node (2314) at (-1, 2) {.};
  \node (3214) at (-1.50000000000000, 3) {.};
  \node (2341) at (1.50000000000000, 3) {$\bullet$};
  \node (3241) at (1, 4) {$\bullet$};
  \node (3421) at (1, 5) {.};
  \node (4321) at (0, 6) {$\bullet$};
  \node (2431) at (2, 4) {.};
  \node (4231) at (0, 5) {$\bullet$};
  \node (1324) at (0, 1) {.};
  \node (3124) at (-2, 2) {.};
  \node (1342) at (2, 2) {$\bullet$};
  \node (3142) at (-0.500000000000000, 3) {.};
  \node (3412) at (0, 4) {.};
  \node (4312) at (-1, 5) {.};
  \node (1432) at (2.50000000000000, 3) {.};
  \node (4132) at (-1, 4) {.};
  \node (1243) at (1, 1) {$\bullet$};
  \node (2143) at (0, 2) {.};
  \node (2413) at (0.500000000000000, 3) {.};
  \node (4213) at (-2, 4) {.};
  \node (1423) at (1, 2) {.};
  \node (4123) at (-2.50000000000000, 3) {.};
  \draw [color=blue,line width=0pt] (2134) -- (1234);
  \draw [color=blue,line width=\fatedge] (2314) -- (1324);
  \draw [color=blue,line width=\fatedge] (3214) -- (3124);
  \draw [color=red,line width=0pt] (3214) -- (2314);
  \draw [color=blue,line width=\fatedge] (2341) -- (1342);
  \draw [color=blue,line width=\fatedge] (3241) -- (3142);
  \draw [color=red,line width=0pt] (3241) -- (2341);
  \draw [color=blue,line width=0pt] (3421) -- (3412);
  \draw [color=red,line width=0pt] (3421) -- (2431);
  \draw [color=blue,line width=0pt] (4321) -- (4312);
  \draw [color=red,line width=0pt] (4321) -- (4231);
  \draw [color=green,line width=0pt] (4321) -- (3421);
  \draw [color=blue,line width=\fatedge] (2431) -- (1432);
  \draw [color=green,line width=\fatedge] (2431) -- (2341);
  \draw [color=blue,line width=\fatedge] (4231) -- (4132);
  \draw [color=green,line width=\fatedge] (4231) -- (3241);
  \draw [color=red,line width=\fatedge] (1324) -- (1234);
  \draw [color=red,line width=\fatedge] (3124) -- (2134);
  \draw [color=red,line width=\fatedge] (1342) -- (1243);
  \draw [color=red,line width=\fatedge] (3142) -- (2143);
  \draw [color=red,line width=0pt] (3412) -- (2413);
  \draw [color=red,line width=0pt] (4312) -- (4213);
  \draw [color=green,line width=0pt] (4312) -- (3412);
  \draw [color=red,line width=0pt] (1432) -- (1423);
  \draw [color=green,line width=\fatedge] (1432) -- (1342);
  \draw [color=red,line width=0pt] (4132) -- (4123);
  \draw [color=green,line width=\fatedge] (4132) -- (3142);
  \draw [color=green,line width=0pt] (1243) -- (1234);
  \draw [color=blue,line width=0pt] (2143) -- (1243);
  \draw [color=green,line width=0pt] (2143) -- (2134);
  \draw [color=blue,line width=\fatedge] (2413) -- (1423);
  \draw [color=green,line width=\fatedge] (2413) -- (2314);
  \draw [color=blue,line width=\fatedge] (4213) -- (4123);
  \draw [color=green,line width=\fatedge] (4213) -- (3214);
  \draw [color=green,line width=\fatedge] (1423) -- (1324);
  \draw [color=green,line width=\fatedge] (4123) -- (3124);
\end{tikzpicture}
 \raisebox{-.5ex}{$\qquad \longmapsto\qquad $} 
\begin{tikzpicture}[xscale=.5,yscale=.5,inner
sep=0pt,baseline=(current bounding box.east), remember picture]
  \node (image1234) at (0, 0) {.};
  \node (image2134) at (-1, 1) {.};
  \node (image2314) at (-1, 2) {.};
  \node (image3214) at (-1.50000000000000, 3) {.};
  \node (image2341) at (1.50000000000000, 3) {.};
  \node (image3241) at (1, 4) {$\bullet$};
  \node (image3421) at (1, 5) {.};
  \node (image4321) at (0, 6) {.};
  \node (image2431) at (2, 4) {.};
  \node (image4231) at (0, 5) {$\bullet$};
  \node (image1324) at (0, 1) {.};
  \node (image3124) at (-2, 2) {.};
  \node (image1342) at (2, 2) {.};
  \node (image3142) at (-0.500000000000000, 3) {$\bullet$};
  \node (image3412) at (0, 4) {.};
  \node (image4312) at (-1, 5) {.};
  \node (image1432) at (2.50000000000000, 3) {.};
  \node (image4132) at (-1, 4) {$\bullet$};
  \node (image1243) at (1, 1) {.};
  \node (image2143) at (0, 2) {$\bullet$};
  \node (image2413) at (0.500000000000000, 3) {.};
  \node (image4213) at (-2, 4) {.};
  \node (image1423) at (1, 2) {.};
  \node (image4123) at (-2.50000000000000, 3) {.};
  \draw [color=blue,line width=0pt] (image2134) -- (image1234);
  \draw [color=blue,line width=0pt] (image2314) -- (image1324);
  \draw [color=blue,line width=0pt] (image3214) -- (image3124);
  \draw [color=red,line width=0pt] (image3214) -- (image2314);
  \draw [color=blue,line width=0pt] (image2341) -- (image1342);
  \draw [color=blue,line width=2pt] (image3241) -- (image3142);
  \draw [color=red,line width=0pt] (image3241) -- (image2341);
  \draw [color=blue,line width=0pt] (image3421) -- (image3412);
  \draw [color=red,line width=0pt] (image3421) -- (image2431);
  \draw [color=blue,line width=0pt] (image4321) -- (image4312);
  \draw [color=red,line width=0pt] (image4321) -- (image4231);
  \draw [color=green,line width=0pt] (image4321) -- (image3421);
  \draw [color=blue,line width=0pt] (image2431) -- (image1432);
  \draw [color=green,line width=0pt] (image2431) -- (image2341);
  \draw [color=blue,line width=2pt] (image4231) -- (image4132);
  \draw [color=green,line width=2pt] (image4231) -- (image3241);
  \draw [color=red,line width=0pt] (image1324) -- (image1234);
  \draw [color=red,line width=0pt] (image3124) -- (image2134);
  \draw [color=red,line width=0pt] (image1342) -- (image1243);
  \draw [color=red,line width=2pt] (image3142) -- (image2143);
  \draw [color=red,line width=0pt] (image3412) -- (image2413);
  \draw [color=red,line width=0pt] (image4312) -- (image4213);
  \draw [color=green,line width=0pt] (image4312) -- (image3412);
  \draw [color=red,line width=0pt] (image1432) -- (image1423);
  \draw [color=green,line width=0pt] (image1432) -- (image1342);
  \draw [color=red,line width=0pt] (image4132) -- (image4123);
  \draw [color=green,line width=2pt] (image4132) -- (image3142);
  \draw [color=green,line width=0pt] (image1243) -- (image1234);
  \draw [color=blue,line width=0pt] (image2143) -- (image1243);
  \draw [color=green,line width=0pt] (image2143) -- (image2134);
  \draw [color=blue,line width=0pt] (image2413) -- (image1423);
  \draw [color=green,line width=0pt] (image2413) -- (image2314);
  \draw [color=blue,line width=0pt] (image4213) -- (image4123);
  \draw [color=green,line width=0pt] (image4213) -- (image3214);
  \draw [color=green,line width=0pt] (image1423) -- (image1324);
  \draw [color=green,line width=0pt] (image4123) -- (image3124);
\end{tikzpicture}
\begin{tikzpicture}[xscale=.5,yscale=-.5,inner sep=0pt,baseline=(current bounding box.east), remember picture, overlay]
  \draw [color=black, line width = 0pt,->] (1234) -- (image2143);
  \draw [color=black, line width = 0pt,->] (1243) -- (image2143);
  \draw [color=black, line width = 0pt,->] (1342) -- (image3142);
  \draw [color=black, line width = 0pt,->] (2341) -- (image3241);
  \draw [color=black, line width = 0pt,->] (3241) -- (image3241);
  \draw [color=black, line width = 0pt,->] (4231) -- (image4231);
  \draw [color=black, line width = 0pt,->] (4321) -- (image4231);
\end{tikzpicture}
}$}}}
  \caption{A partial picture of the graph of the element
    $f:=\pi_1\pi_3\opi_2$ of the monoid $\biheckemonoid(\sg[4])$. On both sides, the
    underlying poset is left order of $\sg[4]$ (with $\one$ at the
    bottom, and the same color code as in Figure~\ref{figure.4312});
    on the right, the bold dots depict the image set of $f$. The arrows
    from the left to the right describe the image of each point along
    some chain from $\one$ to $w_0$.}
  \label{figure.preservation_weak_order}
\end{figure}

The proof of Proposition~\ref{proposition.weak_order} is a consequence
of the associativity of the $0$-Hecke monoid and relies on the following
lemma, which is a nice algebraic (partial) formulation of the Exchange
Property~\cite[Section 1.5]{Bjorner_Brenti.2005}.

\begin{lemma}
  \label{lemma.weak_order}
  Let $w\in W$ and $i,j\in I$ such that $j \not \in \Rec(w)$. Then
  \begin{equation*}
    (s_jw).\pi_i =
    \begin{cases}
      w.\pi_i       & \text{if $j \in \Rec(w.\pi_i)$,}\\
      s_j (w.\pi_i) & \text{otherwise.}
    \end{cases}
  \end{equation*}
  The same result holds with $\pi_i$ replaced by $\opi_i$.
\end{lemma}

\begin{proof}
  Recall that for any $w,v\in W$, $w.\pi_v=\one.(\pi_w\pi_v)$.
  Set $w'=w.\pi_i$. Then
  \begin{multline*}
    (s_j w).\pi_i = 1 . (\pi_{s_j w} \pi_i) = 1 . ((\pi_j\pi_w) \pi_i)
    = \one. (\pi_j (\pi_{w}\pi_{i})) = \one. (\pi_j \pi_{w'})\\
   = \begin{cases}
      \one. \pi_{w'}     =    w' & \text{if $j \in \Rec(w')$,}\\
      \one. \pi_{s_j w'} = s_jw' & \text{otherwise.}
    \end{cases}
  \end{multline*}
  The result for $\opi_i$ follows from Remark~\ref{remark.opi_in_terms_of_pi} and the
  fact that $w_0s_j = s_{j'}w_0$ for some $j'\in I$ by Example~\ref{example.reduced_blocks_w0}
  and Lemma~\ref{lemma.skew_commutation} with $w=w_0$ and $K=\{j\}$.
\end{proof}

\begin{proof}[Proof of Proposition~\ref{proposition.weak_order}]
  Any element $f\in \biheckemonoid(W)$ can be written as a product of $\pi_i$ and
  $\opi_i$. Lemma~\ref{lemma.weak_order} describes the action of $\pi_i$ and $\opi_i$
  on the Hasse diagram of left order. By applying induction, each $\pi_i$ and $\opi_i$ in the
  expansion of $f$ satisfies all desired properties, and hence so does $f$
  (the statement holds trivially for the identity).
\end{proof}

\begin{proposition}
  \label{proposition.weak_order2}
  For $f \in \biheckemonoid(W)$, the following holds:
  \begin{enumerate}[(i)]
  \item \label{item.weak_order}
  $f$ preserves left order:
    \begin{equation*}
     	w \le_L w' \quad \Rightarrow \quad w.f \le_L w'.f \quad \text{for $w,w'\in W$.}
    \end{equation*}
  \item
    \label{item.weak_order.maxchain}
    Take $w\le_L w'$ in $W$, and consider a maximal chain
    \begin{displaymath}
      w.f=v_1\stackrel{i_1}{\rightarrow}
      v_2\stackrel{i_2}{\rightarrow} \cdots
      \stackrel{i_{k-1}}{\rightarrow} v_k=w'.f\,.
    \end{displaymath}
    Then, there is a maximal chain:
    \begin{multline}
      w=
      u_{1,1}\rightarrow \dots \rightarrow u_{1,\ell_1}
      \ \stackrel{i_1}{\rightarrow}\
      u_{2,1}\rightarrow \dots \rightarrow u_{2,\ell_2}
      \ \stackrel{i_2}{\rightarrow}\
      \quad \dots\\
      \dots\quad
      \ \stackrel{i_{k-1}}{\rightarrow}\
      u_{k,1}\rightarrow \dots \rightarrow u_{k,\ell_k} = w'\,,
    \end{multline}
    such that $u_{j,l}.f = v_j$ for all $1\leq j\leq k$ and $1\leq l\leq \ell_j$.
  \item
    \label{item.length_contracting}
    $f$ is length contracting; namely, for $w\le_Lw'$:
    \begin{equation*}
     	\len(w'.f) - \len(w.f) \le \len(w') - \len(w).
    \end{equation*}
    Furthermore, when equality holds, $(w'.f) (w.f)^{-1} = w' w^{-1}$.
  \item
    \label{item.image_set}
    Let $J=[a,b]_L$ be an interval in left order. Then the image of $J$ under $f$ denoted by $J.f$
    has $a.f$ and $b.f$ as minimal and maximal element, respectively. Furthermore, $J.f$
    is connected. If $\len(b.f) - \len(a.f) = \len(b) - \len(a)$,
    then $J.f$ is isomorphic to $J$, that is $x.f = (xa^{-1})(a.f)$
    for $x\in J$.
  \end{enumerate}
\end{proposition}

\begin{proof}
  \eqref{item.weak_order} and~\eqref{item.weak_order.maxchain} are direct consequences of
  Proposition~\ref{proposition.weak_order}, using induction.

  \eqref{item.length_contracting} follows from~\eqref{item.weak_order.maxchain}.

  \eqref{item.image_set} follows from~\eqref{item.weak_order}, \eqref{item.weak_order.maxchain},
  and~\eqref{item.length_contracting} applied to $a\le_L x$ for all $x\in [a,b]_L$.
\end{proof}

\subsection{Preservation of Bruhat order}
\label{ss.bruhat}

Recall the following well-known property of Bruhat order of Coxeter groups.
\begin{proposition}[{Lifting Property~\cite[p.35]{Bjorner_Brenti.2005}}]
  \label{proposition.lifting}
  Suppose $u<_B v$ and $i\in \Des(v)$ but $i\not\in \Des(u)$. Then,
  $u\leq_B vs_i$ and $us_i\leq_B v$.
\end{proposition}
The next proposition is a consequence of the Lifting Property.

\begin{proposition}
  \label{proposition.bruhat}
  The elements $f$ of $\biheckemonoid(W)$ preserve Bruhat order. That is for
  $u,v\in W$
  \begin{equation*}
  	u\leq_B v \quad \Longrightarrow \quad u.f\leq_B v.f.
  \end{equation*}
\end{proposition}
\begin{proof}
  It suffices to show the property for $\pi_i$ and $\opi_i$ since they generate $\biheckemonoid(W)$.
  For these, the claim of the proposition is trivial if $i$ is a right descent of $u$, or $i$ is not a right
  descent of $v$. Otherwise, we can apply the Lifting Property:
  \begin{gather*}
    u.\pi_i = us_i \leq_B v = v.\pi_i,\\
    u.\opi_i = u \leq_B vs_i = v.\opi_i .\qedhere
  \end{gather*}
\end{proof}

\begin{remark}
  By Lemma~\ref{lemma.preimage_convex}, the preimage of a point is a convex set,
  but need not be an interval.  For example, the preimage of $s_1s_3\in \sg[4]$ (or $2143$ in
  one-line notation) of $f= \opi_1 \pi_2 \pi_1 \pi_3 \opi_2 \opi_3 \opi_1 \opi_2$ is
  \begin{equation*}
  \{ 2413, 2341, 4213, 3412, 3241, 2431, 4312, 4231, 3421, 4321\},
  \end{equation*}
  which in Bruhat order has two maximal elements $2413$ and $2341$ and hence is
  not an interval.
\end{remark}

The next result is a corollary of Proposition~\ref{proposition.weak_order2}.
\begin{corollary}
  \label{corollary.Bruhat_regressive}
  Let $f \in \biheckemonoid(W)$.
  \begin{enumerate}[(i)]
  \item \label{item.regressive}
  If $\one.f=\one$, then $f$ is \regressive for Bruhat order: $w.f\le_B w$ for all $w\in W$.
  \item \label{item.extensive}
  If $w_0.f=w_0$, then $f$ is \extensive for Bruhat order: $w.f\ge_B w$ for all $w\in W$.
  \end{enumerate}
\end{corollary}
\begin{proof}
  First suppose that $\one.f=\one$. Let $w.f=s_{i_k}\cdots s_{i_1}$ be a reduced decomposition
  of $w.f$. This defines a maximal chain
  \begin{equation*}
    \one.f=\one =
    v_0     \stackrel{i_1}{\rightarrow} \cdots \stackrel{i_{k-2}}{\rightarrow}
    v_{k-2} \stackrel{i_{k-1}}{\rightarrow}
    v_{k-1}   \stackrel{i_k}{\rightarrow} v_{k}=w.f
  \end{equation*}
  in left order. By
  Proposition~\ref{proposition.weak_order2}~\eqref{item.weak_order.maxchain}
  there is a larger chain from $1$ to $w$ so that there is a reduced word
  for $w$ which contains $s_{i_k}\cdots s_{i_1}$ as a subword. Hence by the subword
  property of Bruhat order $w.f \le_B w$. This proves~\eqref{item.regressive}.

  Now let $w_0.f=w_0$. By similar arguments as above, constructing a maximal chain from $w.f$
  to $w_0.f$ in left order, one finds that $w_0(w.f)^{-1} \le_B w_0 w^{-1}$.
  By~\cite[Proposition 2.3.4]{Bjorner_Brenti.2005}, the map $v\mapsto w_0v$ is a
  Bruhat antiautomorphism and by the subword property $v\mapsto v^{-1}$ is a Bruhat
  automorphism. This implies $w\le_B w.f$ as desired for~\eqref{item.extensive}.
\end{proof}

\subsection{Fibers and image sets}
\label{ss.fiber_image}

Viewing elements of the biHecke monoid $\biheckemonoid(W)$ as
functions on $W$, we now study properties of their fibers and image
sets.

\begin{proposition}
  \label{proposition.image_set.idempotent}
  \mbox{}
  \begin{enumerate}[(i)]
  \item \label{item.connected_image}
  The image set $\im(f)$ for any $f\in \biheckemonoid(W)$ is connected
  (see Definition~\ref{definition.convex_connected}) with a unique minimal and
  maximal element in left order.
  \item \label{item.image_set.idempotent}
  The image set of an idempotent in $\biheckemonoid(W)$ is an interval in left order.
  \end{enumerate}
\end{proposition}
\begin{proof}
The first statement follows immediately from
Proposition~\ref{proposition.weak_order2}~\eqref{item.image_set} with $J=[1,w_0]_L$.

For the second statement, let $e\in \biheckemonoid(W)$ be an idempotent with image set $\im(e)$. By
Proposition~\ref{proposition.weak_order2}~\eqref{item.image_set}
with $J=[1,w_0]_L$, we have that $1.e$ (resp. $w_0.e$) is the minimal (resp. maximal) element
of $\im(e)$. Then by Proposition~\ref{proposition.weak_order2}~\eqref{item.weak_order.maxchain},
for every maximal chain in left order between $1.e$ and $w_0.e$, there is a maximal chain in left
order of preimage points. Since $e$ is an idempotent, there must be such a chain which contains
the original chain. Hence all chains in left order between $1.e$ and $w_0.e$ are in $\im(e)$,
proving that $\im(e)$ is an interval.
\end{proof}

Note that the above proof, in particular
Proposition~\ref{proposition.weak_order2}~\eqref{item.weak_order.maxchain}, heavily uses
the fact that the edges in left order are colored.

\begin{definition}
  For any $f\in \biheckemonoid(W)$, we call the set of fibers of $f$, denoted by
  $\fibers(f)$, the (unordered) set-partition of $W$ associated by the
  equivalence relation $w \equiv w'$ if $w.f=w'.f$.
\end{definition}

\begin{proposition}
  \label{proposition.contract_fibers}
  Take $f\in \biheckemonoid(W)$, and consider the Hasse diagram of
  left order contracted with respect to the fibers of $f$. Then, this
  graph is isomorphic to left order restricted on the image set.
\end{proposition}
\begin{proof}
  See Appendix~\ref{appendix.colored_graphs} on colored graphs.
\end{proof}

\begin{proposition}
  \label{proposition.fibers_image_set}
  Any element $f\in \biheckemonoid(W)$ is characterized by its set of
  fibers and $1.f$.
\end{proposition}
\begin{proof}
  Fix a choice of fibers. Contract the left
  order with respect to the fibers. By Proposition~\ref{proposition.contract_fibers}
  this graph has to be isomorphic to the left order on the image set.

  Once the lowest element in the image set $1.f$ is fixed, this isomorphism is forced,
  since by Proposition~\ref{proposition.image_set.idempotent}~\eqref{item.connected_image}
  the graphs are (weakly) connected, have a unique minimal element, and there is at most one
  arrow of a given color leaving each node.
\end{proof}

Proposition~\ref{proposition.fibers_image_set} makes it possible to
visualize nontrivial elements of the monoid (see Figure~\ref{figure.some_elements}).
\begin{figure}
  \begin{bigcenter}
    \scalebox{1.2}{\fbox{$\stackrel{{\color{blue}1}}{
\begin{tikzpicture}[xscale=.5,yscale=.5,inner sep=0pt,baseline=(current bounding box.east)]
  \node (1234) at (0, 0) {.};
  \node (2134) at (-1, 1) {.};
  \node (2314) at (-1, 2) {.};
  \node (3214) at (-1.50000000000000, 3) {.};
  \node (2341) at (1.50000000000000, 3) {.};
  \node (3241) at (1, 4) {.};
  \node (3421) at (1, 5) {.};
  \node (4321) at (0, 6) {.};
  \node (2431) at (2, 4) {.};
  \node (4231) at (0, 5) {.};
  \node (1324) at (0, 1) {.};
  \node (3124) at (-2, 2) {.};
  \node (1342) at (2, 2) {.};
  \node (3142) at (-0.500000000000000, 3) {.};
  \node (3412) at (0, 4) {.};
  \node (4312) at (-1, 5) {.};
  \node (1432) at (2.50000000000000, 3) {.};
  \node (4132) at (-1, 4) {.};
  \node (1243) at (1, 1) {.};
  \node (2143) at (0, 2) {.};
  \node (2413) at (0.500000000000000, 3) {.};
  \node (4213) at (-2, 4) {.};
  \node (1423) at (1, 2) {.};
  \node (4123) at (-2.50000000000000, 3) {.};
  \draw [color=blue,line width=0pt] (2134) -- (1234);
  \draw [color=blue,line width=2pt] (2314) -- (1324);
  \draw [color=blue,line width=2pt] (3214) -- (3124);
  \draw [color=red,line width=0pt] (3214) -- (2314);
  \draw [color=blue,line width=2pt] (2341) -- (1342);
  \draw [color=blue,line width=2pt] (3241) -- (3142);
  \draw [color=red,line width=0pt] (3241) -- (2341);
  \draw [color=blue,line width=2pt] (3421) -- (3412);
  \draw [color=red,line width=2pt] (3421) -- (2431);
  \draw [color=blue,line width=2pt] (4321) -- (4312);
  \draw [color=red,line width=2pt] (4321) -- (4231);
  \draw [color=green,line width=0pt] (4321) -- (3421);
  \draw [color=blue,line width=2pt] (2431) -- (1432);
  \draw [color=green,line width=2pt] (2431) -- (2341);
  \draw [color=blue,line width=2pt] (4231) -- (4132);
  \draw [color=green,line width=2pt] (4231) -- (3241);
  \draw [color=red,line width=2pt] (1324) -- (1234);
  \draw [color=red,line width=2pt] (3124) -- (2134);
  \draw [color=red,line width=2pt] (1342) -- (1243);
  \draw [color=red,line width=2pt] (3142) -- (2143);
  \draw [color=red,line width=2pt] (3412) -- (2413);
  \draw [color=red,line width=2pt] (4312) -- (4213);
  \draw [color=green,line width=0pt] (4312) -- (3412);
  \draw [color=red,line width=2pt] (1432) -- (1423);
  \draw [color=green,line width=2pt] (1432) -- (1342);
  \draw [color=red,line width=2pt] (4132) -- (4123);
  \draw [color=green,line width=2pt] (4132) -- (3142);
  \draw [color=green,line width=2pt] (1243) -- (1234);
  \draw [color=blue,line width=0pt] (2143) -- (1243);
  \draw [color=green,line width=2pt] (2143) -- (2134);
  \draw [color=blue,line width=2pt] (2413) -- (1423);
  \draw [color=green,line width=2pt] (2413) -- (2314);
  \draw [color=blue,line width=2pt] (4213) -- (4123);
  \draw [color=green,line width=2pt] (4213) -- (3214);
  \draw [color=green,line width=2pt] (1423) -- (1324);
  \draw [color=green,line width=2pt] (4123) -- (3124);
\end{tikzpicture}
 \raisebox{-.5ex}{$\ \mapsto\ $} 
\begin{tikzpicture}[xscale=.5,yscale=.5,inner sep=0pt,baseline=(current bounding box.east)]
  \node (1234) at (0, 0) {.};
  \node (2134) at (-1, 1) {$\bullet$};
  \node (2314) at (-1, 2) {.};
  \node (3214) at (-1.50000000000000, 3) {$\bullet$};
  \node (2341) at (1.50000000000000, 3) {.};
  \node (3241) at (1, 4) {$\bullet$};
  \node (3421) at (1, 5) {.};
  \node (4321) at (0, 6) {$\bullet$};
  \node (2431) at (2, 4) {.};
  \node (4231) at (0, 5) {$\bullet$};
  \node (1324) at (0, 1) {.};
  \node (3124) at (-2, 2) {$\bullet$};
  \node (1342) at (2, 2) {.};
  \node (3142) at (-0.500000000000000, 3) {$\bullet$};
  \node (3412) at (0, 4) {.};
  \node (4312) at (-1, 5) {$\bullet$};
  \node (1432) at (2.50000000000000, 3) {.};
  \node (4132) at (-1, 4) {$\bullet$};
  \node (1243) at (1, 1) {.};
  \node (2143) at (0, 2) {$\bullet$};
  \node (2413) at (0.500000000000000, 3) {.};
  \node (4213) at (-2, 4) {$\bullet$};
  \node (1423) at (1, 2) {.};
  \node (4123) at (-2.50000000000000, 3) {$\bullet$};
  \draw [color=blue,line width=0pt] (2134) -- (1234);
  \draw [color=blue,line width=0pt] (2314) -- (1324);
  \draw [color=blue,line width=2pt] (3214) -- (3124);
  \draw [color=red,line width=0pt] (3214) -- (2314);
  \draw [color=blue,line width=0pt] (2341) -- (1342);
  \draw [color=blue,line width=2pt] (3241) -- (3142);
  \draw [color=red,line width=0pt] (3241) -- (2341);
  \draw [color=blue,line width=0pt] (3421) -- (3412);
  \draw [color=red,line width=0pt] (3421) -- (2431);
  \draw [color=blue,line width=2pt] (4321) -- (4312);
  \draw [color=red,line width=2pt] (4321) -- (4231);
  \draw [color=green,line width=0pt] (4321) -- (3421);
  \draw [color=blue,line width=0pt] (2431) -- (1432);
  \draw [color=green,line width=0pt] (2431) -- (2341);
  \draw [color=blue,line width=2pt] (4231) -- (4132);
  \draw [color=green,line width=2pt] (4231) -- (3241);
  \draw [color=red,line width=0pt] (1324) -- (1234);
  \draw [color=red,line width=2pt] (3124) -- (2134);
  \draw [color=red,line width=0pt] (1342) -- (1243);
  \draw [color=red,line width=2pt] (3142) -- (2143);
  \draw [color=red,line width=0pt] (3412) -- (2413);
  \draw [color=red,line width=2pt] (4312) -- (4213);
  \draw [color=green,line width=0pt] (4312) -- (3412);
  \draw [color=red,line width=0pt] (1432) -- (1423);
  \draw [color=green,line width=0pt] (1432) -- (1342);
  \draw [color=red,line width=2pt] (4132) -- (4123);
  \draw [color=green,line width=2pt] (4132) -- (3142);
  \draw [color=green,line width=0pt] (1243) -- (1234);
  \draw [color=blue,line width=0pt] (2143) -- (1243);
  \draw [color=green,line width=2pt] (2143) -- (2134);
  \draw [color=blue,line width=0pt] (2413) -- (1423);
  \draw [color=green,line width=0pt] (2413) -- (2314);
  \draw [color=blue,line width=2pt] (4213) -- (4123);
  \draw [color=green,line width=2pt] (4213) -- (3214);
  \draw [color=green,line width=0pt] (1423) -- (1324);
  \draw [color=green,line width=2pt] (4123) -- (3124);
\end{tikzpicture}
}$}

}
    \scalebox{1.2}{\fbox{$\stackrel{{\color{red}2}}{
\begin{tikzpicture}[xscale=.5,yscale=.5,inner sep=0pt,baseline=(current bounding box.east)]
  \node (1234) at (0, 0) {.};
  \node (2134) at (-1, 1) {.};
  \node (2314) at (-1, 2) {.};
  \node (3214) at (-1.50000000000000, 3) {.};
  \node (2341) at (1.50000000000000, 3) {.};
  \node (3241) at (1, 4) {.};
  \node (3421) at (1, 5) {.};
  \node (4321) at (0, 6) {.};
  \node (2431) at (2, 4) {.};
  \node (4231) at (0, 5) {.};
  \node (1324) at (0, 1) {.};
  \node (3124) at (-2, 2) {.};
  \node (1342) at (2, 2) {.};
  \node (3142) at (-0.500000000000000, 3) {.};
  \node (3412) at (0, 4) {.};
  \node (4312) at (-1, 5) {.};
  \node (1432) at (2.50000000000000, 3) {.};
  \node (4132) at (-1, 4) {.};
  \node (1243) at (1, 1) {.};
  \node (2143) at (0, 2) {.};
  \node (2413) at (0.500000000000000, 3) {.};
  \node (4213) at (-2, 4) {.};
  \node (1423) at (1, 2) {.};
  \node (4123) at (-2.50000000000000, 3) {.};
  \draw [color=blue,line width=2pt] (2134) -- (1234);
  \draw [color=green,line width=2pt] (2143) -- (2134);
  \draw [color=blue,line width=2pt] (2314) -- (1324);
  \draw [color=blue,line width=0pt] (3214) -- (3124);
  \draw [color=red,line width=2pt] (3214) -- (2314);
  \draw [color=blue,line width=2pt] (2341) -- (1342);
  \draw [color=blue,line width=2pt] (3241) -- (3142);
  \draw [color=red,line width=2pt] (3241) -- (2341);
  \draw [color=red,line width=2pt] (3421) -- (2431);
  \draw [color=red,line width=0pt] (4321) -- (4231);
  \draw [color=green,line width=2pt] (4321) -- (3421);
  \draw [color=blue,line width=2pt] (2431) -- (1432);
  \draw [color=green,line width=0pt] (2431) -- (2341);
  \draw [color=blue,line width=2pt] (4231) -- (4132);
  \draw [color=green,line width=2pt] (4231) -- (3241);
  \draw [color=red,line width=0pt] (1324) -- (1234);
  \draw [color=red,line width=2pt] (3124) -- (2134);
  \draw [color=red,line width=2pt] (1342) -- (1243);
  \draw [color=red,line width=2pt] (3142) -- (2143);
  \draw [color=red,line width=2pt] (3412) -- (2413);
  \draw [color=red,line width=2pt] (4312) -- (4213);
  \draw [color=red,line width=2pt] (1432) -- (1423);
  \draw [color=green,line width=0pt] (1432) -- (1342);
  \draw [color=red,line width=2pt] (4132) -- (4123);
  \draw [color=green,line width=2pt] (4132) -- (3142);
  \draw [color=green,line width=2pt] (1243) -- (1234);
  \draw [color=blue,line width=2pt] (2143) -- (1243);
  \draw [color=blue,line width=2pt] (2413) -- (1423);
  \draw [color=green,line width=2pt] (2413) -- (2314);
  \draw [color=blue,line width=0pt] (4213) -- (4123);
  \draw [color=green,line width=2pt] (4213) -- (3214);
  \draw [color=green,line width=2pt] (1423) -- (1324);
  \draw [color=green,line width=2pt] (4312) -- (3412);
  \draw [color=blue,line width=2pt] (4321) -- (4312);
  \draw [color=blue,line width=2pt] (3421) -- (3412);
  \draw [color=green,line width=2pt] (4123) -- (3124);
\end{tikzpicture}
 \raisebox{-.5ex}{$\ \mapsto\ $} 
\begin{tikzpicture}[xscale=.5,yscale=.5,inner sep=0pt,baseline=(current bounding box.east)]
  \node (1234) at (0, 0) {.};
  \node (2134) at (-1, 1) {.};
  \node (2314) at (-1, 2) {$\bullet$};
  \node (3214) at (-1.50000000000000, 3) {$\bullet$};
  \node (2341) at (1.50000000000000, 3) {.};
  \node (3241) at (1, 4) {.};
  \node (3421) at (1, 5) {$\bullet$};
  \node (4321) at (0, 6) {$\bullet$};
  \node (2431) at (2, 4) {$\bullet$};
  \node (4231) at (0, 5) {.};
  \node (1324) at (0, 1) {$\bullet$};
  \node (3124) at (-2, 2) {.};
  \node (1342) at (2, 2) {.};
  \node (3142) at (-0.500000000000000, 3) {.};
  \node (3412) at (0, 4) {$\bullet$};
  \node (4312) at (-1, 5) {$\bullet$};
  \node (1432) at (2.50000000000000, 3) {$\bullet$};
  \node (4132) at (-1, 4) {.};
  \node (1243) at (1, 1) {.};
  \node (2143) at (0, 2) {.};
  \node (2413) at (0.500000000000000, 3) {$\bullet$};
  \node (4213) at (-2, 4) {$\bullet$};
  \node (1423) at (1, 2) {$\bullet$};
  \node (4123) at (-2.50000000000000, 3) {.};
  \draw [color=blue,line width=0pt] (2134) -- (1234);
  \draw [color=blue,line width=2pt] (2314) -- (1324);
  \draw [color=blue,line width=0pt] (3214) -- (3124);
  \draw [color=red,line width=2pt] (3214) -- (2314);
  \draw [color=blue,line width=0pt] (2341) -- (1342);
  \draw [color=blue,line width=0pt] (3241) -- (3142);
  \draw [color=red,line width=0pt] (3241) -- (2341);
  \draw [color=blue,line width=2pt] (3421) -- (3412);
  \draw [color=red,line width=2pt] (3421) -- (2431);
  \draw [color=blue,line width=2pt] (4321) -- (4312);
  \draw [color=red,line width=0pt] (4321) -- (4231);
  \draw [color=green,line width=2pt] (4321) -- (3421);
  \draw [color=blue,line width=2pt] (2431) -- (1432);
  \draw [color=green,line width=0pt] (2431) -- (2341);
  \draw [color=blue,line width=0pt] (4231) -- (4132);
  \draw [color=green,line width=0pt] (4231) -- (3241);
  \draw [color=red,line width=0pt] (1324) -- (1234);
  \draw [color=red,line width=0pt] (3124) -- (2134);
  \draw [color=red,line width=0pt] (1342) -- (1243);
  \draw [color=red,line width=0pt] (3142) -- (2143);
  \draw [color=red,line width=2pt] (3412) -- (2413);
  \draw [color=red,line width=2pt] (4312) -- (4213);
  \draw [color=green,line width=2pt] (4312) -- (3412);
  \draw [color=red,line width=2pt] (1432) -- (1423);
  \draw [color=green,line width=0pt] (1432) -- (1342);
  \draw [color=red,line width=0pt] (4132) -- (4123);
  \draw [color=green,line width=0pt] (4132) -- (3142);
  \draw [color=green,line width=0pt] (1243) -- (1234);
  \draw [color=blue,line width=0pt] (2143) -- (1243);
  \draw [color=green,line width=0pt] (2143) -- (2134);
  \draw [color=blue,line width=2pt] (2413) -- (1423);
  \draw [color=green,line width=2pt] (2413) -- (2314);
  \draw [color=blue,line width=0pt] (4213) -- (4123);
  \draw [color=green,line width=2pt] (4213) -- (3214);
  \draw [color=green,line width=2pt] (1423) -- (1324);
  \draw [color=green,line width=0pt] (4123) -- (3124);
\end{tikzpicture}
}$}

}

    \scalebox{1.2}{\fbox{$\stackrel{{\color{blue}1}{\color{green}3}{\color{red}\overline{2}}}{
\begin{tikzpicture}[xscale=.5,yscale=.5,inner sep=0pt,baseline=(current bounding box.east)]
  \node (1234) at (0, 0) {.};
  \node (2134) at (-1, 1) {.};
  \node (2314) at (-1, 2) {.};
  \node (3214) at (-1.50000000000000, 3) {.};
  \node (2341) at (1.50000000000000, 3) {.};
  \node (3241) at (1, 4) {.};
  \node (3421) at (1, 5) {.};
  \node (4321) at (0, 6) {.};
  \node (2431) at (2, 4) {.};
  \node (4231) at (0, 5) {.};
  \node (1324) at (0, 1) {.};
  \node (3124) at (-2, 2) {.};
  \node (1342) at (2, 2) {.};
  \node (3142) at (-0.500000000000000, 3) {.};
  \node (3412) at (0, 4) {.};
  \node (4312) at (-1, 5) {.};
  \node (1432) at (2.50000000000000, 3) {.};
  \node (4132) at (-1, 4) {.};
  \node (1243) at (1, 1) {.};
  \node (2143) at (0, 2) {.};
  \node (2413) at (0.500000000000000, 3) {.};
  \node (4213) at (-2, 4) {.};
  \node (1423) at (1, 2) {.};
  \node (4123) at (-2.50000000000000, 3) {.};
  \draw [color=blue,line width=0pt] (2134) -- (1234);
  \draw [color=blue,line width=2pt] (2314) -- (1324);
  \draw [color=blue,line width=2pt] (3214) -- (3124);
  \draw [color=red,line width=0pt] (3214) -- (2314);
  \draw [color=blue,line width=2pt] (2341) -- (1342);
  \draw [color=blue,line width=2pt] (3241) -- (3142);
  \draw [color=red,line width=0pt] (3241) -- (2341);
  \draw [color=blue,line width=0pt] (3421) -- (3412);
  \draw [color=red,line width=0pt] (3421) -- (2431);
  \draw [color=blue,line width=0pt] (4321) -- (4312);
  \draw [color=red,line width=0pt] (4321) -- (4231);
  \draw [color=green,line width=0pt] (4321) -- (3421);
  \draw [color=blue,line width=2pt] (2431) -- (1432);
  \draw [color=green,line width=2pt] (2431) -- (2341);
  \draw [color=blue,line width=2pt] (4231) -- (4132);
  \draw [color=green,line width=2pt] (4231) -- (3241);
  \draw [color=red,line width=2pt] (1324) -- (1234);
  \draw [color=red,line width=2pt] (3124) -- (2134);
  \draw [color=red,line width=2pt] (1342) -- (1243);
  \draw [color=red,line width=2pt] (3142) -- (2143);
  \draw [color=red,line width=0pt] (3412) -- (2413);
  \draw [color=red,line width=0pt] (4312) -- (4213);
  \draw [color=green,line width=0pt] (4312) -- (3412);
  \draw [color=red,line width=0pt] (1432) -- (1423);
  \draw [color=green,line width=2pt] (1432) -- (1342);
  \draw [color=red,line width=0pt] (4132) -- (4123);
  \draw [color=green,line width=2pt] (4132) -- (3142);
  \draw [color=green,line width=0pt] (1243) -- (1234);
  \draw [color=blue,line width=0pt] (2143) -- (1243);
  \draw [color=green,line width=0pt] (2143) -- (2134);
  \draw [color=blue,line width=2pt] (2413) -- (1423);
  \draw [color=green,line width=2pt] (2413) -- (2314);
  \draw [color=blue,line width=2pt] (4213) -- (4123);
  \draw [color=green,line width=2pt] (4213) -- (3214);
  \draw [color=green,line width=2pt] (1423) -- (1324);
  \draw [color=green,line width=2pt] (4123) -- (3124);
\end{tikzpicture}
 \raisebox{-.5ex}{$\ \mapsto\ $} 
\begin{tikzpicture}[xscale=.5,yscale=.5,inner sep=0pt,baseline=(current bounding box.east)]
  \node (1234) at (0, 0) {.};
  \node (2134) at (-1, 1) {.};
  \node (2314) at (-1, 2) {.};
  \node (3214) at (-1.50000000000000, 3) {.};
  \node (2341) at (1.50000000000000, 3) {.};
  \node (3241) at (1, 4) {$\bullet$};
  \node (3421) at (1, 5) {.};
  \node (4321) at (0, 6) {.};
  \node (2431) at (2, 4) {.};
  \node (4231) at (0, 5) {$\bullet$};
  \node (1324) at (0, 1) {.};
  \node (3124) at (-2, 2) {.};
  \node (1342) at (2, 2) {.};
  \node (3142) at (-0.500000000000000, 3) {$\bullet$};
  \node (3412) at (0, 4) {.};
  \node (4312) at (-1, 5) {.};
  \node (1432) at (2.50000000000000, 3) {.};
  \node (4132) at (-1, 4) {$\bullet$};
  \node (1243) at (1, 1) {.};
  \node (2143) at (0, 2) {$\bullet$};
  \node (2413) at (0.500000000000000, 3) {.};
  \node (4213) at (-2, 4) {.};
  \node (1423) at (1, 2) {.};
  \node (4123) at (-2.50000000000000, 3) {.};
  \draw [color=blue,line width=0pt] (2134) -- (1234);
  \draw [color=blue,line width=0pt] (2314) -- (1324);
  \draw [color=blue,line width=0pt] (3214) -- (3124);
  \draw [color=red,line width=0pt] (3214) -- (2314);
  \draw [color=blue,line width=0pt] (2341) -- (1342);
  \draw [color=blue,line width=2pt] (3241) -- (3142);
  \draw [color=red,line width=0pt] (3241) -- (2341);
  \draw [color=blue,line width=0pt] (3421) -- (3412);
  \draw [color=red,line width=0pt] (3421) -- (2431);
  \draw [color=blue,line width=0pt] (4321) -- (4312);
  \draw [color=red,line width=0pt] (4321) -- (4231);
  \draw [color=green,line width=0pt] (4321) -- (3421);
  \draw [color=blue,line width=0pt] (2431) -- (1432);
  \draw [color=green,line width=0pt] (2431) -- (2341);
  \draw [color=blue,line width=2pt] (4231) -- (4132);
  \draw [color=green,line width=2pt] (4231) -- (3241);
  \draw [color=red,line width=0pt] (1324) -- (1234);
  \draw [color=red,line width=0pt] (3124) -- (2134);
  \draw [color=red,line width=0pt] (1342) -- (1243);
  \draw [color=red,line width=2pt] (3142) -- (2143);
  \draw [color=red,line width=0pt] (3412) -- (2413);
  \draw [color=red,line width=0pt] (4312) -- (4213);
  \draw [color=green,line width=0pt] (4312) -- (3412);
  \draw [color=red,line width=0pt] (1432) -- (1423);
  \draw [color=green,line width=0pt] (1432) -- (1342);
  \draw [color=red,line width=0pt] (4132) -- (4123);
  \draw [color=green,line width=2pt] (4132) -- (3142);
  \draw [color=green,line width=0pt] (1243) -- (1234);
  \draw [color=blue,line width=0pt] (2143) -- (1243);
  \draw [color=green,line width=0pt] (2143) -- (2134);
  \draw [color=blue,line width=0pt] (2413) -- (1423);
  \draw [color=green,line width=0pt] (2413) -- (2314);
  \draw [color=blue,line width=0pt] (4213) -- (4123);
  \draw [color=green,line width=0pt] (4213) -- (3214);
  \draw [color=green,line width=0pt] (1423) -- (1324);
  \draw [color=green,line width=0pt] (4123) -- (3124);
\end{tikzpicture}
}$}

}
    \scalebox{1.2}{\fbox{$\stackrel{{\color{red}\overline{2}}{\color{blue}\overline{1}}{\color{red}2}{\color{green}\overline{3}}}{
\begin{tikzpicture}[xscale=.5,yscale=.5,inner sep=0pt,baseline=(current bounding box.east)]
  \node (1234) at (0, 0) {.};
  \node (2134) at (-1, 1) {.};
  \node (2314) at (-1, 2) {.};
  \node (3214) at (-1.50000000000000, 3) {.};
  \node (2341) at (1.50000000000000, 3) {.};
  \node (3241) at (1, 4) {.};
  \node (3421) at (1, 5) {.};
  \node (4321) at (0, 6) {.};
  \node (2431) at (2, 4) {.};
  \node (4231) at (0, 5) {.};
  \node (1324) at (0, 1) {.};
  \node (3124) at (-2, 2) {.};
  \node (1342) at (2, 2) {.};
  \node (3142) at (-0.500000000000000, 3) {.};
  \node (3412) at (0, 4) {.};
  \node (4312) at (-1, 5) {.};
  \node (1432) at (2.50000000000000, 3) {.};
  \node (4132) at (-1, 4) {.};
  \node (1243) at (1, 1) {.};
  \node (2143) at (0, 2) {.};
  \node (2413) at (0.500000000000000, 3) {.};
  \node (4213) at (-2, 4) {.};
  \node (1423) at (1, 2) {.};
  \node (4123) at (-2.50000000000000, 3) {.};
  \draw [color=blue,line width=0pt] (2134) -- (1234);
  \draw [color=blue,line width=0pt] (2314) -- (1324);
  \draw [color=blue,line width=0pt] (3214) -- (3124);
  \draw [color=red,line width=0pt] (3214) -- (2314);
  \draw [color=blue,line width=2pt] (2341) -- (1342);
  \draw [color=blue,line width=2pt] (3241) -- (3142);
  \draw [color=red,line width=0pt] (3241) -- (2341);
  \draw [color=blue,line width=2pt] (3421) -- (3412);
  \draw [color=red,line width=0pt] (3421) -- (2431);
  \draw [color=blue,line width=2pt] (4321) -- (4312);
  \draw [color=red,line width=0pt] (4321) -- (4231);
  \draw [color=green,line width=0pt] (4321) -- (3421);
  \draw [color=blue,line width=2pt] (2431) -- (1432);
  \draw [color=green,line width=0pt] (2431) -- (2341);
  \draw [color=blue,line width=2pt] (4231) -- (4132);
  \draw [color=green,line width=0pt] (4231) -- (3241);
  \draw [color=red,line width=0pt] (1324) -- (1234);
  \draw [color=red,line width=0pt] (3124) -- (2134);
  \draw [color=red,line width=0pt] (1342) -- (1243);
  \draw [color=red,line width=0pt] (3142) -- (2143);
  \draw [color=red,line width=0pt] (3412) -- (2413);
  \draw [color=red,line width=0pt] (4312) -- (4213);
  \draw [color=green,line width=0pt] (4312) -- (3412);
  \draw [color=red,line width=0pt] (1432) -- (1423);
  \draw [color=green,line width=0pt] (1432) -- (1342);
  \draw [color=red,line width=0pt] (4132) -- (4123);
  \draw [color=green,line width=0pt] (4132) -- (3142);
  \draw [color=green,line width=2pt] (1243) -- (1234);
  \draw [color=blue,line width=0pt] (2143) -- (1243);
  \draw [color=green,line width=2pt] (2143) -- (2134);
  \draw [color=blue,line width=0pt] (2413) -- (1423);
  \draw [color=green,line width=2pt] (2413) -- (2314);
  \draw [color=blue,line width=0pt] (4213) -- (4123);
  \draw [color=green,line width=2pt] (4213) -- (3214);
  \draw [color=green,line width=2pt] (1423) -- (1324);
  \draw [color=green,line width=2pt] (4123) -- (3124);
\end{tikzpicture}
 \raisebox{-.5ex}{$\ \mapsto\ $} 
\begin{tikzpicture}[xscale=.5,yscale=.5,inner sep=0pt,baseline=(current bounding box.east)]
  \node (1234) at (0, 0) {.};
  \node (2134) at (-1, 1) {.};
  \node (2314) at (-1, 2) {.};
  \node (3214) at (-1.50000000000000, 3) {.};
  \node (2341) at (1.50000000000000, 3) {.};
  \node (3241) at (1, 4) {.};
  \node (3421) at (1, 5) {.};
  \node (4321) at (0, 6) {.};
  \node (2431) at (2, 4) {.};
  \node (4231) at (0, 5) {.};
  \node (1324) at (0, 1) {$\bullet$};
  \node (3124) at (-2, 2) {.};
  \node (1342) at (2, 2) {.};
  \node (3142) at (-0.500000000000000, 3) {.};
  \node (3412) at (0, 4) {.};
  \node (4312) at (-1, 5) {.};
  \node (1432) at (2.50000000000000, 3) {.};
  \node (4132) at (-1, 4) {.};
  \node (1243) at (1, 1) {.};
  \node (2143) at (0, 2) {.};
  \node (2413) at (0.500000000000000, 3) {$\bullet$};
  \node (4213) at (-2, 4) {.};
  \node (1423) at (1, 2) {$\bullet$};
  \node (4123) at (-2.50000000000000, 3) {.};
  \draw [color=blue,line width=0pt] (2134) -- (1234);
  \draw [color=blue,line width=0pt] (2314) -- (1324);
  \draw [color=blue,line width=0pt] (3214) -- (3124);
  \draw [color=red,line width=0pt] (3214) -- (2314);
  \draw [color=blue,line width=0pt] (2341) -- (1342);
  \draw [color=blue,line width=0pt] (3241) -- (3142);
  \draw [color=red,line width=0pt] (3241) -- (2341);
  \draw [color=blue,line width=0pt] (3421) -- (3412);
  \draw [color=red,line width=0pt] (3421) -- (2431);
  \draw [color=blue,line width=0pt] (4321) -- (4312);
  \draw [color=red,line width=0pt] (4321) -- (4231);
  \draw [color=green,line width=0pt] (4321) -- (3421);
  \draw [color=blue,line width=0pt] (2431) -- (1432);
  \draw [color=green,line width=0pt] (2431) -- (2341);
  \draw [color=blue,line width=0pt] (4231) -- (4132);
  \draw [color=green,line width=0pt] (4231) -- (3241);
  \draw [color=red,line width=0pt] (1324) -- (1234);
  \draw [color=red,line width=0pt] (3124) -- (2134);
  \draw [color=red,line width=0pt] (1342) -- (1243);
  \draw [color=red,line width=0pt] (3142) -- (2143);
  \draw [color=red,line width=0pt] (3412) -- (2413);
  \draw [color=red,line width=0pt] (4312) -- (4213);
  \draw [color=green,line width=0pt] (4312) -- (3412);
  \draw [color=red,line width=0pt] (1432) -- (1423);
  \draw [color=green,line width=0pt] (1432) -- (1342);
  \draw [color=red,line width=0pt] (4132) -- (4123);
  \draw [color=green,line width=0pt] (4132) -- (3142);
  \draw [color=green,line width=0pt] (1243) -- (1234);
  \draw [color=blue,line width=0pt] (2143) -- (1243);
  \draw [color=green,line width=0pt] (2143) -- (2134);
  \draw [color=blue,line width=2pt] (2413) -- (1423);
  \draw [color=green,line width=0pt] (2413) -- (2314);
  \draw [color=blue,line width=0pt] (4213) -- (4123);
  \draw [color=green,line width=0pt] (4213) -- (3214);
  \draw [color=green,line width=2pt] (1423) -- (1324);
  \draw [color=green,line width=0pt] (4123) -- (3124);
\end{tikzpicture}
}$}

}
  \end{bigcenter}
  \caption{The elements $f=\pi_1$, $\pi_2$, $\pi_1\pi_3\opi_2$ and
    $\opi_2\opi_1\pi_2\opi_3$ of $\biheckemonoid(\sg[4])$. As in
    Figure~\ref{figure.preservation_weak_order},
    the underlying poset on both sides is left order on $\sg[4]$, and the
    bold dots on the right sides depict the image set of $f$. On the
    left side, an edge between two elements of $W$ is thick if they
    are not in the same fiber. This information completely describes
    $f$; indeed $u=\one$ on the left is mapped to the lowest element
    of the image set on the right; each time one moves $u$ up along a
    thick edge on the left, its image $u.f$ is moved up along the
    edge of the same color on the right.}
  \label{figure.some_elements}
\end{figure}
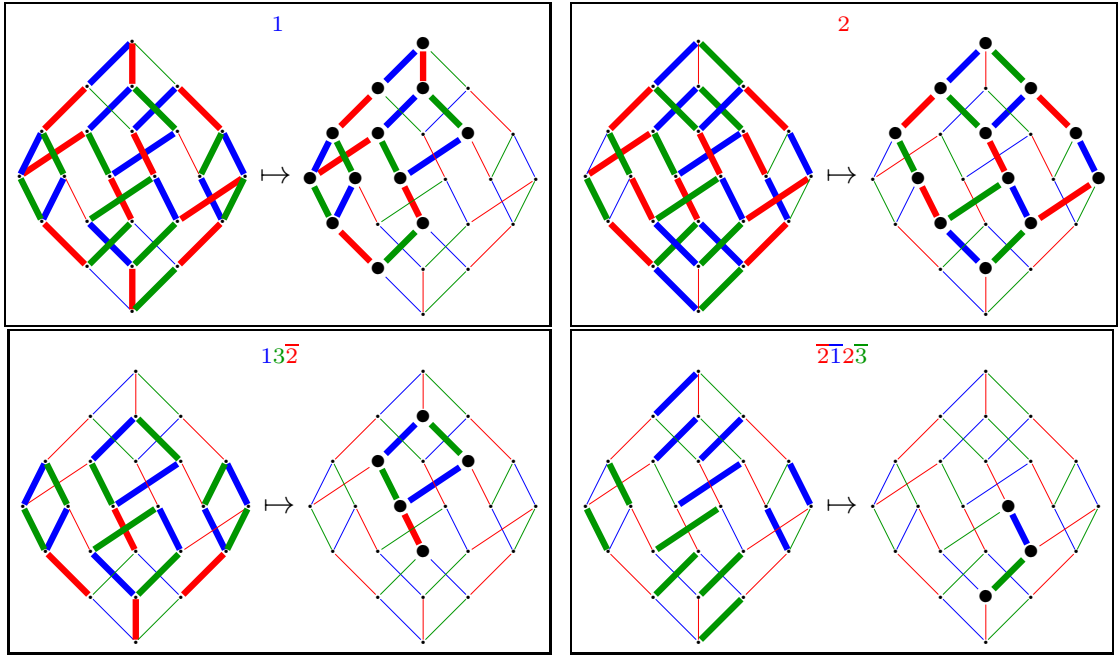

Recall that a set-partition $\Lambda=\{\Lambda_i\}$ is said to be
finer than the set-partition $\Lambda'=\{\Lambda'_i\}$ if for all $i$
there exists a $j$ such that $\Lambda_i\subseteq \Lambda'_j$.  This is
denoted by $\Lambda \finer \Lambda'$. The refinement relation is a
partial order.

For $f\in \biheckemonoid(W)$, define the \emph{type} of $f$ by
\begin{equation}
  \type(f) := \type([\one.f,w_0.f]_L) = (w_0.f)(\one.f)^{-1}\,.
\end{equation}
The \emph{rank} of $f\in \biheckemonoid(W)$ is the cardinality of the image set $\im(f)$.

\begin{lemma}
  \label{lemma.right_module}
  Fix $f\in \biheckemonoid(W)$. For $h=fg\in f\biheckemonoid(W)$, one has:
  \begin{enumerate}
  \item $\fibers(f) \finer \fibers(h)$
  \item $\type(h) \leq_B \type(f)$
  \item $\rank(h) \leq \rank(f)$.
  \end{enumerate}
  Furthermore, the following are equivalent
  \begin{enumerate}[(i)]
  \item $\fibers(h) = \fibers(f)$
  \item $\rank(h)  = \rank(f)$
  \item $\type(h)  = \type(f)$
  \item $\ell(w_0.h)-\ell(\one.h) = \ell(w_0.f) - \ell(\one.f)$.
  \end{enumerate}
  If any, and therefore all, of the above hold, then $h$ is completely
  determined (within $f\biheckemonoid(W)$) by $\one.h$.
\end{lemma}

\begin{proof}
  For $f,g\in \biheckemonoid(W)$, the statement $\fibers(f) \finer \fibers(fg)$ is obvious.

  By Proposition~\ref{proposition.weak_order2}~\eqref{item.length_contracting}
  and~\eqref{item.image_set}, we know that for $f,g\in \biheckemonoid(W)$ either
  $\type(fg)=\type(f)$, or $\ell(w_0.(fg))-\ell(\one.(fg)) <
  \ell(w_0.f) - \ell(\one.f)$. In the latter case by
  Proposition~\ref{proposition.bruhat} $\type(fg) <_B \type(f)$.  The
  second case occurs precisely when $\fibers(f)$ is strictly finer than
  $\fibers(fg)$, or equivalently $\rank(fg)<\rank(f)$.

  The last statement, that if $\fibers(h)=\fibers(f)$ then $h$ is determined by $\one.h$,
  follows from Proposition~\ref{proposition.fibers_image_set}.
\end{proof}

\subsection{Aperiodicity}
\label{ss.aperiodic}

Recall from Section~\ref{ss.preorders on monoids} that a monoid $M$ is
called \emph{aperiodic} if for any $f\in M$, there exists $k>0$ such
that $f^{k+1}=f^k$. Note that, in this case, $f^\omega :=
f^k=f^{k+1}=\dots$ is an idempotent.

\begin{proposition}
  \label{proposition.pseudo_idempotent}
  The biHecke monoid $\biheckemonoid(W)$ is \emph{aperiodic}.
\end{proposition}

\begin{proof}
  From Proposition~\ref{proposition.weak_order2}~\eqref{item.image_set}, we know that
  $\im(f^k)$ has a minimal element $a_k=1.f^k$ and a maximal element $b_k=w_0.f^k$ in
  left order. Since $\im(f^{k+1}) \subseteq \im(f^k)$, we have $a_{k+1} \geq_L a_k$ and
  $b_{k+1} \leq_L b_k$. Therefore, both sequences $a_k$ and $b_k$ must ultimately be
  constant.

  This implies that, for $N$ big enough, $a_N$ and $b_N$ are fixed
  points. Applying Proposition~\ref{proposition.weak_order2}~\eqref{item.length_contracting}
  yields that all elements in $[a_N,b_N]_L$ are fixed points under $f$. It
  follows successively that $\im(f^N) = [a_N,b_N]_L$, $f^N=f^{N+1}=\cdots$,
  and $\fix(f) = [a_N,b_N]_L$.
\end{proof}

\begin{corollary}
  The set of fixed points of an element $f\in \biheckemonoid(W)$ is an
  interval in left order.
\end{corollary}
\begin{proof}
  The set of fixed point of $f$ is the image set of $f^\omega$, which by
  Proposition~\ref{proposition.image_set.idempotent}~\eqref{item.image_set.idempotent}
  is an interval in left order.
\end{proof}

\subsection{Idempotents}
\label{ss.idempotents}

We now study the properties of idempotents in $\biheckemonoid(W)$.
\begin{proposition}\mbox{}
  \label{proposition.uniqueness_fix1}
  \begin{enumerate}[(i)]
  \item \label{e.uniqueness}
    For $w\in W$
    \begin{equation*}
      e_w := \pi_{w^{-1} w_0} \opi_{w_0 w}
    \end{equation*}
    is the unique idempotent such that $1.e_w=1$ and $w_0.e_w=w$. Its
    image set is $[1,w]_L$, and it satisfies:
    \begin{equation*}
      u.e_w = \max_{\le_B} \bigl( [1,u]_B \cap [1,w]_L \bigr) \; .
    \end{equation*}
  \item \label{e.uniqueness.dual}  Similarly, for $w\in W$,
    \begin{equation*}
      \tilde{e}_w := \opi_{w^{-1}} \pi_{w}
    \end{equation*}
    is the unique idempotent with image set $[w,w_0]_L$, and it
    satisfies a dual formula.
  \item \label{e.shift}
    Furthermore,
    \begin{equation*}
      e_{a,b} := \opi_{a^{-1}} e_{ba^{-1}} \pi_a
    \end{equation*}
    is an idempotent with image set $[a,b]_L$.
  \end{enumerate}
\end{proposition}

\begin{proof}
  \eqref{e.uniqueness}: Clearly, the image of $e_w$ is a subset of
  $[1,w]_L$. Applying Remark~\ref{remark.interval} shows that
  $[1,w]_L$ is successively mapped bijectively to $[w^{-1} w_0,w_0]_L$
  and back to $[1,w]_L$. So $e_w$ is an idempotent with image set
  $[1,w]_L$.
  Reciprocally, let $f$ be an idempotent such that $1.f=1$ and
  $w_0.f=w$. Then, by Proposition~\ref{proposition.bruhat} $f$
  preserves Bruhat order and by
  Corollary~\ref{corollary.Bruhat_regressive}~\eqref{item.regressive}
  $u.f\leq_B u$ for all $u\in W$. Furthermore, by
  Proposition~\ref{proposition.image_set.idempotent}, the image set of
  $f$ is the interval $[1,w]_L$.
  Using Proposition~\ref{proposition.uniqueness.idempotents},
  uniqueness and the given formula follow.

  Statement~\eqref{e.uniqueness.dual} is dual to~\eqref{e.uniqueness} and is proved similarly.

  \eqref{e.shift}: The image set of $e_{ba^{-1}}$ is $[1,ba^{-1}]_L$;
  hence the image set of $e_{a,b}$ is a subset of $[a,b]_L$. We
  conclude by checking that $[a,b]_L$ is mapped bijectively at each
  step $\opi_{a^{-1}}$, $e_{ba^{-1}}$ and $\pi_a$ (see also Remark~\ref{remark.interval}),
  and therefore consists of fixed points.
\end{proof}

\begin{remark}
  \label{remark.rmul_e}
  For $f\in \biheckemonoid(W)$, $f e_v = f e_{u.e_v}$, where $u=w_0.f$.
\end{remark}
\begin{proof}
  Use the formula of Proposition~\ref{proposition.uniqueness_fix1}~\eqref{e.uniqueness}.
\end{proof}

\begin{corollary}
  \label{corollary.e}
  For $u,w \in W$, the intersection $[1,u]_B \cap [1,w]_L$ is a $\le_L$-lower~
  set with a unique maximal element $v$ in Bruhat order. The maximum is given
  by $v=u.e_w$.
\end{corollary}

\subsection{Green's relations}
\label{ss.green_relations}

We have now gathered enough information about the combinatorics of
$\biheckemonoid(W)$ to give a partial description of its Green's
relations which will be used in the study of the representation theory
of $\biheckemonoid(W)$. As an example, Figure~\ref{figure.j_graph_A2}
completely describes the Green's relations $\LL$, $\RR$, $\JJ$, and
$\HH$ for $\biheckemonoid(\sg[3])$. In the sequel, we describe
$\RR$-classes for general elements, as well as $\JJ$-order on regular
elements. In particular, we obtain that the $\JJ$-classes of
idempotents are indexed by the elements of $W$, and that $\JJ$-order
on regular classes is given by left-right order $<_{LR}$ on
$W$. Note that the latter is \emph{not} a lattice, unlike for the variety $\mathcal{DA}$
(which consists of all aperiodic monoids all of whose simple modules are dimension $1$;
see e.g.~\cite{Ganyushkin_Mazorchuk_Steinberg.2009}).

\begin{figure}
  \scalebox{.6}{\begin{previewthispiece}
  \pgfdeclarelayer{nodes}
  \pgfsetlayers{main,nodes}
  \begin{tikzpicture}[auto,xscale=.4,yscale=-.28,line width=1pt]
    \tikzstyle{every node}=[fill=white,inner sep=0pt]

    \node(one)	at (50,0) {};

    \node(1)	at (35, 7) {};
    \node(-1)	at (35,17) {};
    \node(2)	at (65, 7) {};
    \node(-2)	at (65,17) {};

    \node(-121)	at (28,30) {};
    \node(-12)	at (28,40) {};
    \node(-1-2)	at (28,50) {};

    \node(21)	at (42,30) {};
    \node(2-1)	at (42,40) {};
    \node(2-1-2)	at (42,50) {};

    \node(12)	at (58,30) {};
    \node(1-2)	at (58,40) {};
    \node(1-2-1)	at (58,50) {};

    \node(-212)	at (72,30) {};
    \node(-21)	at (72,40) {};
    \node(-2-1)	at (72,50) {};

    \node(-1-2-1)	at (50,87) {};
    \node(-2-1-2)	at (50,87) {};

    \node(-1-21)	at (35,80) {};
    \node(-2-12)	at (65,80) {};

    \node(21-2)	at (35,70) {};
    \node(12-1)	at (65,70) {};

    \node(121)	at (50,63) {};

    \input{Fig/J-graph-A2-inner}
  \end{tikzpicture}
\end{previewthispiece}}
  \caption{The graph of $\JJ$-order for $\biheckemonoid(\sg[3])$. The vertices are the $23$ elements
    of $\biheckemonoid(\sg[3])$, each drawn as in Figure~\ref{figure.some_elements}. The edges give both the left
    and right Cayley graph of $\biheckemonoid(\sg[3])$; for example, there is an arrow
     $f\stackrel{\times\pi_1}\longrightarrow g$ if $g=f \pi_1$,
     and an arrow $f\stackrel{\pi_1\times\pi_1}\longrightarrow g$ if $g=f \pi_1=\pi_1f$.
     The picture also highlights the $\JJ$-classes of $\biheckemonoid(\sg[3])$, and
     the corresponding eggbox pictures (i.e. the decomposition of the $\JJ$-classes into
     $\LL$ and $\RR$-classes); namely, from top to bottom, there is
     one $\JJ$-class of size $1 = 1\times 1$, two $\JJ$-classes of
     size $2=1\times 2$, two $\JJ$-classes of size $6=2\times 3$, and
     one $\JJ$-class of size $6=1\times6$, where $n\times m$ gives the
     dimension of the eggbox picture. In other words the $\JJ$-class
     splits into $n$ $\RR$-classes of size $m$ and also into $m$
     $\LL$-classes of size $n$.
     This example is specific in that all $\JJ$-classes are regular.
   }
  \label{figure.j_graph_A2}
\end{figure}

\begin{proposition}
  \label{proposition.characterization_R_class}
  Two elements $f,g\in \biheckemonoid(W)$ are in the same
  $\RR$-class if and only if they have the same fibers.
  In particular, the $\RR$-class of $f$ is given by
  \begin{equation} \label{equation.R_class_rank}
    \RR(f) = \{h\in f\biheckemonoid(W) \suchthat \rank(h)=\rank(f)\}
    = \{ f_u \suchthat u\in [1,\type(f)^{-1}w_0]_R \}\,,
  \end{equation}
  where $f_u$ is the unique element of $\biheckemonoid(W)$ such that
  $\fibers(f_u)=\fibers(f)$ and $\one.f_u=u$.
\end{proposition}

\begin{proof}
  It is a general easy fact about monoids of functions that elements
  in the same $\RR$-class have the same fibers (see also
  Lemma~\ref{lemma.right_module}). Reciprocally, if $g$ has the same
  fibers as $f$, then one can use Remark~\ref{remark.interval} to
  define $g'=g\opi_{(\one.g)^{-1}}\pi_{\one.f}$ such that $\fibers(g')
  = \fibers(f)$ and $\one.g'=\one.f$. Also by
  Proposition~\ref{proposition.fibers_image_set}, $f=g'\in g\biheckemonoid(W)$, and
  similarly, $g\in f\biheckemonoid(W)$.

  Equation~\eqref{equation.R_class_rank} follows using
  Lemma~\ref{lemma.right_module} and Remark~\ref{remark.interval}.
\end{proof}

\begin{lemma}
  \label{lemma.conjugate}
  Let $e$ and $f$ be idempotents of $\biheckemonoid(W)$ with respective
  image sets $[a,b]_L$ and $[c,d]_L$. Then, $f \leq_\JJ e$ if and only
  if $dc^{-1} \le_{LR} ba^{-1}$.

  In particular, two idempotents $e$ and $f$ are $\JJ$-equivalent if and
  only if the intervals $[a,b]_L$ and $[c,d]_L$ are of the same type:
  $dc^{-1} = ba^{-1}$.

  The above properties extend to any two regular elements (elements
  whose $\JJ$-class contains an idempotent).
\end{lemma}
\begin{proof}
  First note that an interval $[c,d]_L$ is isomorphic to a
  subinterval of $[a,b]_L$ if and only $dc^{-1} \le_{LR} ba^{-1}$.
  This follows from Proposition~\ref{proposition.interval} and the fact that $[c,d]_L$ is
  a subinterval of $[a,b]_L$ if and only if $[ca^{-1},da^{-1}]_L$ is a
  subinterval of $[1,ba^{-1}]_L$. But then $dc^{-1}$ is a subfactor
  of $ba^{-1}$.

  Assume first that $dc^{-1} \le_{LR} ba^{-1}$, and let $[c',d']_L$ be a
  subinterval of $[a,b]_L$ isomorphic to $[c,d]_L$. Using
  Proposition~\ref{proposition.interval}, take $u,v\in \biheckemonoid(W)$ which induce
  reciprocal bijections between $[c,d]_L$ and $[c',d']_L$. Then, $f=fuev$, so that $f$
  is $\JJ$-equivalent to $e$.

  Reciprocally, assume that $f=u e v$ with $u,v\in \biheckemonoid(W)$. Without loss of
  generality, we may assume that $u=ue$ so that $\im(u)\subseteq
  [a,b]_L$. Set $c'=c.u$ and $d'=d.u$. Since $f = ff = f uv$, and
  using Proposition~\ref{proposition.weak_order2}, the functions $u$
  and $v$ must induce reciprocal isomorphisms between $[c,d]_L$ and
  $[c',d']_L$, the latter being a subinterval of $[a,b]_L$. Therefore,
  $dc^{-1} \le_{LR} ba^{-1}$.

  To conclude, note that a regular element has the same type as any
  idempotent in its $\JJ$-class.
\end{proof}

\begin{corollary}
  \label{corollary.transversal}
  The idempotents $(e_w)_{w\in W}$ form a complete set of
  representatives of regular $\JJ$-classes in $\biheckemonoid(W)$.
\end{corollary}

\begin{example}
\label{example.e.etilde}
  For $w\in W$, the idempotents $e_w$ and $\tilde{e}_{w^{-1} w_0}$ are in the same $\JJ$-class.
  This follows immediately from Lemma~\ref{lemma.conjugate}, or by direct computation
   using the explicit expressions for $e_w$ and $\tilde{e}_{w^{-1} w_0}$ in
   Proposition~\ref{proposition.uniqueness_fix1}:
   \begin{equation*}
   \begin{split}
     &e_w = e_w^2 = \pi_{w^{-1} w_0} \opi_{w_0 w} \pi_{w^{-1} w_0} \opi_{w_0 w}
                 = \pi_{w^{-1} w_0} \tilde{e}_{w^{-1} w_0} \opi_{w_0 w}\,,\\
     & \tilde{e}_{w^{-1} w_0} = \tilde{e}_{w^{-1} w_0}^2 =
     \opi_{w_0 w}\pi_{w^{-1} w_0}\opi_{w_0 w}\pi_{w^{-1} w_0} =
     \opi_{w_0 w} e_w \pi_{w^{-1} w_0}\,.
   \end{split}
   \end{equation*}
 \end{example}

\begin{corollary}
  The image of a regular element is an interval in left order.
\end{corollary}
\begin{proof}
  A regular element has the same type, and same size of image set as
  any idempotent in its $\JJ$-class.
\end{proof}

\begin{remark}
  The reciprocal is false: In type $B_3$, the element
  $\opi_1\opi_3\opi_2\pi_1\opi_3\opi_2\opi_1$ has the interval $[1,
  s_2s_3s_2]_L$ as image set, but it is not regular. The same holds in
  type $A_4$ with the element
  $\pi_2\pi_1\opi_4\pi_3\opi_2\opi_1\opi_3\pi_4\opi_2\opi_3\opi_4$.
\end{remark}

\begin{problem}
  Describe $\LL$-classes in general, and $\LL$-order, $\RR$-order, as well as
  $\JJ$-order on nonregular elements.
\end{problem}

\subsection{Involutions and consequences}
\label{subsection.involution}

Define an involution $*$ on $W$ by
\begin{equation*}
	w\mapsto w^* := w_0 w,
\end{equation*}
where $w_0$ is the maximal element of $W$. Moreover, define the \emph{bar} map
$\biheckemonoid(W) \to \biheckemonoid(W)$ as the conjugacy by $*$: for a
given $f\in \biheckemonoid (W)$
\begin{equation*}
  w.\overline{f} := (w^*.f)^* \qquad \text{for all $w\in W$.}
\end{equation*}

\begin{proposition}
  The bar involution is a monoid endomorphism of $\biheckemonoid(W)$ which
  exchanges $\pi_i$ and $\opi_i$.
\end{proposition}

\begin{proof}
  This is a consequence of the general fact that for any permutation $\phi$ of
  $W$, conjugation by $\phi$ is an automorphism of the monoid of maps from
  $W$ to itself. Moreover, it is easy to see that bar exchanges $\pi_i$ and
  $\opi_i$, so that it fixes $\biheckemonoid(W)$.
\end{proof}

The previous proposition has some interesting consequences when applied to idempotents:
For any $w\in W$, the bar involution is a bijection from $e_w\biheckemonoid(W)$ to
$\overline{e}_w\biheckemonoid(W)$.
But $\overline{e}_w$ fixes $w_0$ and sends $1=w_0^*$ to $w^*$, so that
$\overline{e}_w = e_{w^*, w_0} = \tilde{e}_{w_0 w}$. The latter is in turn
$\JJ$-equivalent to $e_{w_0 w^{-1} w_0}$ by Example~\ref{example.e.etilde}.
This implies the following result.

\begin{corollary}
  The ideals  $e_w\biheckemonoid(W)$ and $e_{w_0 w^{-1} w_0}\biheckemonoid(W)$ are in bijection.
\end{corollary}

\section{The Borel submonoid $\biheckemonoid_1(W)$ and its representation theory}
\label{section.M1}

In the previous section, we outlined the importance of the idempotents
$(e_w)_{w\in W}$. A crucial feature is that they live in a ``Borel'' submonoid
$\Mone(W) \subseteq \biheckemonoid(W)$ of elements of the biHecke monoid
$\biheckemonoid(W)$ which fix the identity:
\begin{equation*}
  \Mone(W) := \{ f\in \biheckemonoid(W) \mid \one.f = \one\}\, .
\end{equation*}

In this section we study this monoid and its representation theory, as an
intermediate step toward the representation theory of $\biheckemonoid(W)$ (see
Section~\ref{section.translation_modules}). For the representation theory of
$\biheckemonoid(W)$, it is actually more convenient to work with the submonoid
fixing $w_0$ instead of $\one$:
\begin{equation*}
  \Mzero(W) := \{ f\in \biheckemonoid(W) \mid w_0.f = w_0\}\, .
\end{equation*}
However, since both monoids $\Mone(W)$ and $\Mzero(W)$ are isomorphic under
the involution of Section~\ref{subsection.involution} and since the
interaction of $\Mone(W)$ with Bruhat order is notationally simpler, we focus
on $\Mone(W)$ in this section.
\medskip

\noindent\textbf{Note.}  In the remainder of this paper, unless explicitly
stated, we fix a Coxeter group $W$ and use the short-hand notation
$\biheckemonoid:=\biheckemonoid(W)$, $\Mone:=\Mone(W)$ and
$\Mzero:=\Mzero(W)$.

\medskip
From the definition it is clear that $\Mone$ is indeed a submonoid
which contains the idempotents $(e_w)_{w\in W}$.  Furthermore, by
Proposition~\ref{proposition.bruhat} and
Corollary~\ref{corollary.Bruhat_regressive} its elements are both
order-preserving and \regressive for Bruhat order. In fact, a bit more
can be said.
\begin{remark}
  For $w\in W$, $w.\Mone$ is the interval $[1,w]_B$ in Bruhat order.
\end{remark}
\begin{proof}
  By Corollary~\ref{corollary.Bruhat_regressive}, for $f \in \Mone$, we have
  $w.f \leq_B w$. Take reciprocally $v\in [1,w]_B$. Then, using
  Proposition~\ref{proposition.uniqueness_fix1}, $w.e_v = v$.
\end{proof}

As a consequence of the preservation and regressiveness on Bruhat
order, $\Mone$ is \emph{an ordered monoid with $1$ on top}. Namely, for
$f,g \in \Mone$, define the relation $f\leq g$ if $w.f \leq_B w.g$ for
all $w\in W$. Then, $\le$ defines a partial order on $\Mone$ such that
$f\leq 1$, $fg \leq f$ and $fg \leq g$ for all $f,g \in \Mone$. In
other words, $\Mone$ is $\BB$-trivial
(see~\cite[Proposition~2.2]{Denton_Hivert_Schilling_Thiery.JTrivialMonoids},
as well as Section~2.5 there) and in particular $\JJ$-trivial.

In the next two subsections, we first study the combinatorics of $\Mone$ and then apply
the general results on the representation theory of $\JJ$-trivial monoids
of~\cite{Denton_Hivert_Schilling_Thiery.JTrivialMonoids} to $\Mone$.

\subsection{$\JJ$-order on idempotents and minimal generating set}

Recall from Section~\ref{ss.preorders on monoids} that $\JJ$-order is
the partial order $\leq_\JJ$ defined by $f\leq_\JJ g$ if there exists
$x,y \in \Mone$ such that $f=xgy$.  The restriction of $\JJ$-order
to idempotents has a very simple description:
\begin{proposition} \label{proposition.star}
  For $u,v\in W$, the following are equivalent:
  \begin{align*}
  \bullet\ e_u e_v &= e_u&&
  \bullet\ u \leq_L v \\
  \bullet\ e_v e_u &= e_u &&
  \bullet e_u \le_{\mathcal{J}} e_v.
  \end{align*}
  Moreover, $(e_ue_v)^\omega=e_{u\meet_L v}$, where $u\meet_L v$ is the
    meet (or greatest lower bound) of $u$ and $v$ in left order.
\end{proposition}

\begin{proof}
This follows from~\cite[Theorem 3.4, Lemma 3.6]{Denton_Hivert_Schilling_Thiery.JTrivialMonoids}
and Proposition~\ref{proposition.uniqueness_fix1}.
\end{proof}

As a consequence the following definition, which plays a central role
in the representation theory of $\JJ$-trivial monoids
(see~\cite{Denton_Hivert_Schilling_Thiery.JTrivialMonoids}), makes
sense.
\begin{definition} \label{definition.lfix_rfix}
  For any element $x\in \Mone$, define
  \begin{equation*}
    \lfix(x) := \min_{\le_L} \{u\in W \suchthat e_ux=x\} \quad \text{and} \quad
    \rfix(x) := \min_{\le_L} \{u\in W \suchthat xe_u=x\} = w_0.x
  \end{equation*}
  the $\min$ being taken for the left order.
\end{definition}

Interestingly, $\Mone$ can be defined as the submonoid of
$\biheckemonoid$ generated by the idempotents $(e_w)_{w\in W}$, and in
fact the subset of these idempotents indexed by Grassmannian elements
(an element $w\in W$ is \emph{Grassmannian} if it has at most one
descent).
\begin{theorem}
  \label{theorem.M1.generating_set}
  $\Mone$ has a unique minimal generating set which
  consists of the
  idempotents $e_w$ where $w^{-1} w_0$ is right Grassmannian.
\end{theorem}
In type $A_{n-1}$ this minimal generating set is of size $2^n - n$
(which is the number of Grassmannian elements in this case~\cite{manivel.2001}).

\begin{proof}
  Define the length $\len(f)$ of an element $f \in \biheckemonoid$ as the length
  of a minimal expression of $f$ as a product of the generators
  $\pi_i$'s and $\opi_i$'s. We now prove by induction on the length
  that $\Mone$ is generated by $\{e_w \suchthat w\in W\}$.

  Take an element $f \in \Mone$ of length $l$. If $l=0$ we are
  done. Otherwise, since $\one.f=\one$, an expression of $f$ as a product of
  $\pi_i$'s and $\opi_i$ contains at least one $\opi_i$. Write $f=gh$
  where $g=\pi_w\opi_i$ for some $w\in W$ and $h \in \biheckemonoid$ so that
  $\len(w)+1+\len(h)=l$.

  \begin{quote}
  \textbf{Claim:} $f=e_{w_0(ws_i)^{-1}} \pi_wh$ and $\pi_wh\in \Mone$.
  \end{quote}

  It follows from the claim that $\len(\pi_w h)<l$, and hence since $\pi_wh\in \Mone$
  we can apply induction to conclude that $\Mone$ is generated by $\{e_w \mid w\in W\}$.

  Let us prove the claim. By minimality of $l$, $i$ is not a descent
  of $w$ (otherwise, we would obtain a shorter expression for $f$:
  $f=\pi_w \opi_i h = \pi_{w'}\pi_i\opi_i h = \pi_{w'}\opi_i
  h$ where $\ell(w')<\ell(w)$). Therefore, $\one.g=\one.(\pi_w\opi_i) = w$. Since
  $f\in \Mone$ it
  follows that $w.h=\one$ and therefore $\pi_w h \in \Mone$. It further follows that
  $\opi_{w^{-1}}\pi_w$ acts trivially on the image set $[w,w_0]_L$
  of $g$, and therefore $f = g\opi_{w^{-1}}\pi_w h$. Note that $g\opi_{w^{-1}}
  =\pi_w\opi_i \opi_{w^{-1}}=  \pi_w \pi_i \opi_i \opi_{w^{-1}}= e_{w_0(ws_i)^{-1}}$.

  By Proposition~\ref{proposition.star}, the idempotents of $\Mone$ are generated
  by the meet-irreducible idempotents $e_w$ in $\mathcal{J}$ order. Here $x$ is meet-irreducible
  if and only if  $x=a$ or $x=b$ whenever $x=a\meet b$ for some $a,b\in \Mone$.
  These meet-irreducible elements are indexed by the elements $w$ of
  $W$ that are meet-irreducible in left order (or equivalently that
  have at most one left nondescent, that is, $w_0 w^{-1}$ is right
  Grassmannian).

  The uniqueness of the minimal generating set is true for any $\JJ$-trivial monoid with a
  minimal generating set~\cite[Theorem 2]{Doyen.1984}~\cite[Theorem 1]{Doyen.1991}.
\end{proof}

Actually one can be much more precise:
\begin{proposition}\label{propositions.M1.idempotent.generated}
  Any element $f\in \Mone$ can be written as a product $e_{w_1}\cdots
  e_{w_k}$, where:
  \begin{itemize}
  \item $w_1>_B\dots>_B w_k$ is a chain in Bruhat order such that any two
    consecutive terms $w_i$ and $w_{i+1}$ are incomparable in left order;
  \item $w_i = \rfix(e_{w_1}\cdots e_{w_i}) = \lfix(e_{w_i}\cdots e_{w_k}) $\,.
  \end{itemize}
\end{proposition}

\begin{proof}
  Start from any expression $e_{w_1}\cdots e_{w_k}$ for $f$. We show
  that if any of the conditions of the proposition is not satisfied,
  the expression can be reduced to a strictly smaller (in length, or
  in Bruhat, term by term) expression, so that induction can be
  applied.

  \begin{itemize}
  \item If $u\not >_B v$, then by Remark~\ref{remark.rmul_e} $e_u e_v =
    e_u e_{u.e_v}$ with $u.e_v <_B v$.
  \item If $u<_L v$, then $e_u e_v=e_u$, and similarly on the right.
  \item If the left symbol $e_u$ for $e_{w_i}\cdots e_{w_k}$ is not
    $e_{w_i}$, then $u<_L w_i$ and
    \begin{equation*}
      e_{w_i}\cdots e_{w_k}=
      e_ue_{w_i}\cdots e_{w_k}=e_u e_{w_{i+1}} \cdots e_{w_k}\,.
    \end{equation*}
    Similarly on the right.\qedhere
  \end{itemize}
\end{proof}

\begin{corollary}
  \label{corollary.Mone.lfixrfix}
  For $f\in \Mone$, $\lfix(f) \geq_B \rfix(f)$, with equality if and
  only if $f$ is an idempotent.
\end{corollary}

\begin{lemma} \label{lemma.lfix}
If $v\le_B u$ in Bruhat order and $u'=\lfix(e_ue_v)$, then
\begin{equation*}
	v \le_B u' \quad \text{and} \quad u' \le_L u.
\end{equation*}
\end{lemma}

\begin{proof}
By Definition~\ref{definition.lfix_rfix}, $u'\le _L u$ since $e_u (e_u e_v) = e_u e_v$ and
for $\Mone$ the minimum is measured in left order. Also by
Proposition~\ref{proposition.uniqueness_fix1}
\begin{equation*}
v=w_0.e_u e_v = w_0.e_{u'} e_u e_v \le_B u'.\qedhere
\end{equation*}
\end{proof}

\begin{lemma}
If $u$ covers $v$ in Bruhat order and $u'=\lfix(e_u e_v)$, then either $u'=u$, or $u'=v$
and $e_u e_v=e_v e_u$.
\end{lemma}

\begin{proof}
By Lemma~\ref{lemma.lfix}, we have that either $u'=u$ or $u'=v$, since $u$ covers $v$
in Bruhat order. When $u'=v$, we have again by Lemma~\ref{lemma.lfix} that
$v\le_L u$. Hence $e_u e_v=e_v = e_v e_u$.
\end{proof}

\subsection{Representation theory}

In this subsection, we specialize general results about the
representation theory of finite $\JJ$-trivial monoids to describe some of
the representation theory of the Borel submonoid $\Mone$, such as its
simple modules, radical, Cartan invariant matrix and quiver. The description also applies to $\Mzero$,
mutatis mutandis. We follow the presentation
of~\cite{Denton_Hivert_Schilling_Thiery.JTrivialMonoids} (also see this paper for the proofs), though many
of the general results have been previously known; see for
example~\cite{Almeida_Margolis_Steinberg_Volkov.2009,Clifford_Preston.1961,Ganyushkin_Mazorchuk_Steinberg.2009,
Lallement_Petrich.1969,Rhodes_Zalcstein.1991}
and references therein.

\subsubsection{Simple modules and radical}
For each $w\in W$ define $\Sone_w$ (written $\Szero_w$ for $\Mzero$) to
be the one-dimensional vector space with basis $\{\epsilon_w\}$
together with the right operation of any $f\in \Mone$ given by
\begin{equation*}\label{eq.szero}
  \epsilon_w.f :=
  \begin{cases}
    \epsilon_w & \text{if $w.f=w$}, \\
    0          & \text{otherwise}.
  \end{cases}
\end{equation*}
The basic features of the representation theory of $\Mone$
can be stated as follows:
\begin{theorem}
  The radical of $\K\Mone$ is the ideal with basis $f^\omega-f$
  for $f \in \Mone$ non-idempotent.  The quotient of $\K\Mone$ by
  its radical is commutative. Therefore, all simple $\K \Mone$-modules
  are one-dimensional. In fact, the family $\{\Sone_w\}_{w\in W}$ forms
  a complete system of representatives of the simple $\K \Mone$-modules.
\end{theorem}

\subsubsection{Cartan matrix and projective modules}
The projective modules and Cartan invariants can be described as follows:
\begin{theorem}
  There is an explicit basis $(b_x)_{x\in\Mone}$ of
  $\K\Mone$ such that, for all $w\in W$,
  \begin{itemize}
  \item the family $\{b_x \suchthat \text{$x\in \Mone$ with
      $\lfix(x)=w$} \}$ is a basis for the right indecomposable projective module $\Pone_w$
    associated to $\Sone_w$;
  \item the family $\{b_x \suchthat \rfix(x) = w\}_{x \in \Mone}$ is a basis for the left
    indecomposable projective module associated to $\Sone_w$.
  \end{itemize}
  Moreover, the Cartan invariant of $\K\Mone$ defined by $c_{u,v} :=
  \dim(e_u \K\Mone e_v)$ for $u,v \in W$ is given by $c_{u,v} =
  |C_{u,v}|$, where
  \begin{equation*}
    C_{u,v} := \{f\in \Mone \suchthat
                \lfix(f)=u\text{ and } \rfix(f)=v \}\,.
  \end{equation*}

  In particular, the Cartan matrix of $\K\Mone$ is upper-unitriangular
  with respect to Bruhat order.
\end{theorem}
\begin{proof}
  Apply~\cite[Section~3.4]{Denton_Hivert_Schilling_Thiery.JTrivialMonoids}
  and conclude with Corollary~\ref{corollary.Mone.lfixrfix}.
\end{proof}

\begin{remark}
  \label{remark.characters.M1}
  In terms of characters, the previous theorem can be restated as
  \begin{equation}
    [\Pone_u] = \sum_{f\in \Mone, \lfix(f)=u} [\Sone_{w_0.f}]\, ,
  \end{equation}
  which gives the following character for the right regular representation
  \begin{equation}
    [\K\Mone] = \sum_{f\in \Mone} [\Sone_{w_0.f}]\,.\qedhere
  \end{equation}
\end{remark}

\begin{problem}
  Describe the Cartan matrix and projective modules of $\K\Mone$ more
  explicitly, if at all possible in terms of the combinatorics of the
  Coxeter group $W$.
\end{problem}

\subsubsection{Quiver}
We now turn to a description of the quiver of $\K \Mone$ in terms of the
combinatorics of left and Bruhat order. Recall that $\Mone$ is a
submonoid of the monoid of regressive and order preserving
functions. As such, it is not only $\JJ$-trivial but also ordered with
$1$ on top, that is $\BB$-trivial (see~\cite[Section~2.5 and
Proposition~2.2]{Denton_Hivert_Schilling_Thiery.JTrivialMonoids}).
By~\cite[Theorem 3.35 and Corollary~3.44]{Denton_Hivert_Schilling_Thiery.JTrivialMonoids}
we know that the vertices of the quiver of a $\JJ$-trivial monoid generated by idempotents are
labeled by its idempotents $(e_x)_x$ and there is an edge from vertex $e_x$ to vertex $e_z$, if $q:=e_xe_z$
is not idempotent, has $\lfix(q) = x $ and $\rfix(q) = z$, and
does not admit any factorization $q=uv$ which is non-trivial: $eu\neq e$ and $vf\neq f$.
By~\cite[Proposition 3.31]{Denton_Hivert_Schilling_Thiery.JTrivialMonoids} the condition that
$q$ has a non-trivial factorization is equivalent to $q$ having a compatible factorization $q=uv$,
meaning that $u,v$ are non-idempotents and $\lfix(q)=\lfix(u)$, $\rfix(u)=\lfix(v)$ and $\rfix(v)=\rfix(q)$.

Let $e_x,e_y,e_z \in \Mone$ be idempotents.
Call $e_y$ an \emph{intermediate factor} for $q:=e_x e_z$ if $e_x e_y
e_z=e_xe_z$. Call further $e_y$ a \emph{non-trivial intermediate
  factor} if $e_xe_y\ne e_x$, and $e_ye_z\ne e_z$.

\begin{lemma}
  \label{lemma.M1.quiver}
  The quiver of $\K \Mone$ is the graph with $W$ as vertex set and edges
  $(x,z)$ for all $x\ne z$ such that $q:=e_xe_z$ satisfies $\lfix(q) = x$ and
  $\rfix(q) = z$ and admits no non-trivial intermediate
  factor $e_y$ with $y \in W$. 
\end{lemma}
\begin{proof}
  Take $q:=e_x e_z$ admitting a non-trivial intermediate factor $e_y$.
  Then $q$ admits a non-trivial factorization $q=(e_xe_y)(e_ye_z)$ in the sense
  of~\cite[Definition~3.25]{Denton_Hivert_Schilling_Thiery.JTrivialMonoids},
  and is therefore not in the quiver.

  Reciprocally, assume that $q$ admits a compatible factorization,
  that is $q=uv$ with $\lfix(u) = x$, $\rfix(u) =\lfix(v)$, $\rfix(v) = z$.
  By~\cite[Lemma~3.29]{Denton_Hivert_Schilling_Thiery.JTrivialMonoids},
  this factorization is non-trivial: $e_x u\ne e_x$ and $ve_z \ne e_z$.
  Using Proposition~\ref{propositions.M1.idempotent.generated}, write
  $u$ and $v$ as $u=e_x e_{y_1}\cdots e_{y_k}$ and $v= e_{y_k}\cdots
  e_{y_\ell}e_z$, with $x>_B y_1>_B\cdots >_B y_\ell>_B z$.  Then,
  $e_x e_{y_i} e_z=e_xe_z$ for any $i$; indeed, since $\Mone$ is
  $\BB$-trivial,
  \begin{displaymath}
    e_x e_z = e_x e_{y_1}\cdots e_{y_\ell} e_z \leq_{\BB} e_x e_{y_i}
    e_z \leq_\BB e_xe_z\,.
  \end{displaymath}
  If any $e_{y_i}$ is a non-trivial intermediate factor for
  $q$, we are done by setting $y=y_i$. Otherwise, for any $i$
  $e_{y_i}e_z =e_z$ ($e_x e_{y_i}=e_x$ is impossible since $x>_B
  y_i$).  But then, $v=e_{y_k}\cdots e_{y_\ell} e_z = e_z$, a
  contradiction.
\end{proof}

\begin{problem}
  Can Lemma~\ref{lemma.M1.quiver} be generalized to any $\BB$-trivial
  monoid? Its statement has been tested successfully on the $0$-Hecke
  monoid in type $A_1-A_6$, $B_3-B_4$, $D_4-D_5$, $H_3-H_4$, $G_2$,
  $I_{135}$, $F_4$.
\end{problem}

Lemma~\ref{lemma.M1.quiver} admits a combinatorial reformulation in
terms of the combinatorics of $W$. For $x,y,z\in W$ such that $x>_Bz$,
call $y\in W$ an \emph{intermediate factor} for $x,z$ if $[1,y]_L$
intersects all intervals $[c,a]_B$ with $a\in [1,x]_L$ and $c \in
[1,z]_L$ non-trivially. Call further $y$ a \emph{non-trivial
  intermediate factor} if $x>_B y >_B z$ and $y\not>_L z$.
\begin{theorem}
  \label{theorem.M1.quiver}
  The quiver of $\K \Mone$ is the graph with $W$ as vertex set, and edges
  $(x,z)$ for all $x>_B z$ and $x\not>_L z$ admitting no non-trivial
  intermediate factor $y$. 
  Each such edge can be associated with the element $q:=e_x e_z$ of the monoid.

  In particular, the quiver of $\K \Mone$ is acyclic and every cover $x
  \succ_B z$ in Bruhat order which is not a cover in left order
  contributes one edge to the quiver.
\end{theorem}

\begin{proof}
  Consider a non-idempotent product $e_x e_z$. Using
  Proposition~\ref{propositions.M1.idempotent.generated}, we may
  assume without loss of generality that $x>_B z$ and $x\not>_L z$,
  and furthermore that $\lfix(e_xe_z) = x$ and $\rfix(e_xe_z)=z$.

  We now show that the combinatorial definition of intermediate factor
  on an element of $y\in W$ is a reformulation of the monoidal one on
  the idempotent $e_y$ of $\Mone$.

  Assume that $e_y$ is an intermediate factor for $e_x e_z$, that is
  $e_x e_y e_z =e_x e_z$. Take $a\in [1,x]_L$ and $c\in [1,z]_L$ with
  $a\geq_B c$, and write $b=a.e_y\in [1,y]_L$. Using
  Proposition~\ref{proposition.uniqueness_fix1}, $a\geq_B b$ and
  $a.e_z \geq_B c$. Furthermore, since $a$ is in the image set of
  $e_x$, one has $b.e_z = a.e_y.e_z = a.e_z\geq_B c$. Therefore,
  $[1,y]_L$ intersects $[c,a]_B$ at least in $b$.
  Hence, $y$ is an intermediate factor for $x,z$.

  For the reciprocal, take any $a\in [1,x]_L$. Since $\Mone$ preserves
  Bruhat order and is regressive, $a.e_y.e_z \leq_B a.e_z$. Set
  $c=a.e_z$, and take $b$ in $[c,a]_B \cap [1,y]_L$. Using
  Proposition~\ref{proposition.uniqueness_fix1},
  \begin{displaymath}
    a.e_y.e_z \geq_B b.e_z \geq_B c=a.e_z\,,
  \end{displaymath}
  and equality holds. Hence, $e_y$ is an intermediate factor for
  $e_x,e_y$: $e_xe_ye_z = e_xe_z$.

  The combinatorial reformulation of non-triviality for intermediate
  factors is then straightforward using
  Proposition~\ref{propositions.M1.idempotent.generated}.
\end{proof}

\begin{problem}
  Exploit the interrelations between left order and Bruhat order to
  find a more satisfactory combinatorial description of the quiver of
  $\K\Mone$.
\end{problem}

\subsubsection{Connection with the representation theory of the $0$-Hecke monoid}

Recall that the $0$-Hecke monoid $H_0(W)$ is a submonoid of $\Mzero(W)$. As a
consequence any $\K \Mzero(W)$-module is a $H_0(W)$-module and one can consider
the decomposition map $G_0(\Mzero(W)) \to G_0(H_0(W))$. It is given by the
following formula:
\begin{proposition}
  \label{proposition.restriction.M.H0}
  For $w\in W$, let $\Szero_w$ be the simple $\K \Mzero(W)$-module defined by
  \begin{equation*}
    \epsilon_w.f :=
    \begin{cases}
      \epsilon_w & \text{if $w.f=w$}, \\
      0          & \text{otherwise}.
    \end{cases}
  \end{equation*}
  Furthermore, for $J\subseteq I$, let $\SHzero_J$ be the simple
  $H_0(W)$-module defined by
  \begin{equation*}
    \mu_J.\pi_i :=
    \begin{cases}
      \mu_I      & \text{if $i \in J$}, \\
      0          & \text{otherwise}.
    \end{cases}
  \end{equation*}
  Then, the restriction of $\Szero_w$ to $H_0(W)$ is isomorphic to
  $\SHzero_{\Des(w)}$. The decomposition map $G_0(\Mzero(W)) \to
  G_0(H_0(W))$ is therefore given by $[\Szero_w] \mapsto [\SHzero_{\Des(w)}]$.
\end{proposition}
\begin{proof}
  By definition of the action, $w.\pi_i = w$ if and only if $i\in\Des(W)$.
\end{proof}

\subsubsection{The tower of $\Mone(\sg[n])$ monoids (type $A$)}

\begin{problem}
  The monoids $\Mone(\sg[n])$, for $n\in \N$, form a tower of monoids with the
  natural embeddings
  $\Mone(\sg[n])\times\Mone(\sg[m])\hookrightarrow\Mone(\sg[m+n])$. Due to the
  involution of Section~\ref{subsection.involution}, one has also embeddings
  $\Mzero(\sg[n])\times\Mzero(\sg[m])\hookrightarrow\Mzero(\sg[m+n])$.  As
  outlined in the introduction, it would hence be interesting to understand
  the induction and restriction functors in this setting, and in particular to
  describe the bialgebra obtained from the associated Grothendieck
  groups. This would give a representation theoretic interpretation of some
  bases of $\fqsym$.
\end{problem}

In this context, Proposition~\ref{proposition.restriction.M.H0} provides an
interpretation of the surjective coalgebra morphism
$\fqsym\twoheadrightarrow\qsym$, through the restriction along the
following commutative diagram of monoid inclusions (see
\cite{Duchamp_Hivert_Thibon.2002} for more details):
\begin{equation*}
  \begin{tikzpicture}[description/.style={fill=white,inner sep=2pt}]
    \matrix (m) [matrix of math nodes, row sep=3em, column sep=2.5em]{
      H_0(\sg[n]) \times H_0(\sg[m]) & \Mzero(\sg[n]) \times \Mzero(\sg[m]) \\
      H_0(\sg[n+m])                  & \Mzero(\sg[n+m]) \, .\\
    };
    \draw[right hook->] (m-1-1) -- (m-1-2);
    \draw[right hook->] (m-1-1) -- (m-2-1);
    \draw[right hook->] (m-1-2) -- (m-2-2);
    \draw[right hook->] (m-2-1) -- (m-2-2);
  \end{tikzpicture}
\end{equation*}

\section{Translation modules and $w$-biHecke algebras}
\label{section.translation_modules}

The main purpose of this section is to pave the ground for the
construction of the simple modules $S_w$ of the biHecke monoid
$\biheckemonoid:=\biheckemonoid(W)$ in
Section~\ref{subsection.M.simple_modules}.

As for any aperiodic monoid, each such simple module is associated
with some regular $\JJ$-class $D$ of the monoid, and can be
constructed as a quotient of the span $\K\RR(f)$ of the $\RR$-class
of any idempotent $f$ in $D$, endowed with its natural right
$\K M$-module structure (see Section~\ref{ss.representation_theory_monoid}).

In Section~\ref{subsection.translation}, we endow the interval
$[1,w]_R$ with a natural structure of a combinatorial
$\K \biheckemonoid$-module $T_w$, called \emph{translation module}, and
show that, for any $f\in \biheckemonoid$, regular or not, the
right $\K M$-module $\K\RR(f)$ is always isomorphic to some $T_w$.

The translation modules will play an ubiquitous role for the
representation theory of $\K\biheckemonoid$ in
Section~\ref{section.M.representation_theory}: indeed $T_w$ can be
obtained by induction from the simple modules $S_w$ of
$\K\biheckemonoid$, and the right regular representation of
$\K\biheckemonoid$ admits a filtration in terms of the $T_w$
which mimics the composition series of the right regular
representation of $\K\Mzero$ in terms of its simple modules
$S_w$. Reciprocally $T_w$, and therefore the right regular
representation of $\K\biheckemonoid$, restricts naturally to $\Mzero$.
Finally, $T_w$ is closely related to the projective module $P_w$ of
$\K\biheckemonoid$ (Corollary~\ref{corollary.translation_module}).

By taking the quotient of $\K\biheckemonoid$ through its
representation on $T_w$, we obtain a $w$-analogue
$\wbiheckealgebra{w}$ of the biHecke algebra $\biheckealgebra$. This
algebra turns out to be interesting in its own right, and we proceed
by generalizing most of the results
of~\cite{Hivert_Thiery.HeckeGroup.2007} on the representation theory
of $\biheckealgebra$.

As a first step, we introduce in
Section~\ref{subsection.translation.P} a collection of submodules
$P_J^{(w)}$ of $T_w$, which are analogues of the projective modules of
$\biheckealgebra$. Unlike for $\biheckealgebra$, not any subset $J$ of
$I$ yields such a submodule, and this is where the combinatorics of
the blocks of $w$ as introduced in Section~\ref{section.blocks_cutting_poset}
enters the game. In a second step, we derive in
Section~\ref{subsection.translation.triangular} a lower bound on the
dimension of $\wbiheckealgebra{w}$; this requires a (fairly involved)
combinatorial construction of a family of functions on $[1,w]_R$ which
is triangular with respect to Bruhat order.
In Section~\ref{subsection.bihecke} we combine these results to derive
the dimension and representation theory of $\wbiheckealgebra{w}$:
projective and simple modules, Cartan matrix, quiver, etc
(see Theorem~\ref{theorem.wBihecke.representations}).

\subsection{Translation modules and $w$-biHecke algebras}
\label{subsection.translation}

The goal of this section is to study the combinatorics of the right class
modules for the biHecke monoid, in particular a combinatorial
model for them. Indeed, we show that the right class modules correspond to uniform translations of
image sets, hence the name ``translation modules''.

Fix $f\in \biheckemonoid$. Recall from
Definition~\ref{definition.right_class_module} that the right class module
associated to $f$ is defined as the quotient
\begin{equation*}
  \KRR(f) := \K f \biheckemonoid / \KRR_<(f)\,.
\end{equation*}
The basis of $\KRR(f)$ is the right class $\RR(f)$ which is described in
Proposition~\ref{proposition.characterization_R_class}.
Recall from there that $f_u$ denote the unique element of $\biheckemonoid(W)$
such that $\fibers(f_u)=\fibers(f)$ and $\one.f_u=u$.
\begin{proposition} \label{proposition.combinatorial_translation}
  Set $w=\type(f)^{-1}w_0$. Then $(f_u)_{u\in [1,w]_R}$ forms a basis
  of $\KRR(f)$ such that:
  \begin{equation}
  \label{equation.action_pi_translation}
  \begin{aligned}
    f_u.\pi_i &=
    \begin{cases}
      f_u       & \text{if $i\in \Des(u)$}\\
      f_{u s_i} & \text{if $i\not\in \Des(u)$ and $us_i\in [1,w]_R$}\\
      0         & \text{otherwise;}
    \end{cases}\\
    f_u.\opi_i &=
    \begin{cases}
      f_{u s_i} & \text{if $i\in \Des(u)$}\\
      f_u       & \text{if $i\not\in \Des(u)$ and $us_i\in [1,w]_R$}\\
      0         & \text{otherwise.}
    \end{cases}
  \end{aligned}
  \end{equation}
  In particular, the action of any $g\in \biheckemonoid$ on a basis
  element $f_u$ of the right class module either annihilates $f_u$ or
  agrees with the usual action on $W$: $f_u.g = f_{u.g}$.
\end{proposition}

\begin{proof}
  By Definition~\ref{definition.right_class_module} and
  Proposition~\ref{proposition.characterization_R_class}, $(f_u)_{u\in [1,w]_R}$ form a
  basis of $\KRR(f)$.

  The action of $\pi_i$ agrees with right multiplication, except when the index $v$ of the
  new $f_v$ is no longer in $[1,w]_R$, in which case the element is annihilated.
  The action of $\opi_i$ also agrees with right multiplication. However, due to the relations
  $\pi_i \opi_i = \opi_i$ and $\opi_i \pi_i=\pi_i$, we need that $\opi_i$ annihilates
  $f_u$ if $i\not\in \Des(u)$ and $us_i \not \in [1,w]_R$.

  The last statement follows by induction writing $f \in \biheckemonoid$ in terms of the
  generators $\pi_i$ and $\opi_i$ and using~\eqref{equation.action_pi_translation}.
\end{proof}

Proposition~\ref{proposition.combinatorial_translation} gives a combinatorial model for
right class modules. It is clear that two functions with the same type yield
isomorphic right class modules. The converse also holds:

\begin{proposition}
\label{proposition.isom_translation}
  For any $f,f'\in \biheckemonoid$, the right class modules $\KRR(f)$ and $\KRR(f')$ are isomorphic
  if and only if $\type(f) = \type(f')$.
\end{proposition}
\begin{proof}
  By Proposition~\ref{proposition.combinatorial_translation}, it is clear that if $\type(f)=\type(f')$,
  then $\KRR(f) \cong \KRR(f')$.

  Conversely, suppose $\type(f)\neq \type(f')$. Then we also have $w\neq w'$, where
  $w=\type(f)^{-1} w_0$ and $w'=\type(f')^{-1} w_0$. Without loss of generality, we may assume
  that $\ell(w)\ge \ell(w')$. Using the combinatorial model
  of Proposition~\ref{proposition.combinatorial_translation}, we then have
  \begin{equation*}
  	f_1.\pi_w = f_w \neq 0 \qquad \text{and} \qquad f_1'.\pi_w = 0 \; ,
  \end{equation*}
  so that $\KRR(f)\not \cong \KRR(f')$.
\end{proof}

Note that it is not obvious from the combinatorial action of $\pi_i$ and $\opi_i$ of
Proposition~\ref{proposition.combinatorial_translation} that the result indeed gives a module.
However, since it agrees with the right action on the quotient space as in
Definition~\ref{definition.right_class_module}, this is true.
By Proposition~\ref{proposition.isom_translation}, we may choose a canonical representative
for right class modules.
\begin{definition}
\label{definition.translation_module}
The module $T_w := \KRR(e_{w,w_0})$ for all $w\in W$ is called the
\emph{translation module associated to $w$}. We identify its basis with
$[1,w]_R$ via $u\mapsto f_u$, where $f=e_{w,w_0}$.
\end{definition}
For the remainder of this section for $f\in \biheckemonoid$ and $u\in [1,w]_R$, unless
otherwise specified, $u.f$ means the action of $f$ on $u$ in the translation module $T_w$.

\begin{definition}
  The $w$-biHecke algebra $\wbiheckealgebra{w}$ is the natural
  quotient of $\K\biheckemonoid$ through its representation on
  $T_w$.  In other words, it is the subalgebra of $\End(T_w)$
  generated by the operators $\pi_i$ and $\opi_i$ of
  Proposition~\ref{proposition.combinatorial_translation}.
\end{definition}

\subsection{Left antisymmetric submodules}
\label{subsection.translation.P}

By analogy with the simple reflections in the Hecke group algebra, we
define for each $i\in I$ the operator $s_i := \pi_i + \opi_i -1$.
For $u\in [1,w]_R$, the action on the translation module $T_w$ is given by
\begin{equation}
  u.s_i=
  \begin{cases}
    us_i & \text{ if $us_i\in [1,w]_R$,}\\
    -u   & \text{ otherwise.}
  \end{cases}
\end{equation}
These operators are still involutions, but do not always satisfy the braid relations.

\begin{example}
Take $W$ of type $A_2$ and $w=s_1$. The translation module $T_w$ has two
basis elements $B=(1,s_1)$ and the matrices for $s_1$ and $s_2$ on this basis are given by
\begin{equation*}
	s_1 = \begin{pmatrix} 0&1\\1&0\end{pmatrix} \qquad \text{and} \qquad
	s_2 = \begin{pmatrix} -1&0\\0&-1\end{pmatrix} \; .
\end{equation*}
It is not hard to check that then $s_1s_2s_1 \neq s_2 s_1 s_2$.
\end{example}

Similarly, one can define operators $\ls[s_i]$ acting on the left on the translation module $T_w$:
\begin{equation}
  \ls[s_i].u =
  \begin{cases}
    s_iu & \text{ if $s_i u \in [1,w]_R$,}\\
    -u   & \text{ otherwise.}
  \end{cases}
\end{equation}

\begin{definition}
  For $J\subseteq I$, set $P^{(w)}_J := \{ v\in T_w \suchthat \ls[s_i].v=-v\quad
  \forall i \in J\}$.
\end{definition}
For $w=w_0$, these are the projective modules $P_J$ of the biHecke
algebra~\cite{Hivert_Thiery.HeckeGroup.2007}.

\begin{proposition}
  \label{proposition.cond_Jtiling}
  Take $w\in W$ and $J\subseteq I$. Then, the following are
  equivalent:
  \begin{enumerate}[(i)]
  \item \label{enumerate.tiling} $\lcoset{w}{J}$ is a cutting point of $w$;
  \item \label{enumerate.submodule} $P^{(w)}_J$ is an $\K \biheckemonoid$-submodule of $T_w$.
  \end{enumerate}
  Furthermore, when any, and therefore all, of the above hold,
  $P^{(w)}_J$ is isomorphic to $T_{\lcoset{w}{J}}$, and its basis is
  indexed by $[1, {\lcoset{w}{J}}]_R$, that is, assuming $J\in
  \JJs{w}$, $\{v \in [1,w]_R, J\subset \Jblock{w}(v)\}$.
\end{proposition}

\begin{proof}
  \eqref{enumerate.submodule} $\Rightarrow$ \eqref{enumerate.tiling}:
  Set
  \begin{equation*}
      v^{w}_J := \sum_{u\in [1,w_J]_R} (-1)^{\len(u)-\len(w_J)} u.
  \end{equation*}
  Up to a scalar factor, this is the unique vector in $P^{(w)}_J$ with
  support contained in $[1,w_J]_R$. Then,
  \begin{equation*}
  \begin{split}
    v^{w}_J.\pi_{\lcoset{w}{J}} &= \sum_{\substack{u\in [1,w_J]_R\\ \text{s.t. } u\lcoset{w}{J} \in [1,w]_R}}
    (-1)^{\len(uv)-\len(w_Jv)}u\lcoset{w}{J},\\
    v^{w}_J.\pi_v\opi_{v^{-1}} &= \sum_{\substack{u\in [1,w_J]_R\\ \text{s.t. } u\lcoset{w}{J} \in [1,w]_R}}
    (-1)^{\len(u)-\len(w_J)}u.
  \end{split}
  \end{equation*}
  Therefore, if $\lcoset{w}{J} \not\leq_R w$, then
  $v^{w}_J.\pi_{\lcoset{w}{J}}\opi_{\lcoset{w}{J}^{-1}}$ is a nonzero
  vector with support strictly included in $[1,w_J]_R$ and therefore
  not in $P_J^{(w)}$. By Proposition~\ref{proposition.reduced_block}
  this proves \eqref{enumerate.submodule} $\Rightarrow$
  \eqref{enumerate.tiling}.

  \eqref{enumerate.tiling} $\Rightarrow$ \eqref{enumerate.submodule}:
  If \eqref{enumerate.tiling} holds, then the action of $\pi_i$
  (resp. $\opi_i$) on $v^{w}_J.\pi_{v}$ either leaves it unmodified,
  kills it (if $vs_i=s_jv$ for some $j$) or maps it to
  $v^{w}_J.\pi_{vs_i}$. The vectors $(v^{w}_J.\pi_{v})_{v\in
    [1,\lcoset{w}{J}]_R}$ form a basis of $P^{(w)}_J$ which is stable
  by $\biheckemonoid$.

  The last statement follows straightforwardly.
\end{proof}

It is clear from the definition that for $J_1, J_2 \subseteq I$,
$P^{(w)}_{J_1\cup J_2} = P^{(w)}_{J_1} \cap P^{(w)}_{J_2}$.  Since the
set $\RJs{w}$ of left blocks of $w$ is stable under union, the set of
$\K \biheckemonoid$-modules $(P^{(w)}_J)_{J\in \RJs{w}}$ is stable under intersection.
On the other hand, unless $J_1$ and $J_2$ are comparable,
$P^{(w)}_{J_1\cup J_2}$ is a strict subspace of $P^{(w)}_{J_1} +
P^{(w)}_{J_2}$. This motivates the following definition.

\begin{definition}
  \label{definition.Sw}
  For $J\in \JJs{w}$, we define the module
  \begin{equation}
    S^{(w)}_J:=P^{(w)}_J / \sum_{J'\supsetneq J, J' \in \RJs{w}} P^{(w)}_{J'},
  \end{equation}
\end{definition}
\begin{remark}
  \label{remark.Sw}
  By the last statement of Proposition~\ref{proposition.cond_Jtiling},
  and the triangularity of the natural basis of the modules
  $P^{(w)}_{J'}$, the basis of $S^{(w)}_J$ is given by
  \begin{equation}
    [1,\lcoset{w}{J}]_R \backslash \bigcup_{v \cut \lcoset{w}{J}} [1,v]_R =
    \{v \in [1,w]_R, J\subset \Jblock{w}(v)\}
    \,.
  \end{equation}
\end{remark}

\subsection{A (maximal) Bruhat-triangular family of the $w$-biHecke algebra}
\label{subsection.translation.triangular}

Consider the submonoid $F$ in $\wbiheckealgebra{w}$ generated by the
operators $\pi_i$, $\opi_i$, and $s_i$, for $i\in I$. For $f\in F$ and
$u\in[1,w]_R$, we have $u.f = ± v$ for some $v\in[1,w]_R$. For our
purposes, the signs can be ignored and $f$ be considered as a function
from $[1,w]_R$ to $[1,w]_R$.

\begin{definition}
  For $u,v \in [1,w]_R$, a function $f\in F$ is called
  \emph{$(u,v)$-triangular} (for Bruhat order) if $v$ is the unique
  minimal element of $\im(f)$ and $u$ is the unique maximal element of
  $f^{-1}(v)$ (all minimal and maximal elements in this context are
  with respect to Bruhat order).
\end{definition}

Recall the notion of maximal reduced right block $\Kblock{w}(u)$ of
Definition~\ref{definition.max_block}.

\begin{proposition}
  \label{proposition.triangular}
  Take $u,v \in [1,w]_R$ such $\Kblock{w}(u) \subseteq
  \Kblock{w}(v)$. Then, there exists a $(u,v)$-triangular function
  $f_{u,v}$ in $F$.
\end{proposition}
For example, for $w=4312$ in $\sg[4]$, the condition on $u$ and $v$ is
equivalent to the existence of a path from $u$ to $v$ in the digraph
$G^{(4312)}$ (see Figure~\ref{figure.4312} and Section~\ref{subsection.bihecke}).

The proof of Proposition~\ref{proposition.triangular} relies on several remarks and
lemmas that are given in the sequel of this section.
The construction of $f_{u,v}$ is explicit, and the triangularity
derives from $f_{u,v}$ being either in $\biheckemonoid$, or close enough to be
bounded below by an element of $\biheckemonoid$.  It follows from the upcoming
Theorem~\ref{theorem.wBihecke.representations} that the condition on
$u$ and $v$ is not only sufficient but also necessary.

\begin{remark}
  \label{remark.triangular.composition}
  If $f$ is $(u,v)$-triangular and $g$ is $(v,v')$-triangular, then
  $fg$ is $(u,v')$-triangular.
\end{remark}

\begin{remark}
  \label{remark.bound_s}
  Take $x\in [1,w]_R$ and let $i\in I$. Then, $x.\opi_i \leq_R x.s_i$.

  By repeated application, for $S\subseteq I$, and $i_1,\dots,i_k\in S$,
  $x.\opi_S \leq_R x.s_{i_1}\cdots s_{i_k}$, where recall that $\opi_S$ is
  the longest element in the generators $\{\opi_j \mid j\in S\}$.
\end{remark}

\begin{lemma}
  \label{lemma.triangular.uu}
  Take $u\in [1,w]_R$, and define $f_{u,u} :=
  e_{u,w_0}=\opi_{u^{-1}}\pi_{u}$. Then:
  \begin{enumerate}[(i)]
  \item \label{i.triangular}
  	$f_{u,u}$ is $(u,u)$-triangular.
  \item \label{i.bruhat}
  	For $v\in[1,w]_R$, either $v.f_{u,u}=0$ or $v.f_{u,u}\geq_B v$.
  \item \label{i.image}
  	$\im(f_{u,u}) = [u,w_0]_L \cap [1,w]_R$.
  \end{enumerate}
\end{lemma}

\begin{proof}
  First consider the case $w=w_0$. Then, \eqref{i.bruhat} and~\eqref{i.image} hold by
  Lemma~\ref{proposition.uniqueness_fix1}.

  Now take any $w\in W$. By Proposition~\ref{proposition.combinatorial_translation}
  the action of $f\in \biheckemonoid$ on the translation module $T_w$ either agrees with the
  action on $W$ or yields $0$. Hence in particular Proposition~\ref{proposition.bruhat}
  still applies, which yields~\eqref{i.bruhat}. This also implies the inclusion
  $\im(f_{u,u}) \backslash \{0\} \subset [u,w_0]_L \cap [1,w]_R$. The reverse inclusion is
  straightforward: if $u' = xu$, then $u'.f_{u,u}= xu.\opi_{u^{-1}}\pi_{u}=x\pi_{u}=xu=u'$.
  Therefore~\eqref{i.image} holds as well.

  Finally, \eqref{i.image} implies that $u$ is the unique minimal element of
  $\im(f_{u,u})$, and~\eqref{i.bruhat} implies that $u$ is the unique maximal element in
  $f_{u,u}^{-1}(u)$; therefore~\eqref{i.triangular} holds.
\end{proof}

\begin{lemma}
  \label{lemma.triangular.down}
  If $u>_R v$, then $f_{u,v} := f_{u,u}\opi_{u^{-1}v}$ is $(u,v)$-triangular.
\end{lemma}

\begin{proof}
  By Lemma~\ref{lemma.triangular.uu}~\eqref{i.image}, the image set of $f_{u,u}$
  is a subset of $[u,w_0]_L$. Therefore, by Remark~\ref{remark.interval}, $\opi_{u^{-1}v}$
  translates it isomorphically to the interval $[v,w_0u^{-1}v]_L$. In particular,
  the fibers are preserved: $f_{u,v}^{-1}(v) = f_{u,u}^{-1}(u)$, and
  the triangularity of $f_{u,v}$ follows.
\end{proof}

\begin{lemma}
  \label{lemma.triangular.up}
  Take $u\in [1,w]_R$. Then, either $u$ is a cutting point of $w$, or
  there exists a $(u,v)$-triangular function $f_{u,v}$ in $F$ with $u<_R
  v\leq_R w$.
\end{lemma}
\begin{proof}
  Let $J$ be the set of short nondescents $i$ of $u$, and set $V:=U_u
  \cap [1,w]_R$ (recall from Definition~\ref{definition.short_nondescents} that
  $U_u := uW_J$). By Proposition~\ref{proposition.cutting.U}, $V$ is the set of $w'\in [1,w]_R$
  such that $u\cuteq w'$. Furthermore, $V$ is a lattice (it is the
  intersection of the two lattices $(uW_J, <_R)$ and $[1,w]_R$) with
  $u$ as unique minimal element; in particular, $V\subset
  [u,w]_R$.

  If $w\in V$ (which includes the case $u=w$ and $J=\{\}$), then $u$
  is a cutting point for $w$ and we are done.

  Otherwise, consider a shortest sequence $i_1,\dots,i_k$ such that
  $\{i_1,\dots,i_k\}$ does not intersect $\Des(u)$, and $v' = u s_{i_1}\cdots
  s_{i_k}\not \in V$. Such a sequence must exist since $w\not \in
  V$. Set $S:=\{i_1,\dots,i_k\}$. Note that $i_1,\dots,i_{k-1}$ are in
  $J$ but $i_k$ is not. Furthermore, $u\not \cuteq v'$ while $u = v'^S$
  because $v'\in u W_S$ and $S\cap\Des(u)=\emptyset$.

  Case 1: $v'\in \im(f_{u,u})$. Then, $u<_L v'$. Combining this with
  $u=v'^S$ yields that $u\cuteq v'$, a contradiction.

  Case 2: $v'\not\in \im(f_{u,u})$. Set $v:= us_{i_1}$, and define
  $f_{u,v}:=f_{u,u} \sigma \pi_{i_1}$, where
  \begin{equation}
    \sigma := s_{i_2}\cdots s_{i_{k-1}}s_{i_k} s_{i_{k-1}}\cdots s_{i_2}\,.
  \end{equation}
  Note that for $k=1$, we have $\sigma=\one$.
  We now prove that $f_{u,v}$ is $(u,v)$-triangular.

  First, we consider the fiber $f_{u,v}^{-1}(v)$. By minimality of
  $k$, and up to sign, $s_{i_k}$ fixes all the elements of $V$ at
  distance at most $k-2$ of $u$. Hence, $\sigma^{-1}(u) =
  u$. Simultaneously,
  \begin{equation}
    v.\sigma^{-1}= v.s_{i_2}\cdots s_{i_{k-1}}s_{i_k} s_{i_{k-1}}\cdots s_{i_2}
    = v'.s_{i_{k-1}}\cdots s_{i_2} \in v'W_J\,.
  \end{equation}
  Hence, $v.\sigma^{-1}\not\in \im(f_{u,u})$ because $v'\not\in \im(f_{u,u})$
  and $\im(f_{u,u})$ is stable under right multiplication by $s_j$ for
  $j\in J$.
  Putting everything together:
  \begin{equation}
    f_{u,v}^{-1}(v)
    = f_{u,u}^{-1}(\sigma^{-1}(\pi_{i_1}^{-1}(v)))
    = f_{u,u}^{-1}(\sigma^{-1}(\{u,v\}))
    = f_{u,u}^{-1}(\{u\})
    = [1,u]_B \cap [1,w]_R\,.
  \end{equation}
  Therefore, $u$ is the unique length maximal element of
  $f_{u,v}^{-1}(v)$, as desired.

  We take now $x\in \im(f_{u,u})$, and apply
  Proposition~\ref{proposition.bruhat} repeatedly. To start with:
  \begin{equation}
    u = \one.f_{u,u} \leq_B x.f_{u,u}\,.
  \end{equation}
  Using Remark~\ref{remark.bound_s}:
  \begin{equation}
    u = u.\opi_S \leq_B (x.f_{u,u}).\opi_S \leq_B (x.f_{u,u}).\sigma =
    x.f_{u,u}.\sigma\,.
  \end{equation}
  It follows that:
  \begin{equation}
    v = u.\pi_{i_1} \leq_B  (x.f_{u,u}.\sigma).\pi_{i_1} = x.f_{u,v}\,.
  \end{equation}
  In particular, $v$ is the unique Bruhat minimal element of $\im(f_{u,v})$,
  as desired.
\end{proof}

\begin{proof}[Proof of Proposition~\ref{proposition.triangular}]
  Since $W$ is finite, repeated application of
  Lemma~\ref{lemma.triangular.up} yields a finite sequence of
  triangular functions
  \begin{equation*}
    f_{u,u_1},\dots,f_{u_{k-1},u_k}\,,\quad\text{where}\quad
    u<_R u_1 <_R \dots <_R u_k
  \end{equation*}
  and $u_k$ is a cutting point $w^J$ of $w$. Since
  $u<_R w^J$, one has $J\subset \Kblock{w}(u) \subset
  \Kblock{w}(v)$, and therefore $u_k=w^J >_R v$. Then, applying
  Lemma~\ref{lemma.triangular.down} one can construct a
  $(u_k,v)$-triangular function $f_{u_k,v}$. Finally, by
  Remark~\ref{remark.triangular.composition}, composing all these
  triangular functions gives a $(u,v)$-triangular function
  $f_{u,u_1}\cdots f_{u_{k-1},u_k}f_{u_k,v}$.
\end{proof}

\subsection{Representation theory of the $w$-biHecke algebra}
\label{subsection.bihecke}

Consider the digraph $G^{(w)}$ on $[1,w]_R$ with an edge $u\mapsto v$
if $u=vs_i$ for some $i$ and $\Jblock{w}(u) \subseteq \Jblock{w}(v)$.
Up to orientation, this is the Hasse diagram of right order
(see for example Figure~\ref{figure.4312}).
The following theorem is a generalization of~\cite[Section 3.3]{Hivert_Thiery.HeckeGroup.2007}.
\begin{theorem}
  \label{theorem.wBihecke.representations}
  $\wbiheckealgebra{w}$ is the maximal algebra stabilizing all modules $P_J^{(w)}$ for
  $J\in \RJs{w}$
  \begin{displaymath}
    \wbiheckealgebra{w} = \{ f\in \End(T_w) \suchthat f(P_J^{(w)}) \subseteq P_J^{(w)} \}
  \end{displaymath}

  The elements $f_{u,v}$ of Proposition~\ref{proposition.triangular}
  form a basis $\wbiheckealgebra{w}$; in particular,
  \begin{equation}
    \dim \wbiheckealgebra{w} = |\{ (u,v) \suchthat \Jblock{w}(u)
    \subseteq \Jblock{w}(v) \}| \; .
  \end{equation}

  $\wbiheckealgebra{w}$ is the digraph algebra of the graph $G^{(w)}$.

  The family $(P^{(w)}_J)_{J\in \RJs{w}}$ forms a complete system of
  representatives of the indecomposable projective modules of
  $\wbiheckealgebra{w}$.

  The family $(S^{(w)}_J)_{J\in \RJs{w}}$ forms a complete system of
  representatives of the simple modules of $\wbiheckealgebra{w}$. The
  dimension of $S^{(w)}_J$ is the size of the corresponding $w$-nondescent class.

  $\wbiheckealgebra{w}$ is Morita equivalent to the poset algebra of
  the lattice $[1,w]_\cuteq$. In particular, its Cartan matrix is the
  incidence matrix and its quiver the Hasse diagram of this lattice.
\end{theorem}

\begin{proof}
  From Proposition~\ref{proposition.triangular}, one derives by
  triangularity that $\dim \wbiheckealgebra{w} \geq \{ (u,v) \suchthat
  \Kblock{w}(u) \subseteq \Kblock{w}(v) \}$.
  The stability of all the
  subspaces $P^{(w)}_J$ imposes the converse equality. Hence,
  $\wbiheckealgebra{w}$ is exactly the subalgebra of $\End(T_w)$
  stabilizing each $P^{(w)}_J$. The remaining statements follow
  straightforwardly, as in~\cite[Section 3.3]{Hivert_Thiery.HeckeGroup.2007}.
  See also
  e.g.~\cite[Section 3.7.4]{Denton_Hivert_Schilling_Thiery.JTrivialMonoids}
  for the Cartan matrix and quiver of a poset algebra.
\end{proof}

\section{Representation theory of $\biheckemonoid(W)$}
\label{section.M.representation_theory}

In this section, we gather all results of the preceding sections in
order to describe the representation theory of
$\biheckemonoid:=\biheckemonoid(W)$. The main result is
Theorem~\ref{theorem.simple} which gives the simple modules of
$\K \biheckemonoid$. We further relate the representation theory of
$\K \biheckemonoid$ to the representation theory of $\K \Mzero$. In
particular, we prove that the translation modules are exactly the
modules induced by the simple modules of $\K \Mzero$. We then conclude by
computing some characters and the decomposition map from
$\K \biheckemonoid$ to $\K \Mzero$.

\subsection{Simple modules}
\label{subsection.M.simple_modules}

We now study the simple modules of the biHecke monoid $\K \biheckemonoid$
and also show that the translation modules are indecomposable.

\begin{theorem}\mbox{}
  \label{theorem.simple}
  \begin{enumerate}[(i)]
  \item \label{item.simples_of_M}
    The biHecke monoid $\biheckemonoid$ admits $|W|$ non-isomorphic
    simple modules $(S_w)_{w\in W}$ (resp.  projective indecomposable
    modules $(P_w)_{w\in W}$).
  \item \label{item.simples_translation}
    The simple module $S_w$ is isomorphic to the top simple module
    \begin{equation*}
    	S^{(w)}_{\{\}} = T_w / \sum_{v\cut w} T_v
    \end{equation*}
    of the translation module $T_w$. Its dimension is given by
    \begin{equation*}
      \dim S_w = \Big| [1,w]_R \backslash \bigcup_{v \cut w} [1,v]_R \Big| \; .
    \end{equation*}
    In general, the simple quotient module $S^{(w)}_{J}$ of $T_w$ is
    isomorphic to $S_{\lcoset{w}{J}}$ of $\biheckemonoid$.
  \end{enumerate}
\end{theorem}

\begin{proof}
  Since $\biheckemonoid$ is aperiodic
  (Proposition~\ref{proposition.pseudo_idempotent}), we may apply the
  special form of Clifford, Munn, and Ponizovski\v \i's construction
  of the simple modules (see
  Theorem~\ref{theorem.simple_modules.aperiodic}). Namely, the simple
  modules are indexed by the regular $\JJ$- classes of
  $\biheckemonoid$; by Corollary~\ref{corollary.transversal}, there
  are $|W|$ of them. Using that, for any finite-dimensional algebra,
  the simple and indecomposable projective modules share the same
  indexing set (see~\cite[Corollary 54.14]{Curtis_Reiner.1962}), this
  yields \eqref{item.simples_of_M}.

  Clifford, Munn, and Ponizovski{\v \i} further construct $S_w$ as the
  top of the right class modules, that is in our case, of the
  translation module $T_w$. Our explicit description of the radical of
  $T_w$ as $\sum_{v\cut w} T_v$ in~\eqref{item.simples_translation} is
  a straightforward application of
  Theorem~\ref{theorem.wBihecke.representations}. The dimension
  formula follows using Remark~\ref{remark.Sw}.
\end{proof}

For a direct proof of the equality $\rad T_w = \sum_{v\cut w} T_v$,
without using Theorem~\ref{theorem.wBihecke.representations}, one
would want to show that $\sum_{v\cut w} T_v$ is exactly the
annihilator of $\JJ(e_{w,w_0})$. One inclusion is easy, thanks to the
following remark.
\begin{remark}
  The submodule $T_v$ is annihilated by $\JJ(e_{w,w_0})=\JJ(\opi_w)$.
\end{remark}
\begin{proof}
  Fix $w$ and take $v$ such that $v\cut w$. Then $\opi_w$ annihilates $T_v\subset
  T_w$. Indeed, combining $\opi_w(w)=1$ with
  Propositions~\ref{proposition.combinatorial_translation}
  and~\ref{proposition.bruhat}, one obtains that $\opi_w$ either
  annihilates $f_u$ or maps it to $f_1$. Take now $x\in T_v$, and
  write $x.\opi_w=\lambda f_1$. Since $T_v$ is a submodule, $\lambda
  f_1$ lies in $T_v$; however the basis elements of $T_v$ have
  disjoint support and since $v\cut w$ none of them are collinear to
  $f_1$. Therefore $x.\opi_w=0$.
\end{proof}

\begin{table}
  \newcommand{\petit}[1]{{\scriptstyle #1}}
  \begin{displaymath}
    \begin{array}{|l|r|r|r|r|r|}
      \hline
      \text{Type} & |W| & |\Mzero| & |\biheckemonoid|  &\!(\dim S_w)_w\!&
      \sum_w\!\dim S_w\\\hline
      A_0        &   1&      1 &       1 & 1                  &       1 \\\hline
      A_1        &   2&      2 &       3 & 1^2                &       2 \\\hline
      A_2=I_2(3) &   6&      8 &      23 & 1^42^2          &       8 \\\hline
      A_3        &  24&     71 &     477 & 1^82^43^44^65^2    &      62 \\\hline
      A_4        & 120&   1646 &   31103 & 1^{16}2^{10}3^84^{16}5^{16}6^6%
                                                 \cdots20^6       &     770\\\hline
      A_5        & 720& 118929 & 7505009 & 1^{32}2^{24}3^{20}4^{42}5^{38}6^{40}%
                                                 \cdots120^{2}   & 13080\\\hline  %

      B_2=I_2(4) &   8&     14 &      49 & 1^42^23^2       &      14\\\hline
      B_3        &  48&    498 &    5455 & 1^8 2^4 3^4 4^6 5^7 6^4 7^4 8^4 9^1 10^2 11^2 12^2 &     246\\\hline
      B_4        & 384& 149622 &6664553 & 1^{16}2^{10}3^{10}4^{14}5^{17}6^{16}%
                                                 \cdots80^{2} & 6984\\\hline
      G_2=I_2(6) &  12&     32 &     153 & 1^42^23^24^25^2 &      32\\\hline
      H_3        & 120&     87 &    1039 & 1^{8}2^{4}3^{4}4^{8}5^{6}6^{7}%
                                                 \cdots36^{2} &      1404\\\hline
      A_1\times A_1& 4&    4 &       9 & 1^21^2             &       4\\\hline
      I_2(p)     &  2p&p^2\!-\!p\!+\!2&\frac23 p^3\!+\!\frac43p\!+\!1& 1^42^2\cdots(p-1)^2&p^2\!-\!p\!+\!2\\\hline
    \end{array}
  \end{displaymath}
  \caption{Statistics on the biHecke monoids $\biheckemonoid:=\biheckemonoid(W)$ for
    the small Coxeter groups. In column four, $1^82^4\cdots5^2$ means that there are
    $8$ simple modules of dimension $1$, $4$ of dimension $2$, and so
    on. The sequence $p^2\!-\!p\!+\!2$ is \#A014206 in~\cite{OEIS}.}
  \label{table.dim_simple}
\end{table}

\begin{example}
  \label{example.S4312}
The simple module $S_{4312}$ is of dimension $3$, with
basis indexed by $\{4312,4132,1432\}$ (see
Figure~\ref{figure.4312}). The other simple modules $S_{3412}$,
$S_{4123}$, and $S_{1234}$ are of dimension $5$, $3$, and $1$,
respectively. See also Table~\ref{table.dim_simple}.
\end{example}

In general, the two extreme cases are, on the one hand, when $w$ is
the maximal element of a parabolic subgroup, in which case the simple
module is of dimension $1$ and, on the other hand, when $w$ is an
immediate successor of $\one$ in the cutting poset (see
Example~\ref{example.cutting_poset.almost_minimal}), in which case the
simple module is of dimension $|T_w|-1$.
In the other cases, one can use Theorem~\ref{theorem.cutting} to
calculate the dimension of $S_w$ by inclusion-exclusion from the sizes
of the intervals $[1,\lcoset{w}{J}]_R$ where $\lcoset{w}{J}$ runs
through the free sublattice at the top of the interval $[1,w]_\cuteq$
of the cutting poset.
Note that the sizes of the intervals in $W$ can also be computed by a
similar inclusion-exclusion (the M\"obius function for right order is
given by $\mu(u,w)=(-1)^k$ if the interval $[u,w]_R$ is isomorphic to
some $W_J$ with $|J|=k$, and $0$ otherwise). This may open the door for
some generating series manipulations to derive statistics like the sum
of the dimension of the simple modules.

\begin{corollary}
  \label{corollary.translation_module}
  The translation module $T_w$ is an indecomposable
  $\K \biheckemonoid$-module, quotient of the projective module $P_w$
  of $\K \biheckemonoid$.
\end{corollary}
\begin{proof}
  Direct application of
  Corollary~\ref{corollary.right_class_modules_quotient_of_projective.aperiodic}
\end{proof}

\subsection{From $\Mzero(W)$ to $\biheckemonoid(W)$}
\label{subsection.translation.M1}
In this section, we use our knowledge of $\Mzero$ to learn more about $\biheckemonoid$.
\begin{proposition}
  \label{proposition.translation_module_induced_from_M0}
  The translation module $T_w$ is isomorphic to the induction to
  $\K \biheckemonoid$ of the simple module $\Szero_w$ of $\K \Mzero$.
\end{proposition}

The proof of this proposition follows from the upcoming lemmas giving
some simple conditions on a general inclusion of monoids $B\subseteq A$
under which the (regular) right class modules of $\K A$ are induced from
those of $\K B$.

\begin{lemma}
  \label{lemma.induction.algebrapart}
  Let $B \subseteq A$ be two finite monoids and $f\in B$. If
  $$\KRR^A_<(f)=\KRR^B_<(f)A\,,$$ then the right class module $\KRR^A(f)$ is
  isomorphic to the induction from $\K B$ to $\K A$ of the right class module
  $\KRR^B(f)$:
  \begin{equation*}
    \KRR^A(f) \cong \induced{\KRR^B(f)}{\K B}{\K A}\,.
  \end{equation*}
\end{lemma}
\begin{proof}
  Recall that for a $\K B$-module $Y$, the module $\induced{Y}{\K B}{\K A}$
  induced by $Y$ from $\K B$ to $\K A$ is given by $\induced{Y}{\K B}{\K A} :=
  Y\tensorBA$.

  By construction of the right class modules (see
  Definition~\ref{definition.right_class_module}), we have the following exact
  sequences:
  \begin{gather}
    \label{equ.B.exact}
    0 \rightarrow \KRR^B_<(f) \rightarrow \K fB \rightarrow \KRR^B(f) \rightarrow 0\; ,\\
    \label{equ.A.exact}
    0 \rightarrow \KRR^A_<(f) \rightarrow \K fA \rightarrow \KRR^A(f) \rightarrow 0 \; .
  \end{gather}

  Consider now the sequence obtained by tensoring~\eqref{equ.B.exact} by $\K A$:
  \begin{equation}\label{equ.tensor.exact}
    0 \rightarrow \KRR^B_<(f)\tensorBA \rightarrow \K fB\tensorBA
    \rightarrow \KRR^B(f)\tensorBA \rightarrow 0 \; .
  \end{equation}
  We want to prove that it is exact and isomorphic to~\eqref{equ.A.exact}.

  First note that, since $\K B$ is a subalgebra of $\K A$, we have $b \otimes a =
  1 \otimes ba$ for $b\in B$ and $a\in A$. Therefore the product map
  \begin{displaymath}
    \mu\ :\ \begin{cases}
      \K fB\tensorBA &\longrightarrow \K fA\\
      fb \otimes a      &\longmapsto fba
    \end{cases}
  \end{displaymath}
  is an isomorphism of $\K A$-modules.

  Next consider the restriction of $\mu$ to $\KRR^B_<(f)\tensorBA$. Its image
  set is $\KRR^B_<(f)A$, which is equal to $\KRR^A_<(f)$ by
  hypothesis. Therefore, $\mu$ restricts to an $A$-module isomorphism from
  $\KRR^B_<(f)\otimes_{\K B} \K A$ to $\KRR^A_<(f)$. As a consequence, the
  following diagram is commutative, all vertical arrows being isomorphisms
  (for short we write here $\otimes$ for $\otimes_{\K B}$):
  \def\Ind#1{{#1}\otimes\K A}
  \begin{equation*}
    \begin{CD}
      0 @>>> \Ind{\KRR^B_<(f)} @>>> \Ind{\K fB}
      @>>> \Ind{\KRR^B(f)} @>>> 0 \\
      @.  @V\mu VV     @V\mu VV     @V\id VV     @.  \\
      0 @>>> \KRR^A_<(f) @>>> \K fA
      @>>> \Ind{\KRR^B(f)} @>>> 0\\
    \end{CD}
  \end{equation*}
  It is a well-known fact that the functor $? \otimes_{\K B}\K A$ is right
  exact, so that the middle and right part of the top sequence is exact. The
  left part of the bottom sequence is clearly exact. Therefore they are both
  exact sequences.

  Comparing with~\eqref{equ.A.exact}, we obtain that
  \begin{equation*}
    \KRR^A(f) \cong \KRR^B(f)\tensorBA\,,
  \end{equation*}
  where the latter is isomorphic to $\induced{\KRR^B(f)}{\K B}{\K A}$ by
  definition.
\end{proof}

In the following lemma we denote by $<_{\RR^A}$ the strict right preorder
on a monoid $A$; namely $x<_{\RR^A} y$ if $x\leq_{\RR^A} y$ but
$x\notin\RR^A(y)$.
\begin{lemma}
  \label{lemma.induction}
  Let $B \subseteq A$ be two finite monoids and assume that:
  \begin{enumerate}[(i)]
  \item \label{item.RR_induction}%
    $\RR$-order on $B$ is induced by $\RR$-order on $A$; that is, for
    all $x,y\in B$,
    \begin{displaymath}
      x<_{\RR^A} y \qquad \Longleftrightarrow \qquad x<_{\RR^B} y\,.
    \end{displaymath}
  \item \label{item.RR_intersection}
  Any $\RR$-class of $A$ intersects $B$.
  \end{enumerate}
  Then, for any $f\in B$, the equality $\KRR^B_<(f)A= \KRR^A_<(f)$
  holds. In particular,
  \begin{displaymath}
    \KRR^A(f) \cong \induced{\KRR^B(f)}{\K B}{\K A}\,.
  \end{displaymath}
  Moreover, Condition~\eqref{item.RR_induction} may be replaced be the
  following stronger condition:

  \newcounter{foobarcounter}
  \renewcommand{\thefoobarcounter}{coucou\value{foobarcounter}'}

  \begin{enumerate}[(i)]
    \gdef\theenumi{\roman{enumi}'}
  \item\label{item.RR_induction_strong}
    \hspace{2.5cm}
    $x\leq_{\RR^A} y \qquad \Longleftrightarrow \qquad x\leq_{\RR^B} y\,.$
  \end{enumerate}
\end{lemma}

\begin{proof}\mbox{}

  Inclusion $\subseteq$: Take $b\in B$ with $b<_{\RR^B} f$ and $a\in
  A$. Then, using~\eqref{item.RR_induction}, $ba\in \KRR^A_<(f)$:
  \begin{displaymath}
    ba \leq_{\RR^A} b < _{\RR^A} f\,.
  \end{displaymath}

  Inclusion $\supseteq$: Take $a\in A$ with $a<_{\RR^A}f$.
  Using~\eqref{item.RR_intersection} choose an element $b\in B$ such
  that $b \; \RR^A \; a$. Then $b\leq_{\RR^A} a <_{\RR^A} f$ and
  therefore, by~\eqref{item.RR_induction}, $b\in \KRR^B_<(f)$. It
  follows that $a\in \KRR^B_<(f) A$.

  The statement $\KRR^A(f) \cong \induced{\KRR^B(f)}{\K B}{\K A}$ follows from 
  Lemma~\ref{lemma.induction.algebrapart}.
\end{proof}

Here is an example of what can go wrong when
Condition~\eqref{item.RR_induction} fails.
\begin{example}
  \label{example.rees counterexample}
  Let $A$ be the (multiplicative) submonoid of $M_2(\Z)$ with elements given
  by the matrices:
  \begin{displaymath}
    1    :=\left(\begin{smallmatrix}1&0\\0&1\end{smallmatrix}\right),\,
    b_{11}:=\left(\begin{smallmatrix}1&0\\0&0\end{smallmatrix}\right),\,
    b_{12}:=\left(\begin{smallmatrix}0&1\\0&0\end{smallmatrix}\right),\,
    a_{21}:=\left(\begin{smallmatrix}0&0\\1&0\end{smallmatrix}\right),\,
    b_{22}:=\left(\begin{smallmatrix}0&0\\0&1\end{smallmatrix}\right),\,
    0    :=\left(\begin{smallmatrix}0&0\\0&0\end{smallmatrix}\right)\,.
  \end{displaymath}
  Alternatively, $A$ is the aperiodic Rees matrix monoids (see
  Definition~\ref{definition.rees}) whose non-trivial $\JJ$-class is
  described by:
  \begin{displaymath}
    \begin{pmatrix}
      b_{11}^* & b_{12}  \\
      a_{21}   & b_{22}^*\\
    \end{pmatrix}\,,
  \end{displaymath}
  where the ${}^*$ marks the elements which are idempotent. In other
  words $A=M(P)$, where
  $P:=\left(\begin{smallmatrix}1&0\\0&1\end{smallmatrix}\right)$, and
  for convenience the above matrix specifies names for the elements of
  the non-trivial $\JJ$-class. Recall that the non-trivial left and
  right classes of $A$ are given respectively by the columns and rows
  of this matrix.

  Let $B$ be the submonoid $\{1,b_{11},b_{12},b_{22},0\}$. Then $B$
  satisfies Condition~\eqref{item.RR_intersection} but not
  Condition~\eqref{item.RR_induction}: indeed $b_{11}\; \RR^A\;
  b_{12}$ whereas $b_{11}<_{\RR^B} b_{12}$. Then, taking $f=b_{11}$,
  one obtains $\RR^B_<(b_{11}) = \{0, b_{12}\}$ so that
  $\RR^B_<(b_{11})A = \{0, b_{11}, b_{12}\}$, and therefore:
  \begin{displaymath}
    \K\{0\} = \KRR^A_<(b_{11}) \subset \KRR^B_<(b_{11})A = \K\{0,b_{11},b_{12}\}\,.
  \end{displaymath}
  Now $\KRR^B(b_{11}) = \K\{0, b_{11}, b_{12}\} / \K\{0, b_{12}\}$, so
  that $\KRR^B(b_{11})$ is one-dimensional, spanned by $x := b_{11}
  \mod (\K\{0, b_{12}\})$. The action of $B$ is given by $x.1 =
  x.b_{11} = x$ and $x.m = 0$ for any $m\in B \setminus \{1, b_{11}\}$.

  We claim that
  \begin{displaymath}
  \induced{\KRR^B(b_{11})}{\K B}{\K A} = \KRR^B(b_{11})\tensorBA  = 0\,.
  \end{displaymath}
  Indeed, $x \otimes 1 = x.b_{11} \otimes 1 = x \otimes b_{11} =
  x \otimes b_{12} a_{21} = x.b_{12} \otimes a_{21} = 0$. Thus
  \begin{equation*}
    \KRR^A(b_{11}) \not\cong \induced{\KRR^B(b_{11})}{\K B}{\K A}\,.
  \end{equation*}
\end{example}

\medskip

As shown in the following example,
Condition~\eqref{item.RR_induction_strong} may be strictly stronger
than Condition~\eqref{item.RR_induction} because $<_\RR$ is only a preorder.
\begin{example} \label{example.rees}
  Let $A$ be the aperiodic Rees matrix monoid with non-trivial $\JJ$-class
  given by:
  \begin{displaymath}
    \begin{pmatrix}
      a_{11}^* & b_{12} & b_{13}\\
      a_{21}^* & b_{22}^* & a_{23}\\
      a_{31}^* & a_{32} & b_{33}^*\\
    \end{pmatrix}\,,
  \end{displaymath}
  Let $B$ be the submonoid $\{1,b_{12},b_{13},b_{22},b_{33},0\}$. Then
  $B$ satisfies Conditions~\eqref{item.RR_induction} and~\ref{item.RR_intersection}, but not
  Condition~\eqref{item.RR_induction_strong}: $b_{12}$ and $b_{13}$ are
  incomparable for $\leq_{\RR^B}$ whereas they are in the same right
  class for $A$.
\end{example}
\medskip

We now turn to the proof of
Proposition~\ref{proposition.translation_module_induced_from_M0} by showing
that $\Mzero(W) \subseteq \biheckemonoid(W)$ satisfy the conditions of
Lemma~\ref{lemma.induction}. We use the stronger
condition~\eqref{item.RR_induction_strong}.
\begin{lemma}
  \label{lemma.conditions}
  The biHecke monoid and its Borel submonoid $\Mzero(W) \subseteq \biheckemonoid(W)$
  satisfy conditions~\eqref{item.RR_induction_strong} and~\eqref{item.RR_intersection} of
  Lemma~\ref{lemma.induction}.
\end{lemma}

\begin{proof}
  By Proposition~\ref{proposition.characterization_R_class}, for any
  $f\in \biheckemonoid$ there exists a unique $f_1\in \RR(f) \cap
  \Mone$. Using the bar involution of
  Section~\ref{subsection.involution}, one finds similarly, a unique
  $\overline{f}_1\in \RR(f)\cap \Mzero$. This proves
  condition~\eqref{item.RR_intersection}.

  We now prove the non-trivial implication in
  Condition~\eqref{item.RR_induction_strong}. Take $f,g\in \Mzero$ with
  $f\leq_{\RR^{\biheckemonoid}} g$. Then, $f=gx$ for some $x\in
  \biheckemonoid$. Note that $w_0.f = w_0.g=w_0$, which implies that
  $w_0.x = w_0$ as well. Hence $x$ is in fact in $\Mzero$ and
  $f\leq_{\RR^{\Mzero}} g$.
\end{proof}

\begin{proof}[Proof of Proposition~\ref{proposition.translation_module_induced_from_M0}]
  Let $g_w:= e_{w,w_0}$. By definition, the translation module is the
  quotient $T_w = \K g_w\biheckemonoid / \KRR_<(g_w)$, whereas
  $\Szero_w= \K g_w\Mzero / \KRR_<^{w_0}(g_w)$. By
  Lemma~\ref{lemma.conditions}, $\Mzero \subseteq \biheckemonoid$
  satisfy the two conditions of Lemma~\ref{lemma.induction};
  Proposition~\ref{proposition.translation_module_induced_from_M0}
  follows.
\end{proof}

\begin{theorem}
  The right regular representation of $\K\biheckemonoid$ admits a
  filtration with factors all isomorphic to translation
  modules, and its character is given by
  \begin{equation}
    [\K \biheckemonoid] = \sum_{f\in \Mzero} [T_{\one.f}]\,.
  \end{equation}
\end{theorem}
\begin{proof}
  As any monoid algebra, $\K \biheckemonoid$ admits a filtration where each
  composition factor is given by (the linear span of) a $\RR$-class of
  $\biheckemonoid$. By Proposition~\ref{proposition.isom_translation},
  each such composition factor is isomorphic to the translation module
  $T_{\one.f}$, where $f$ is the unique element of the $\RR$-class
  which lies in $\Mzero$. The character formula
  follows. Alternatively, it can be obtained using
  Proposition~\ref{proposition.translation_module_induced_from_M0} and
  the character formula for the right regular representation of
  $\Mzero$ (see Remark~\ref{remark.characters.M1}):
  \begin{equation}
    [\K \Mzero]_\Mzero = \sum_{f\in \Mzero} [\Szero_{\one.f}]_\Mzero\,.\qedhere
  \end{equation}
\end{proof}

\begin{proposition}
  \label{proposition.character_T_S0}
  For any $w\in W$ the translation module $T_w$ is multiplicity-free
  as an $\K\Mzero$-module and its character is given by
  \begin{equation}
    [T_w]_\Mzero = \sum_{u\in [1,w]_R} [\Szero_u]_{\Mzero}\,.
  \end{equation}
\end{proposition}

\begin{proof}
  Let $f$ be an element in $\biheckemonoid$ which yields the translation module $T_w$, and
  define $f_u$ as in Proposition~\ref{proposition.characterization_R_class}.

  Take some sequence $u_1,\dots,u_m$ (for $m=|[1,w]_R|$) of the elements of
  $[1,w]_R$ which is length increasing, and define the corresponding sequence
  of subspaces by $X_i:=\K \{u_1,\dots,u_i\}$. Using
  Lemma~\ref{lemma.triangular.uu}, each such subspace is stable by $\Mzero$,
  and $X_0 \subset \dots \subset X_m$ forms an $\Mzero$-composition series of
  $T_w$ since $X_i/X_{i-1}$ is of dimension $1$.

  Consider now a composition factor $X_i/X_{i-1}$. Again, by
  Lemma~\ref{lemma.triangular.uu}, $e_{v,w_0}$ fixes $u_i$ if and only
  if $v\leq_L u_i$ (that is if the image set $[u_i,w^{-1}w_0u_i]_L$ of $f_{u_i}$
  is contained in the image set $[v,w_0]_L$ of $e_{v,w_0}$), and kills
  it otherwise. Hence, $X_i/X_{i-1}$ is isomorphic to
  $\Szero_{u_i}$.
\end{proof}

\begin{theorem}
  The decomposition map of $\K\biheckemonoid$ over $\K\Mzero$ is lower uni-triangular
  for right order, with $0,1$ entries. More explicitly,
  \begin{equation}
    [S_w]_\Mzero = \sum_{u\in [1,w]_R \backslash \bigcup_{v \cut w} [1,v]_R} [\Szero_u]_\Mzero\,.
  \end{equation}
\end{theorem}
\begin{proof}
  Since $S_w$ is a quotient of $T_w$, its composition factors form a
  subset of the composition factors for $T_w$. Hence, using
  Proposition~\ref{proposition.character_T_S0}, the decomposition
  matrix of $\biheckemonoid$ over $\Mzero$ is lower triangular for right order,
  with $0,1$ entries. Furthermore, by construction (see
  Remark~\ref{remark.Sw} and Theorem~\ref{theorem.simple} (ii)), $S_w
  = T_w / \sum_{v\cut w} T_v$; using
  Proposition~\ref{proposition.character_T_S0} the sum on the right
  hand side contains at least one composition factor isomorphic to
  $\Szero_u$ for each $u$ in $[1,v]_R$ with $v\cut w$; therefore $S_w$
  has no such composition factor. We conclude using the dimension
  formula of Theorem~\ref{theorem.simple} (ii).
\end{proof}
\begin{example}
  Following up on Example~\ref{example.S4312}, the decomposition of
  the $\K\biheckemonoid$-simple module $S_{4312}$ over $\K\Mzero$ is given by
  $[S_{4312}]_\Mzero =
  [\Szero_{4312}]+[\Szero_{4132}]+[\Szero_{1432}]$.  See also
  Figure~\ref{figure.4312} and the decomposition matrices given in
  Appendix~\ref{appendix.decomposition.matrices}.
\end{example}

\subsection{Example: the rank $2$ Coxeter groups}

We now give a complete description of the representation theory of the
biHecke monoid for each rank $2$ Coxeter group $I_p$. The proofs are
left as exercises for the reader.

\begin{example}
  \label{example.bihecke.Ip}
  Let $\biheckemonoid$ be the biHecke monoid for the dihedral group
  $W:=I_p$ of order $2p$. Then, $\biheckemonoid$ is a regular monoid.

  The right class module $\KRR_w:=\KRR(e_{w,w_0})$ is the translation
  module spanned by $[1,w]_R$. It is of dimension $2p$ for $w=w_0$,
  and $\len(w)$ otherwise. The left class modules $\KLL_1$ and
  $\KLL_{w_0}$ are respectively the trivial module spanned by $1$ and
  the zero module spanned by $w_0$. For $w\ne1,w_0$, the left class
  module $\KLL_w$ is of dimension $\len(w)-1$, and its structure is as
  in Figure~\ref{fig.left_right_class_module.Ip}. In particular,
  \begin{displaymath}
    |\biheckemonoid| = 2p+1 + 2 \sum_{k=1}^{p-1} k (k+1) = \frac{2}{3}p^3 + \frac{4}{3}p + 1\,.
  \end{displaymath}
  \begin{figure}
    \centering
    \mbox{\begin{tikzpicture}[>=latex,line join=bevel,
  every node/.style={draw,draw=none,inner sep=.5ex},
  every loop/.style={shorten >= 1pt, looseness=.5},
  yscale=1]

  \node (one)  at (0,4) {$\phantom{\opi_{2121}}\pi_{1212}$};
  \node (1)    at (0,3) {$\phantom{\scriptstyle 121}\opi_{1}\pi_{1212}$};
  \node (21)   at (0,2) {$\phantom{\scriptstyle 11}\opi_{21}\pi_{1212}$};
  \node (121)  at (0,1) {$\phantom{\scriptstyle 1}\opi_{121}\pi_{1212}$};
  \node (2121) at (0,0) {$\opi_{2121}\pi_{1212}$};

  \draw [->,color=blue] (one) to [bend right=10] node [left] {$\overline 1$} (1);
  \draw [->,color=red ] (1)   to [bend right=10] node [left] {$\overline 2$} (21);
  \draw [->,color=blue] (21)  to [bend right=10] node [left] {$\overline 1$} (121);
  \draw [->,color=red ] (121) to [bend right=10] node [left] {$\overline 2$} (2121);

  \draw [<-,color=red ] (one) to [bend left =10] node [right] {$2$} (1);
  \draw [<-,color=blue] (1)   to [bend left =10] node [right] {$1$} (21);
  \draw [<-,color=red ] (21)  to [bend left =10] node [right] {$2$} (121);
  \draw [<-,color=blue] (121) to [bend left =10] node [right] {$1$} (2121);

  \draw [<-,color=red ] (one) to [loop right] node [right] {$2$} (one);
  \draw [<-,color=blue] (1)   to [loop right] node [right] {$1$} (1);
  \draw [<-,color=red ] (21)  to [loop right] node [right] {$2$} (21);
  \draw [<-,color=blue] (121) to [loop right] node [right] {$1$} (121);
  \draw [<-,color=red ] (2121)to [loop right] node [right] {$2$} (2121);

  \draw [<-,color=red ] (one) to [loop left] node [left] {$\overline 2$} (one);
  \draw [<-,color=blue] (1)   to [loop left] node [left] {$\overline 1$} (1);
  \draw [<-,color=red ] (21)  to [loop left] node [left] {$\overline 2$} (21);
  \draw [<-,color=blue] (121) to [loop left] node [left] {$\overline 1$} (121);
  \draw [<-,color=red ] (2121)to [loop left] node [left] {$\overline 2$} (2121);
\end{tikzpicture}

 \qquad\qquad  \begin{tikzpicture}[>=latex,line join=bevel,
  every node/.style={draw,draw=none,inner sep=.5ex},
  every loop/.style={shorten >= 1pt, looseness=.5},
  yscale=1]

  \node (one)  at (0,4) {$\pi_{1212}\phantom{\opi_{2121}}$};
  \node (1)    at (0,3) {$\pi_{1212}\opi_{2\phantom{121}}$};
  \node (21)   at (0,2) {$\pi_{1212}\opi_{21\phantom{21}}$};
  \node (121)  at (0,1) {$\pi_{1212}\opi_{212\phantom{1}}$};
  \node (2121) at (0,0) {$\pi_{1212}\opi_{2121}$};

  \draw [->,color=red ] (one) to [bend right=10] node [left] {$\overline 2$} (1);
  \draw [->,color=blue] (1)   to [bend right=10] node [left] {$\overline 1$} (21);
  \draw [->,color=red ] (21)  to [bend right=10] node [left] {$\overline 2$} (121);
  \draw [->,color=blue] (121) to [bend right=10] node [left] {$\overline 1$} (2121);

  \draw [<-,color=red ] (one) to [bend left =10] node [right] {$2$} (1);
  \draw [<-,color=blue] (1)   to [bend left =10] node [right] {$1$} (21);
  \draw [<-,color=red ] (21)  to [bend left =10] node [right] {$2$} (121);
  \draw [<-,color=blue] (121) to [bend left =10] node [right] {$1$} (2121);

  \draw [<-,color=red ] (one) to [loop right] node [right] {$2$} (one);
  \draw [<-,color=blue] (1)   to [loop right] node [right] {$1$} (1);
  \draw [<-,color=red ] (21)  to [loop right] node [right] {$2$} (21);
  \draw [<-,color=blue] (121) to [loop right] node [right] {$1$} (121);

  \draw [<-,color=red ] (1)   to [loop left] node [left] {$\overline 2$} (21);
  \draw [<-,color=blue] (21)  to [loop left] node [left] {$\overline 1$} (21);
  \draw [<-,color=red ] (121) to [loop left] node [left] {$\overline 2$} (121);
  \draw [<-,color=blue] (2121)to [loop left] node [left] {$\overline 1$} (2121);
\end{tikzpicture}

}
    \caption{The left and right class modules indexed by
      $w:=s_1s_2s_1s_2$ for the biHecke monoid $\biheckemonoid(I_p)$
      with $p\geq 5$.  The left picture also describes the left simple
      module $S_w$ of $\biheckemonoid(I_p)$, and the projective module
      $P^{w_0}_w$ of the Borel submonoid $\Mzero(I_p)$. }
    \label{fig.left_right_class_module.Ip}
  \end{figure}
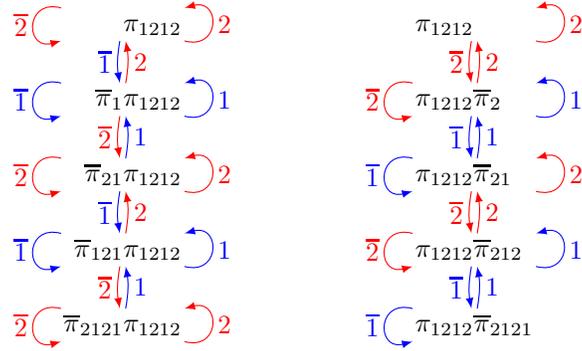

  The simple right module $S_w$ can be constructed from the cutting
  poset. Namely, $S_1$ is the trivial module spanned by $1$, while
  $S_{w_0}$ is the zero module spanned by $w_0$ and, for $w\ne 1,w_0$,
  $S_w$ is the quotient of the right class module by the line spanned
  by alternating sum of $[1,w]_R$. The simple left module $S_w$ is
  directly given by the left class module $L_w$.

  The quiver of $\biheckemonoid$ is given by the cutting poset (see
  Figure~\ref{figure.quiver.Ip}). The $q$-Cartan matrix is given by
  the path algebra of this quiver; namely there is an extra arrow from
  $1$ to $w_0$ with weight $q^2$. In particular, it is upper
  unitriangular and of determinant $1$.
\end{example}
\begin{figure}
  \centering
  \begin{tikzpicture}[>=latex,line join=bevel,]
  \node (one) at ( 0,0) [draw,draw=none] {$1$};
  \node (1)   at (-2,1) [draw,draw=none] {$s_1$};
  \node (2)   at ( 2,1) [draw,draw=none] {$s_2$};
  \node (12)  at (-3,1) [draw,draw=none] {$s_1s_2$};
  \node (21)  at ( 3,1) [draw,draw=none] {$s_2s_1$};
  \node (121)  at (-4,1) [draw,draw=none] {$s_1s_2s_1$};
  \node (212)  at ( 4,1) [draw,draw=none] {$s_2s_1s_2$};
  \node (w01) at (-1,1) [draw,draw=none] {$w_0s_1$};
  \node (w02) at ( 1,1) [draw,draw=none] {$w_0s_2$};
  \node (w0)  at ( 0,2) [draw,draw=none] {$w_0$};

  \draw [->] (one) -- (1);
  \draw [->] (one) -- (2);
  \draw [->] (one) -- (12);
  \draw [->] (one) -- (21);
  \draw [->] (one) -- (121);
  \draw [->] (one) -- (212);
  \draw [->] (one) -- (w01);
  \draw [->] (one) -- (w02);
  \draw [->] (w01) -- (w0);
  \draw [->] (w02) -- (w0);
\end{tikzpicture}

  \caption{The Hasse diagram of the cutting poset for the dihedral
    group $W:=I_5$. This is also the quiver of the biHecke monoid for
    that group.}
  \label{figure.quiver.Ip}
\end{figure}
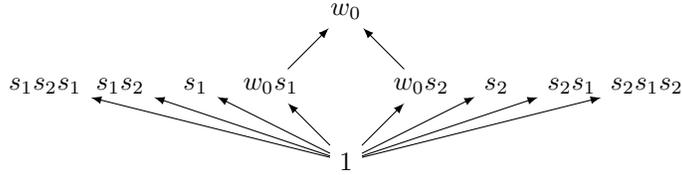
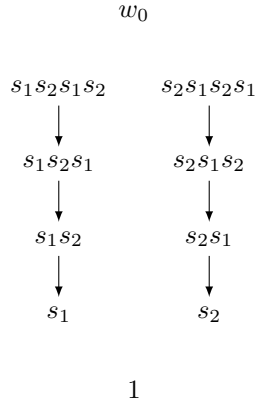
\begin{figure}
  \centering
  \begin{tikzpicture}[>=latex,line join=bevel,]
  \node (one) at ( 0,0) [draw,draw=none] {$1$};
  \node (1)   at (-1,1) [draw,draw=none] {$s_1$};
  \node (2)   at ( 1,1) [draw,draw=none] {$s_2$};
  \node (12)  at (-1,2) [draw,draw=none] {$s_1s_2$};
  \node (21)  at ( 1,2) [draw,draw=none] {$s_2s_1$};
  \node (121) at (-1,3) [draw,draw=none] {$s_1s_2s_1$};
  \node (212) at ( 1,3) [draw,draw=none] {$s_2s_1s_2$};
  \node (1212)at (-1,4) [draw,draw=none] {$s_1s_2s_1s_2$};
  \node (2121)at ( 1,4) [draw,draw=none] {$s_2s_1s_2s_1$};
  \node (w0)  at ( 0,5) [draw,draw=none] {$w_0$};

  \draw [<-] (1) -- (12);
  \draw [<-] (2) -- (21);
  \draw [<-] (12) -- (121);
  \draw [<-] (21) -- (212);
  \draw [<-] (121) -- (1212);
  \draw [<-] (212) -- (2121);
\end{tikzpicture}

  \caption{The quiver of the Borel submonoid $\Mzero(I_5)$ of the
    biHecke monoid for the dihedral group $I_5$.}
  \label{figure.quiver.M0.Ip}
\end{figure}

\begin{example}
  \label{example.M0.Ip}
  Let $\Mzero$ be the Borel submonoid of the biHecke monoid for the
  dihedral group $W:=I_p$ of order $2p$.

  The projective module $P_w$ of $\Mzero$ is given by the left simple
  modules $S_w$, or equivalently the left-class-module $L_w$ of
  $\biheckemonoid$. In particular,
  \begin{displaymath}
    |\Mzero| = 1 + 1 + 2\sum_{k=1}^{p-1} k = p^2-p+2\,.
  \end{displaymath}

  The quiver of $\Mzero$ is given by the cover relations in Bruhat
  order (or equivalently right order) which are not covers in left
  order (see Figure~\ref{figure.quiver.M0.Ip}); this gives two chains
  of length $p-1$. The monoid algebra is isomorphic to the path
  algebra of this quiver, which gives right away its radical
  filtration. Combinatorially speaking, every non-idempotent element
  $f$ of the monoid admits a unique minimal factorization $e_we_u$,
  with $\len(u)<\len(w)$ and $u\not\le_L w$; namely, $u:=f(1)$ and $w$
  is the smallest element such that $f(w)=1$.
\end{example}
\section{Research in progress}
\label{section.in_progress}

Our guiding problem is the search for a formula for the cardinality of
the biHecke monoid. Using a standard result of the representation theory
of finite-dimensional algebras together with the results of this paper,
we can now write
\begin{displaymath}
  |\biheckemonoid(W)| = \sum_{w\in W} \dim S_w \dim P_w\,,
\end{displaymath}
where $\dim S_w$ is given by an inclusion-exclusion formula. It
remains to determine the dimensions of the projective modules $P_w$.

While studying the representation theory of the Borel submonoid
$\Mone$ as an intermediate step, the authors realized that many of the combinatorial
ingredients that arose were well-known in the semigroup community (for
example the Green's relations and related classes, automorphism groups,
etc.), and hence the representation theory of $\Mone$ is naturally
expressed in the context of $\JJ$-trivial monoids
(see~\cite{Denton_Hivert_Schilling_Thiery.JTrivialMonoids}). This
sparked their interest in the representation theory of more general
classes of monoids, in particular aperiodic monoids.

At the current stage, it appears that the Cartan matrix of an
aperiodic monoid (and therefore the composition series of its
projective modules, and by consequence their dimensions) is completely
determined by the knowledge of the composition series for both left
and right class modules. In other words, the study in this paper of
right class modules (i.e. translation modules), whose original purpose
was to construct the simple modules
using~\cite[Theorem~7]{Ganyushkin_Mazorchuk_Steinberg.2009}, turns out
to complete half of this program. The remaining half, in progress, is
the decomposition of left class modules.

At the combinatorial level, this requires to control
$\LL$-order. Loosely speaking, $\LL$-order is essentially given by
left and right order in $W$; however, within $\LL$-classes the
structure seems more elusive, in particular because fibers are more
difficult to describe than image sets. Another difficulty is that,
unlike for $\RR$-class modules, $\LL$-class modules are not all
isomorphic to regular ones (i.e. classes containing idempotents).

Yet, the general theory gives that the decomposition matrix should be
upper triangular for left-right order for regular classes, and upper
triangular for Bruhat order for nonregular ones, with no left-right
``arrow'' for left-right order. Pushing this further gives that the
Cartan matrix has determinant $1$.

We conclude by illustrating the above for $W=\sg[4]$ in
Figure~\ref{figure.cartan_matrix.S4}.  The blue arrows are the
covering relations of the cutting poset, which encode the composition
series of the translation modules (i.e. right class modules).
Namely, the character of $T_w$ is given by the sum of $q^k [S_u]$ for
$u$ below $w$ in the cutting poset, with $k$ the distance from $u$ to
$w$ in that poset. For example:
\begin{align*}
  [T_{2143}] &= [S_{2143}] + q [S_{1243}] + q [S_{2134}] + q^2 [S_{1234}]\\
  [T_{2341}] &= [S_{2341}] + q [S_{1234}]\\
  [T_{4123}] &= [S_{4123}] + q [S_{4123}].
\end{align*}
Similarly the black (resp. red) arrows encode the composition series
of regular (resp. nonregular) left classes. In this simple example,
the $q$-character of a right projective module $P_w$ is then given by
\begin{equation*}
  [P_w] = [T_w] + \sum_u q [T_u],
\end{equation*}
where $(u,w)$ is a black or red arrow in the graph. For example,
\begin{equation*}
  \begin{split}
    [P_{2143}] &= [T_{2143}] + q [T_{2341}] + q [T_{4123}]\\
    &= [S_{2143}] + q [S_{1243}] + q [S_{2134}] + q [S_{2341}] + q [S_{4123}] + 3q^2 [S_{1234}]\,.
  \end{split}
\end{equation*}

\begin{figure}
  \begin{bigcenter}
    \scalebox{.5}{\begin{tikzpicture}[>=latex,line join=bevel,]
  \node (1234) at (365bp,0bp) [draw,draw=none,blue] {$\text{1234: 1  1  1}$};

  \node (1423) at (60bp,50bp) [draw,draw=none,blue] {$\text{1423: 2  3  3}$};
  \node (1342) at (125bp,50bp) [draw,draw=none,blue] {$\text{1342: 2  3  3}$};
  \node (3412) at (365bp,50bp) [draw,draw=none,blue] {$\text{3412: 5  6  6}$};
  \node (2341) at (450bp,50bp) [draw,draw=none,blue] {$\text{2341: 3  4  4}$};
  \node (3124) at (281bp,50bp) [draw,draw=none,blue] {$\text{3124: 2  3  3}$};
  \node (2314) at (209bp,50bp) [draw,draw=none,blue] {$\text{2314: 2  3  3}$};
  \node (4123) at (703bp,50bp) [draw,draw=none,blue] {$\text{4123: 3  4  4}$};
  \node (2413) at (770bp,50bp) [draw,draw=none,blue] {$\text{2413: 4  5  5}$};
  \node (3142) at (835bp,50bp) [draw,draw=none,blue] {$\text{3142: 4  5  5}$};

  \node (1432) at (59bp,100bp) [draw,draw=none] {$\text{1432: 1  6 12}$};
  \node (3214) at (209bp,100bp) [draw,draw=none] {$\text{3214: 1  6 12}$};
  \node (1243) at (134bp,100bp) [draw,draw=none] {$\text{1243: 1  2  8}$};
  \node (2431) at (450bp,100bp) [draw,draw=none,blue] {$\text{2431: 4  8  8}$};
  \node (4312) at (520bp,100bp) [draw,draw=none,blue] {$\text{4312: 3 12 12}$};
  \node (2134) at (284bp,100bp) [draw,draw=none] {$\text{2134: 1  2  8}$};
  \node (3241) at (600bp,100bp) [draw,draw=none,blue] {$\text{3241: 4  8  8}$};
  \node (3421) at (365bp,100bp) [draw,draw=none,blue] {$\text{3421: 3 12 12}$};
  \node (4231) at (703bp,100bp) [draw,draw=none,blue] {$\text{4231: 5 12 12}$};
  \node (4132) at (770bp,100bp) [draw,draw=none,blue] {$\text{4132: 4  8  8}$};
  \node (4213) at (835bp,100bp) [draw,draw=none,blue] {$\text{4213: 4  8  8}$};

  \node (2143) at (365bp,150bp) [draw,draw=none] {$\text{2143: 1  4 12}$};
  \node (4321) at (520bp,150bp) [draw,draw=none,blue] {$\text{4321: 1 24 24}$};
  \node (1324) at (703bp,150bp) [draw,draw=none] {$\text{1324: 1  2 22}$};

  \draw [->,blue] (1234) to [out=172,in=-62] (1243);
  \draw [->,blue] (1234) to [bend left=7] (1342);
  \draw [->,blue] (1234) to (2314);
  \draw [->,blue] (1234) to (3124);
  \draw [->,blue] (1234) to (3412);
  \draw [->,blue] (1234) to (2341);

  \draw [->,blue] (1234) to (2134);
  \draw [->,blue] (1234) to [bend left=6] (1423);
  \draw [->,blue] (1234) to [bend right=17] (1324);
  \draw [->,blue] (1234) to [bend right=2] (4123);
  \draw [->,blue] (1234) to [bend right=2] (2413);
  \draw [->,blue] (1234) to [bend right=4] (3142);

  \draw [->,blue] (1243) to (2143);

  \draw [->,blue] (1342) to (1432);

  \draw [->,blue] (1423) to (1432);

  \draw [->,blue] (2134) to (2143);

  \draw [->,blue] (2314) to (3214);

  \draw [->,blue] (2341) to (3421);
  \draw [->,red ] (2341) to (2143);
  \draw [->,blue] (2341) to (3241);
  \draw [->,blue] (2341) to (2431);
  \draw [->     ] (2341) to [bend right=18] (1324);
  \draw [->,blue] (2341) to [bend right=8] (4231);

  \draw [->,blue] (3124) to (3214);

  \draw [->,red ] (3412) to (1432);
  \draw [->     ] (3412) to (1243);
  \draw [->,red ] (3412) to (3214);
  \draw [->     ] (3412) to (2134);
  \draw [->,blue] (3412) to (3421);
  \draw [->,blue] (3412) to (4312);

  \draw [->,blue] (3421) to (4321);

  \draw [->,blue] (4123) to (4312);
  \draw [->,red ] (4123) to [out=136, in=-5] (2143);
  \draw [->,blue] (4123) to (4231);
  \draw [->     ] (4123) to [bend right=60] (1324);
  \draw [->,blue] (4123) to [bend right=9] (4132);
  \draw [->,blue] (4123) to [bend right=5] (4213);

  \draw [->,blue] (4231) to (4321);
  \draw [->     ] (4231) to (1324);

  \draw [->,blue] (4312) to (4321);

\end{tikzpicture}}
  \end{bigcenter}
  \caption{Graph encoding the characters of left and right class
    modules, and therefore the Cartan invariant matrix for
    $\biheckemonoid(\sg[4])$. See the text for details.}
  \label{figure.cartan_matrix.S4}
\end{figure}
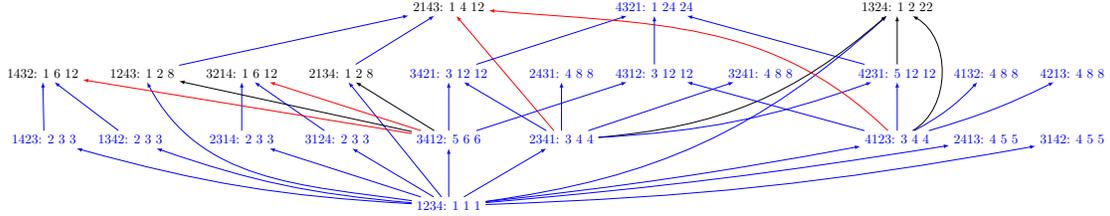

\appendix

\section{Monoid of edge surjective morphism of a colored graph}
\label{appendix.colored_graphs}

Let $C$ be a set whose elements are called colors. We consider colored
simple digraphs without loops. More precisely, a $C$-colored graph is a
triple $G=(V, E, c)$, where $V$ is the set of vertices of $G$,
$E\subset V\times V / \{(x,x) \mid x\in V\}$ is the set of (oriented)
edges of $G$, and $c:E\to C$ is the coloring map.

\begin{definition}
  Let $G=(V, E, c)$ and $G'=(V', E', c')$ be two colored graphs. An
  \emph{edge surjective} morphism (or ES-morphism) from $G$ to $G'$
  is a map $f:V\to V'$ such that
  \begin{itemize}
  \item For any edge $(a,b)\in E$, either $f(a)=f(b)$, or
    $(f(a),f(b))\in E'$ and $c(a,b) = c'(f(a), f(b))$.
  \item For any edge $(a',b')\in E'$ with $a'$ and $b'$ in the image set
    of $f$ there exists an edge $(a,b)\in E$ such that
    $f(a) = a'$ and $f(b) = b'$.
  \end{itemize}
\end{definition}
Note that by analogy to categories, instead of ES-morphism, we can speak
about full morphisms.

The following proposition shows that colored graphs together with edge
surjective morphisms form a category.
\begin{proposition}
  For any colored graphs $G, G_1,G_2,G_3$:
  \begin{itemize}
  \item The identity $\id:G\to G$ is an ES-morphism;
  \item For any ES-morphism $f:G_1\to G_2$ and $g:G_2\to G_3$ the
    composed function $g \circ f:G_1\to G_3$ is an ES-morphism.
  \end{itemize}
\end{proposition}
\begin{corollary}
  For any colored graph $G$, the set of ES-morphisms from $G$ to $G$
  is a submonoid of the monoid of the functions from $G$ to $G$.
\end{corollary}

Here are some general properties of ES-morphisms:
\begin{proposition}
  Let $G_1$ and $G_2$ be two colored graphs and $f$ an ES-morphism
  from $G_1$ to $G_2$. Then the image of any path in $G_1$ is a path
  in $G_2$.
\end{proposition}

In our particular case, we have some more properties:
\begin{enumerate}[(i)]
\item The graph is acyclic, with unique source and sink.
  In particular, it is (weakly) connected.
\item The graph is ranked by the integers, and edges occur only between two
  consecutive ranks.
\item The graph is $C$-regular, which means that for any vertex $v$
  and any color $c$, there is exactly one edge entering or leaving $v$
  with color $c$.
\end{enumerate}

\begin{remarks}
  Proposition~\ref{proposition.weak_order} gives that our monoid is
  a submonoid of the $\biheckemonoid(G)$ monoid for left order.

  Propositions~\ref{proposition.weak_order2}
  and~\ref{proposition.fibers_image_set} are generic, and would apply
  to any $\biheckemonoid(G)$. For the latter, we just need that $G$ is
  $C$-regular.
\end{remarks}

A natural source of colored graphs are crystal graphs. A question that arises
is how the $G$-monoid of a crystal looks like.

\newpage
\section{Tables}
\label{appendix.tables}

\subsection{$q$-Cartan invariant matrices}
We give the Cartan invariant matrix for $\K \biheckemonoid_{w_0}$ and $\K \biheckemonoid$
in types $A_1$, $A_2$ and $A_3$. The $q$-parameter records the layer in the radical filtration. The
extra rows and columns entitled ``Simp.'' and ``Proj.'' give the dimension of
the simple and projective modules, on the right for right modules and below for
left modules. When all simple modules are one-dimensional, the column is
omitted.

Using~\cite{Thiery.CartanMatrixMonoid}, it is possible to go further,
and compute for example the Cartan invariant matrix for
$\biheckemonoid$ in type $A_4$ in about one hour (though at $q=1$
only).\bigskip \newcommand{\pc}[1]{\rotatebox[origin=Bl]{90}{#1}}

$q$-Cartan invariant matrix of $\biheckemonoid_{w_0}(\sg[2])$ (type $A_1$):
\medskip

{\small
\begin{displaymath}
\begin{array}{c|p{0.3cm}@{\,}p{0.3cm}@{}|c}
  &\pc{12}&\pc{21}& \text{Proj.} \\
\hline
12&1&.& 1  \\
21&.&1& 1  \\
\hline
\text{Proj.} &1&1\\
\end{array}
\end{displaymath}}
\bigskip

$q$-Cartan invariant matrix of $\biheckemonoid_{w_0}(\sg[3])$ (type $A_2$):
\medskip

{\small
\begin{displaymath}
\begin{array}{c|p{0.3cm}@{\,}p{0.3cm}@{\,}p{0.3cm}@{\,}p{0.3cm}@{\,}p{0.3cm}@{\,}p{0.3cm}@{}|c}
  &\pc{123}&\pc{132}&\pc{213}&\pc{231}&\pc{312}&\pc{321} & \text{Proj.} \\
\hline
123&1&.&.&.&.&.&  1  \\
132&.&1&.&.&$q$&.&  2  \\
213&.&.&1&$q$&.&.&  2  \\
231&.&.&.&1&.&.&  1  \\
312&.&.&.&.&1&.&  1  \\
321&.&.&.&.&.&1&  1  \\
\hline
\text{Proj.} &1&1&1&2&2&1\\
\end{array}
\end{displaymath}}
\bigskip

$q$-Cartan invariant matrix of $\biheckemonoid_{w_0}(\sg[4])$ (type $A_3$):
\medskip

{
\tiny
\begin{displaymath}
\begin{array}{c|p{0.3cm}@{\,}p{0.3cm}@{\,}p{0.3cm}@{\,}p{0.3cm}@{\,}p{0.3cm}@{\,}p{0.3cm}@{\,}p{0.3cm}@{\,}p{0.3cm}@{\,}p{0.3cm}@{\,}p{0.3cm}@{\,}p{0.3cm}@{\,}p{0.3cm}@{\,}p{0.3cm}@{\,}p{0.3cm}@{\,}p{0.3cm}@{\,}p{0.3cm}@{\,}p{0.3cm}@{\,}p{0.3cm}@{\,}p{0.3cm}@{\,}p{0.3cm}@{\,}p{0.3cm}@{\,}p{0.3cm}@{\,}p{0.3cm}@{\,}p{0.3cm}@{}|c}
     & \pc{1234}&\pc{1243}&\pc{1324}&\pc{1342}&\pc{1423}&\pc{1432}&\pc{2134}&\pc{2143}&\pc{2314}&\pc{2341}&\pc{2413}&\pc{2431}&\pc{3124}&\pc{3142}&\pc{3214}&\pc{3241}&\pc{3412}&\pc{3421}&\pc{4123}&\pc{4132}&\pc{4213}&\pc{4231}&\pc{4312}&\pc{4321}& \text{Proj.} \\
\hline
1234 & 1&.&.&.&.&.&.&.&.&.&.&.&.&.&.&.&.&.&.&.&.&.&.&.        & 1  \\
1243 & .&1&.&.&$q$&.&.&.&.&.&$q$&$q^2$&.&.&.&.&$q$&.&$q^2$&.&.&.&.&.& 6  \\
1324 & .&.&1&$q$&.&.&.&.&.&$q$&.&.&$q$&$q^2$&.&$q^2$&$q^3$&.&$q$&$q^2$&.&$q$&.&.& 10 \\
1342 & .&.&.&1&.&.&.&.&.&.&.&.&.&$q$&.&.&$q^2$&.&.&$q$&.&.&.&.& 4  \\
1423 & .&.&.&.&1&.&.&.&.&.&.&.&.&.&.&.&.&.&$q$&.&.&.&.&.    & 2  \\
1432 & .&.&.&.&.&1&.&.&.&.&.&.&.&.&.&.&$q$&.&.&$q$&.&.&$q^2$&.& 4  \\
2134 & .&.&.&.&.&.&1&.&$q$&$q^2$&$q$&.&.&.&.&.&$q$&.&.&.&$q^2$&.&.&.& 6  \\
2143 & .&.&.&.&.&.&.&1&.&$q$&$q$&$q^2$&.&.&.&.&.&.&$q$&.&$q^2$&$q^3$&.&.& 7  \\
2314 & .&.&.&.&.&.&.&.&1&$q$&.&.&.&.&.&.&.&.&.&.&.&.&.&.    & 2  \\
2341 & .&.&.&.&.&.&.&.&.&1&.&.&.&.&.&.&.&.&.&.&.&.&.&.    & 1  \\
2413 & .&.&.&.&.&.&.&.&.&.&1&$q$&.&.&.&.&.&.&.&.&$q$&$q^2$&.&.& 4  \\
2431 & .&.&.&.&.&.&.&.&.&.&.&1&.&.&.&.&.&.&.&.&.&$q$&.&.    & 2  \\
3124 & .&.&.&.&.&.&.&.&.&.&.&.&1&$q$&.&$q$&$q^2$&.&.&.&.&.&.&.& 4  \\
3142 & .&.&.&.&.&.&.&.&.&.&.&.&.&1&.&.&$q$&.&.&.&.&.&.&.    & 2  \\
3214 & .&.&.&.&.&.&.&.&.&.&.&.&.&.&1&$q$&$q$&$q^2$&.&.&.&.&.&.& 4  \\
3241 & .&.&.&.&.&.&.&.&.&.&.&.&.&.&.&1&.&$q$&.&.&.&.&.&.    & 2  \\
3412 & .&.&.&.&.&.&.&.&.&.&.&.&.&.&.&.&1&.&.&.&.&.&.&.    & 1  \\
3421 & .&.&.&.&.&.&.&.&.&.&.&.&.&.&.&.&.&1&.&.&.&.&.&.    & 1  \\
4123 & .&.&.&.&.&.&.&.&.&.&.&.&.&.&.&.&.&.&1&.&.&.&.&.    & 1  \\
4132 & .&.&.&.&.&.&.&.&.&.&.&.&.&.&.&.&.&.&.&1&.&.&$q$&.    & 2  \\
4213 & .&.&.&.&.&.&.&.&.&.&.&.&.&.&.&.&.&.&.&.&1&$q$&.&.    & 2  \\
4231 & .&.&.&.&.&.&.&.&.&.&.&.&.&.&.&.&.&.&.&.&.&1&.&.    & 1  \\
4312 & .&.&.&.&.&.&.&.&.&.&.&.&.&.&.&.&.&.&.&.&.&.&1&.    & 1  \\
4321 & .&.&.&.&.&.&.&.&.&.&.&.&.&.&.&.&.&.&.&.&.&.&.&1    & 1  \\
\hline
\text{Proj.} &1&1&1&2&2&1&1&1&2&5&4&4&2&4&1&4&9&3&5&4&4&6&3&1\\
\end{array}
\end{displaymath}}

\newpage

$q$-Cartan invariant matrix of $\biheckemonoid(\sg[2])$ (type $A_1$):
\medskip

{\small
\begin{displaymath}
\begin{array}{c|p{0.3cm}@{\,}p{0.3cm}@{}|cc}
&\pc{12}&\pc{21}& \text{Simp.} & \text{Proj.} \\
\hline
12&1&.& 1 & 1  \\
21&$q$&1& 1 & 2 \\
\hline
\text{Simp.} &1&1\\
\text{Proj.} &2&1\\
\end{array}
\end{displaymath}}
\bigskip

$q$-Cartan invariant matrix of $\biheckemonoid(\sg[3])$ (type $A_2$):
\medskip

{\small
\begin{displaymath}
\begin{array}{c|p{0.3cm}@{\,}p{0.3cm}@{\,}p{0.3cm}@{\,}p{0.3cm}@{\,}p{0.3cm}@{\,}p{0.3cm}@{}|cc}
  &\pc{123}&\pc{132}&\pc{213}&\pc{231}&\pc{312}&\pc{321} & \text{Simp.} & \text{Proj.} \\
\hline
123&1    &.&.&.&.&        &  1  & 1\\
132&$q$  &1&.&.&.&        &  1  & 2\\
213&$q$  &.&1&.&.&        &  1  & 2\\
231&$q$  &.&.&1&.&        &  2  & 3\\
312&$q$  &.&.&.&1&        &  2  & 3\\
321&$q^2$&.&.&$q$&$q$&1   &  1  & 6\\
\hline
\text{Simp.} &1&1&1&2&2&1\\
\text{Proj.}  &8&1&1&3&3&1\\
\end{array}
\end{displaymath}}
\bigskip

$q$-Cartan invariant matrix of $\biheckemonoid(\sg[4])$ (type $A_3$):
\medskip

{
\tiny
\begin{displaymath}
\begin{array}{c|p{1.5cm}@{\,}p{0.3cm}@{\,}p{0.3cm}@{\,}p{0.3cm}@{\,}p{0.3cm}@{\,}p{0.3cm}@{\,}p{0.3cm}@{\,}p{0.3cm}@{\,}p{0.3cm}@{\,}p{0.7cm}@{\,}p{0.3cm}@{\,}p{0.3cm}@{\,}p{0.3cm}@{\,}p{0.3cm}@{\,}p{0.3cm}@{\,}p{0.3cm}@{\,}p{0.3cm}@{\,}p{0.3cm}@{\,}p{0.7cm}@{\,}p{0.3cm}@{\,}p{0.3cm}@{\,}p{0.3cm}@{\,}p{0.3cm}@{\,}p{0.3cm}@{}|cc}
     & \pc{1234}&\pc{1243}&\pc{1324}&\pc{1342}&\pc{1423}&\pc{1432}&\pc{2134}&\pc{2143}&\pc{2314}&\pc{2341}&\pc{2413}&\pc{2431}&\pc{3124}&\pc{3142}&\pc{3214}&\pc{3241}&\pc{3412}&\pc{3421}&\pc{4123}&\pc{4132}&\pc{4213}&\pc{4231}&\pc{4312}&\pc{4321}&  \text{Simp.} &\text{Proj.} \\
\hline
1234 &  $1$&.&.&.&.&.&.&.&.&.&.&.&.&.&.&.&.&.&.&.&.&.&.&.  & 1 & 1 \\
1243 &  $q^2 + q$&$1$&.&.&.&.&.&.&.&.&.&.&.&.&.&.&$q$&.&.&.&.&.&.&.  & 1 & 8 \\
1324 &  $q^3 + 2q^2 + q$&.&$1$&.&.&.&.&.&.&$q^2 + q$&.&.&.&.&.&.&.&.&$q^2 + q$&.&.&$q$&.&.  & 1 & 22 \\
1342 &  $q$&.&.&$1$&.&.&.&.&.&.&.&.&.&.&.&.&.&.&.&.&.&.&.&.  & 2 & 3 \\
1423 &  $q$&.&.&.&$1$&.&.&.&.&.&.&.&.&.&.&.&.&.&.&.&.&.&.&.  & 2 & 3 \\
1432 &  $2q^2$&.&.&$q$&$q$&$1$&.&.&.&.&.&.&.&.&.&.&$q$&.&.&.&.&.&.&.  & 1 & 12 \\
2134 &  $q^2 + q$&.&.&.&.&.&$1$&.&.&.&.&.&.&.&.&.&$q$&.&.&.&.&.&.&.  & 1 & 8 \\
2143 &  $3q^2$&$q$&.&.&.&.&$q$&$1$&.&$q$&.&.&.&.&.&.&.&.&$q$&.&.&.&.&.  & 1 & 12 \\
2314 &  $q$&.&.&.&.&.&.&.&$1$&.&.&.&.&.&.&.&.&.&.&.&.&.&.&.  & 2 & 3 \\
2341 &  $q$&.&.&.&.&.&.&.&.&$1$&.&.&.&.&.&.&.&.&.&.&.&.&.&.  & 3 & 4 \\
2413 &  $q$&.&.&.&.&.&.&.&.&.&$1$&.&.&.&.&.&.&.&.&.&.&.&.&.  & 4 & 5 \\
2431 &  $q^2$&.&.&.&.&.&.&.&.&$q$&.&$1$&.&.&.&.&.&.&.&.&.&.&.&.  & 4 & 8 \\
3124 &  $q$&.&.&.&.&.&.&.&.&.&.&.&$1$&.&.&.&.&.&.&.&.&.&.&.  & 2 & 3 \\
3142 &  $q$&.&.&.&.&.&.&.&.&.&.&.&.&$1$&.&.&.&.&.&.&.&.&.&.  & 4 & 5 \\
3214 &  $2q^2$&.&.&.&.&.&.&.&$q$&.&.&.&$q$&.&$1$&.&$q$&.&.&.&.&.&.&.  & 1 & 12 \\
3241 &  $q^2$&.&.&.&.&.&.&.&.&$q$&.&.&.&.&.&$1$&.&.&.&.&.&.&.&.  & 4 & 8 \\
3412 &  $q$&.&.&.&.&.&.&.&.&.&.&.&.&.&.&.&$1$&.&.&.&.&.&.&.  & 5 & 6 \\
3421 &  $q^2$&.&.&.&.&.&.&.&.&$q$&.&.&.&.&.&.&$q$&$1$&.&.&.&.&.&.  & 3 & 12 \\
4123 &  $q$&.&.&.&.&.&.&.&.&.&.&.&.&.&.&.&.&.&$1$&.&.&.&.&.  & 3 & 4 \\
4132 &  $q^2$&.&.&.&.&.&.&.&.&.&.&.&.&.&.&.&.&.&$q$&$1$&.&.&.&.  & 4 & 8 \\
4213 &  $q^2$&.&.&.&.&.&.&.&.&.&.&.&.&.&.&.&.&.&$q$&.&$1$&.&.&.  & 4 & 8 \\
4231 &  $q^2$&.&.&.&.&.&.&.&.&$q$&.&.&.&.&.&.&.&.&$q$&.&.&$1$&.&.  & 5 & 12 \\
4312 &  $q^2$&.&.&.&.&.&.&.&.&.&.&.&.&.&.&.&$q$&.&$q$&.&.&.&$1$&.  & 3 & 12 \\
4321 &  $q^3$&.&.&.&.&.&.&.&.&$q^2$&.&.&.&.&.&.&$q^2$&$q$&$q^2$&.&.&$q$&$q$&$1$  & 1 & 24 \\
\hline
\text{Simp.} &1&1&1&2&2&1&1&1&2&3&4&4&2&4&1&4&5&3&3&4&4&5&3&1 \\
\text{Proj.}  &71&2&1&3&3&1&2&1&3&23&4&4&3&4&1&4&16&4&23&4&4&7&4&1\\

\end{array}
\end{displaymath}}
\bigskip

\newpage

\subsection{Decomposition matrices}
\label{appendix.decomposition.matrices}

Since $\biheckemonoid_{w_0}$ is a submonoid of $\biheckemonoid$, any
simple $\biheckemonoid$-module is also a simple
$\biheckemonoid_{w_0}$-module.  The following matrices give the
(generalized) $\biheckemonoid_{w_0}$ character of the simple
$\biheckemonoid$-module. The table reads as follows: for any two
permutations $\sigma, \tau$, the coefficient $m_{\sigma,\tau}$ gives
the Jordan-H\"older multiplicity of the $\biheckemonoid_{w_0}$-module
$S^{w_0}_\tau$ in the $\biheckemonoid$-module $S_\sigma$. In
particular, since the simple $\biheckemonoid_{w_0}$-modules are of
dimension $1$, summing each line one recovers the dimension of the
simple $\biheckemonoid$-modules, as shown.

\bigskip

Decomposition matrix of $\biheckemonoid(\sg[2])$ on $\biheckemonoid_{w_0}(\sg[2])$ (type $A_1$):
\medskip
{\small
\begin{displaymath}
\begin{array}{c|p{0.3cm}@{\,}p{0.3cm}|c}
& \pc{12}&\pc{21}&\text{Simp.}\\
\hline
12&1&.&1 \\
21&.&1&1 \\
\hline
\end{array}
\end{displaymath}}

Decomposition matrix of $\biheckemonoid(\sg[3])$ on $\biheckemonoid_{w_0}(\sg[3])$ (type $A_2$):
{\small
\begin{displaymath}
\begin{array}{c|p{0.3cm}@{\,}p{0.3cm}@{\,}p{0.3cm}@{\,}p{0.3cm}@{\,}p{0.3cm}@{\,}p{0.3cm}@{}|c}
  &\pc{123}&\pc{132}&\pc{213}&\pc{231}&\pc{312}&\pc{321}&\text{Simp.}\\
\hline
123&1&.&.&.&.&. &1 \\
132&.&1&.&.&.&. &1 \\
213&.&.&1&.&.&. &1 \\
231&.&.&1&1&.&. &2 \\
312&.&1&.&.&1&. &2 \\
321&.&.&.&.&.&1 &1 \\
\hline
\end{array}
\end{displaymath}}

Decomposition matrix of $\biheckemonoid(\sg[4])$ on $\biheckemonoid_{w_0}(\sg[4])$ (type $A_3$):
{
\tiny
\begin{displaymath}
\begin{array}{c|p{0.3cm}@{\,}p{0.3cm}@{\,}p{0.3cm}@{\,}p{0.3cm}@{\,}p{0.3cm}@{\,}p{0.3cm}@{\,}p{0.3cm}@{\,}p{0.3cm}@{\,}p{0.3cm}@{\,}p{0.3cm}@{\,}p{0.3cm}@{\,}p{0.3cm}@{\,}p{0.3cm}@{\,}p{0.3cm}@{\,}p{0.3cm}@{\,}p{0.3cm}@{\,}p{0.3cm}@{\,}p{0.3cm}@{\,}p{0.3cm}@{\,}p{0.3cm}@{\,}p{0.3cm}@{\,}p{0.3cm}@{\,}p{0.3cm}@{\,}p{0.3cm}@{}|c}
     & \pc{1234}&\pc{1243}&\pc{1324}&\pc{1342}&\pc{1423}&\pc{1432}&\pc{2134}&\pc{2143}&\pc{2314}&\pc{2341}&\pc{2413}&\pc{2431}&\pc{3124}&\pc{3142}&\pc{3214}&\pc{3241}&\pc{3412}&\pc{3421}&\pc{4123}&\pc{4132}&\pc{4213}&\pc{4231}&\pc{4312}&\pc{4321}&\text{Simp.}\\
\hline
1234 & 1&.&.&.&.&.&.&.&.&.&.&.&.&.&.&.&.&.&.&.&.&.&.&. &1\\
1243 & .&1&.&.&.&.&.&.&.&.&.&.&.&.&.&.&.&.&.&.&.&.&.&. &1\\
1324 & .&.&1&.&.&.&.&.&.&.&.&.&.&.&.&.&.&.&.&.&.&.&.&. &1\\
1342 & .&.&1&1&.&.&.&.&.&.&.&.&.&.&.&.&.&.&.&.&.&.&.&. &2\\
1423 & .&1&.&.&1&.&.&.&.&.&.&.&.&.&.&.&.&.&.&.&.&.&.&. &2\\
1432 & .&.&.&.&.&1&.&.&.&.&.&.&.&.&.&.&.&.&.&.&.&.&.&. &1\\
2134 & .&.&.&.&.&.&1&.&.&.&.&.&.&.&.&.&.&.&.&.&.&.&.&. &1\\
2143 & .&.&.&.&.&.&.&1&.&.&.&.&.&.&.&.&.&.&.&.&.&.&.&. &1\\
2314 & .&.&.&.&.&.&1&.&1&.&.&.&.&.&.&.&.&.&.&.&.&.&.&. &2\\
2341 & .&.&.&.&.&.&1&.&1&1&.&.&.&.&.&.&.&.&.&.&.&.&.&. &3\\
2413 & .&1&.&.&.&.&1&1&.&.&1&.&.&.&.&.&.&.&.&.&.&.&.&. &4\\
2431 & .&1&.&.&.&.&.&1&.&.&1&1&.&.&.&.&.&.&.&.&.&.&.&. &4\\
3124 & .&.&1&.&.&.&.&.&.&.&.&.&1&.&.&.&.&.&.&.&.&.&.&. &2\\
3142 & .&.&1&1&.&.&.&.&.&.&.&.&1&1&.&.&.&.&.&.&.&.&.&. &4\\
3214 & .&.&.&.&.&.&.&.&.&.&.&.&.&.&1&.&.&.&.&.&.&.&.&. &1\\
3241 & .&.&1&.&.&.&.&.&.&.&.&.&1&.&1&1&.&.&.&.&.&.&.&. &4\\
3412 & .&.&1&1&.&.&.&.&.&.&.&.&1&1&.&.&1&.&.&.&.&.&.&. &5\\
3421 & .&.&.&.&.&.&.&.&.&.&.&.&.&.&1&1&.&1&.&.&.&.&.&. &3\\
4123 & .&1&.&.&1&.&.&.&.&.&.&.&.&.&.&.&.&.&1&.&.&.&.&. &3\\
4132 & .&.&1&1&.&1&.&.&.&.&.&.&.&.&.&.&.&.&.&1&.&.&.&. &4\\
4213 & .&.&.&.&.&.&1&1&.&.&1&.&.&.&.&.&.&.&.&.&1&.&.&. &4\\
4231 & .&.&.&.&.&.&.&1&.&.&1&1&.&.&.&.&.&.&.&.&1&1&.&. &5\\
4312 & .&.&.&.&.&1&.&.&.&.&.&.&.&.&.&.&.&.&.&1&.&.&1&. &3\\
4321 & .&.&.&.&.&.&.&.&.&.&.&.&.&.&.&.&.&.&.&.&.&.&.&1 &1\\
\hline
\end{array}
\end{displaymath}}

\newpage

\bibliographystyle{web-alpha}

\bibliography{main}

\end{document}